\documentclass[12pt, reqno]{amsart}
\usepackage{fullpage}
\usepackage{graphicx}
\usepackage{nicematrix}
\usepackage[dvipsnames]{xcolor}
\usepackage{mathtools}
\usepackage{todonotes}
\usepackage{tikz-cd}
\usepackage[all]{xy}
\usepackage{textcomp}
\usepackage[vflt]{floatflt}
\usepackage{enumerate}
\usepackage{rotating}
\usepackage{pdflscape}
\usepackage{comment}
\usepackage{nicematrix}
\usepackage{colortbl}
\usepackage{color}
\usepackage{blkarray}
\usepackage{setspace}
\usepackage{enumitem}
\usepackage{charter}
\usepackage{newtxmath}
\usepackage{etoolbox} % for patching

\makeatletter
\everydisplay{%
  \fontsize{11pt}{13pt}\selectfont
}
\makeatother

\usepackage[hypertexnames=false, colorlinks, citecolor=red, linkcolor=blue, urlcolor=red, bookmarksopen]{hyperref}

\usepackage{enumitem}
\newcounter{dummy}
\makeatletter
\newcommand\myitem[1][]{\item[#1]\refstepcounter{dummy}\def\@currentlabel{#1}}
\makeatother

\usepackage[backend=biber, style=alphabetic, sorting=nyt, nohashothers=true, maxbibnames=99, backref=true]{biblatex}
\usepackage[depth=3]{bookmark}
\let\oldtocsection=\tocsection
\let\oldtocsubsection=\tocsubsection
\let\oldtocsubsubsection=\tocsubsubsection
\renewcommand{\tocsection}[2]{\hspace{0em}\vspace{0.5mm}\oldtocsection{#1}{#2}\vspace{0.5mm}}
\renewcommand{\tocsubsection}[2]{\hspace{2em}\vspace{0.25mm}\oldtocsubsection{#1}{#2}\vspace{0.25mm}}
\renewcommand{\tocsubsubsection}[2]{\hspace{2em}\oldtocsubsubsection{#1}{#2}}

\numberwithin{equation}{section}

\newtheorem{theorem}{Theorem}[section]
\newtheorem{lemma}[theorem]{Lemma}
\newtheorem{prop}[theorem]{Proposition}

\newtheorem{cor}[theorem]{Corollary}
\newtheorem*{thm*}{Theorem}

\newtheorem{corollary}[theorem]{Corollary}

\theoremstyle{definition}

\newtheorem{remark}[theorem]{Remark}
\newtheorem{definition}[theorem]{Definition}

\newcommand{\be}{\begin{equation}}
\newcommand{\ee}{\end{equation}}
\newcommand{\bes}{\begin{equation*}}
\newcommand{\ees}{\end{equation*}}

\newcommand{\cH}{\mathcal{H}}

\newcommand{\cR}{\mathcal{R}}
\newcommand{\cS}{\mathcal{S}}

\newcommand{\cU}{\mathcal{U}}

\newcommand{\lel}{\left\langle}
\newcommand{\rir}{\right\rangle}

\newcommand{\bB}{\mathbb{B}}
\newcommand{\bC}{\mathbb{C}}

\newcommand{\bF}{\mathbb{F}}
\newcommand{\bH}{\mathbb{H}}

\newcommand{\bN}{\mathbb{N}}

\newcommand{\lip}{\langle}
\newcommand{\rip}{\rangle}
\newcommand{\ip}[1]{\lip #1 \rip}
\newcommand{\bip}[1]{\big\lip #1 \big\rip}
\newcommand{\Bip}[1]{\Big\lip #1 \Big\rip}

\newcommand{\supp}{\operatorname{supp}}

\newcommand{\col}{\operatorname{col}}

\newcommand{\ba}{{\mathbf{a}}}

\newcommand{\sB}{\scr{B}}
\newcommand{\sD}{\scr{D}}

\newcommand{\sH}{\scr{H}}

\newcommand{\sS}{\scr{S}}

\newcommand{\foral}{\text{ for all }}
\newcommand{\qand}{\quad\text{and}\quad}

\newcommand{\AND}{\text{ and }}

\def\C{\mathbb{C}}
\def\B{\mathbb{B}}
\def\D{\mathbb{D}}
\def\cJ{\mathcal{J}}
\def\cK{\mathcal{K}}
\def\mr{\mathrm}
\def\scr{\mathscr}
\def\mult{\mathbb{H}^\infty_d}

\def\hardy{\mathbb{H}^2_d}
\def\nbker{\mathrm{Ker} \,}
\def\nbran{\mathrm{Ran} \,}
\def\nbdom{\mathrm{Dom} \,}
\def\bi{\begin{itemize}}
\def\ei{\end{itemize}}

\def\mbf{\mathbf}
\def\ga{\gamma}
\def\Ga{\Gamma}
\def\N{\mathbb{N}}
\def\F{\mathbb{F} _d ^+}
\def\Om{\Omega}
\def\hardy{\mathbb{H} ^2 _d}
\def\ncu{\mathbb{C} ^{(\N \times \N) \cdot d}}
\def\rball{\mathbb{B} ^d _\N}
\def\bdn{\mathbb{B}  ^d _n}
\def\cdn{\mathbb{C} ^{(n\times n)\cdot d}}
\def\cdm{\mathbb{C} ^{(m\times m)\cdot d}}
\def\bpm{\begin{pmatrix}}
\def\epm{\end{pmatrix}}
\newcommand{\bsm}{\left[\begin{smallmatrix}}
\newcommand{\esm}{\end{smallmatrix} \right] }
\newcommand{\ipr}[2]{\ensuremath{\left\langle {#1} , {#2} \right\rangle}}
\newcommand{\BMob}[1]{\ensuremath{B^{\left\langle {#1} \right\rangle}}}
\newcommand{\Mob}[2]{\ensuremath{{#1}^{\left\langle {#2} \right\rangle}}}
\def\fz{\mathfrak{z}}
\def\om{\omega}
\def\vs{\vspace{0.1cm}}
\def\vsm{\vspace{-0.1cm}}
\def\ba{\begin{eqnarray*}}
\def\ea{\end{eqnarray*}}
\def\mrt{\mathrm{t}}
\def\wt{\widetilde}
\def\ov{\overline}
\def\la{\lambda}
\def\nbdim{\mathrm{dim} \,}

\begin{document}

\nobreakdash

\title{A non-commutative de Branges--Rovnyak model for row contractions}

\author{Robert T.W. Martin}
\address{Department of Mathematics, University of Manitoba, Winnipeg, Canada}
\email{Robert.Martin@umanitoba.ca}

\author{Jeet Sampat}
\address{Department of Mathematics, University of Manitoba, Winnipeg, Canada}
\email{Jeet.Sampat@umanitoba.ca}

%\subjclass[2020]{}

\thanks{RTWM is partially supported by NSERC Discovery Grant 2020-05683. JS acknowledges partial funding support from the Pacific Institute for the Mathematical Sciences.}

\begin{abstract}
We extend the de Branges--Rovnyak model for completely non-coisometric (CNC) linear contractions on a Hilbert space to the non-commutative multivariate setting of CNC row contractions. Namely, we show that any CNC contraction from several copies of a Hilbert space into a single copy is unitarily equivalent to the adjoint of the restricted backward right shifts acting on the de Branges--Rovnyak space of a contractive left multiplier between vector-valued ``free Hardy spaces" of square--summable power series in several non-commuting (NC) variables. This contractive, operator--valued left multiplier, the \emph{characteristic function of the CNC row contraction}, is a complete unitary invariant and it is always \emph{column--extreme} as a contractive left multiplier.

Our construction builds a model reproducing kernel Hilbert space of NC functions using a ``non-commutative resolvent" of the row contraction, $T$, which is the inverse of the monic, affine linear pencil of $T$ in a certain NC unit row-ball of the NC universe of all row tuples of square matrices of all finite sizes. 

\vspace{2mm}


\end{abstract}

\maketitle

\section{Introduction}\label{sec:intro}

The de Branges--Rovnyak and Nagy--Foais models are two widely-used and equivalent models for completely non-unitary (CNU) linear contractions on Hilbert spaces \cite{dBmodel,dBss,BallK-model,Nik1986,Sarason-dB,NF}. The de Branges--Rovnyak model for a completely non-coisometric (CNC) contraction is decidedly natural and compelling as it shows that any such CNC contraction, $T$, can be represented as the adjoint of the restriction of the backward shift from a vector--valued Hardy space in the complex unit disk, $\D$, to the de Branges--Rovnyak space of a contractive (operator--valued) analytic function, $b_T$, in $\D$. (The full de Branges--Rovnyak model for an arbitrary CNU contraction is considerably more complicated \cite{BallK-model,dBmodel}.) This contractive $b_T$, the \emph{characteristic function of $T$}, is necessarily an extreme point of the closed, convex set of contractive operator--valued functions. The characteristic function, $b_T$, is unique up to a natural isomorphism (conjugation by constant unitaries) and is a complete unitary invariant for $T$. Moreover, many properties of $T$, including spectral information, are encoded in properties of its characteristic function. 

Recall that the classical Hardy space, $H^2$, of the complex unit disk, can be defined as the Hilbert space of all square--summable Taylor series (at $0$), equipped with the $\ell^2-$inner product of the Taylor coefficients. Such power series have radius of convergence at least $1$ and hence define analytic functions in $\D$. It is clear from this definition that the \emph{shift}, $S=M_z$, is an isometry on $H^2$, and this operator plays an important role in Hardy space theory. Its adjoint, the \emph{backward shift} acts as a difference quotient at $0$, 
$$ (S^*h) (z) = \frac{h(z)-h(0)}{z}; \quad \quad z \neq 0, $$ and $(S^*h) (0) = h'(0)$. One can show that the unital \emph{Hardy algebra}, $H^\infty$, of all uniformly bounded analytic functions in $\D$, equipped with the supremum norm, is the \emph{multiplier algebra} of $H^2$. That is, given any function, $h$ in $\D$, $h \in H^\infty$ if and only if $h\cdot f \in H^2$ for all $f \in H^2$. Moreover for any $h \in H^\infty$ one has that the operator norm of the multiplication operator, $h(S) = M_h$ is equal to the supremum norm of its symbol, $h$, in $\D$. It follows that an analytic function, $b$, is contractive in $\D$ if and only if it defines a contractive left multiplier of $H^2$. Given any such contractive $b \in [H^\infty ]_1$, one can define its \emph{de Branges--Rovnyak space}, $\cH (b)$, as the \emph{operator--range space}, $\scr{R} (\sqrt{I - b(S)b(S)^*})$, of the defect operator, $\sqrt{I-b(S)b(S)^*}$. While this space is equal to $\nbran \sqrt{I-b(S)b(S)^*}$ as a vector space, it is equipped with an inner product that `makes' $\sqrt{I-b(S)b(S)^*}$ a coisometry onto its range, see \cite{FM1,FM2,Sarason-dB}. In particular, every de Branges--Rovnyak space, $\cH (b)$, is contained, contractively, in $H^2$, and is always co-invariant for the shift. 

The main goal of this paper is to extend the de Branges--Rovnyak model for CNC contractions, and in particular the Kre\u{\i}n--Liv\v{s}ic resolvent model construction of \cite{AMR,Krein,GMR,Livsicone,Livsic,Martin-ext,MR19} to the non-commutative (NC) and multivariate setting of CNC \emph{row contractions}, \emph{i.e.} CNC contractions from several copies of a Hilbert space into itself. (An NC Sz.-Nagy--Foais model has been constructed and studied in \cite{Popchar,Popmodel}.) In this NC and multivariable de Branges--Rovnyak model, the Hardy space of square--summable Taylor series in the complex unit disk is replaced by the \emph{free Hardy space}, $\hardy$, or full Fock space over $\C ^d$, which can be defined as the Hilbert space of square--summable ``free" formal power series (FPS) in $d \in \N$ non-commuting variables. This can be viewed as a Hilbert space of non-commutative (NC) functions acting on a certain \emph{NC unit row-ball} centered at the origin, $0=(0, \cdots, 0) \in \C ^{1\times d}$, of the $d-$dimensional, complex \emph{NC universe} of all row $d-$tuples of complex, square matrices, of any finite size $n \in \N$. Namely, the NC universe is
$$ \ncu := \bigsqcup _{n=1} ^\infty \cdn; \quad \quad \cdn := \C ^{n\times n} \otimes \C ^{1\times d}, $$ and the NC unit row-ball is
$$ \rball := \bigsqcup _{n=1} ^\infty \B ^d _n; \quad \quad \B ^d _n := \left\{ \left. X = (X_1, \cdots, X_d ) \in \cdn \right| \ \| X \| _{\scr{B} (\C ^n \otimes \C ^d, \C ^n)} < 1 \right\}. $$

By \cite[Theorem 1.1]{Pop-freeholo}, elements of the free Hardy space have ``radius of convergence" at least $1$ and hence define holomorphic and analytic NC functions in the NC unit row-ball, in the sense of modern NC function theory, which is a canonical ``quantization" or extension of complex analysis to several NC variables \cite{AgMcY,KVV,Taylor2,Taylor,Voic,Voic2}. 

The theory of the free Hardy space, $\hardy$, is an extremely close parallel to classical Hardy space theory, see \emph{e.g.} \cite{DP-inv,JM-subFock,JM-ncFatou,JMS-ncBSO,Pop-factor,Pop-multi,Pop-entropy,Pop-freeholo}. In particular, left multiplication by any of the $d$ independent NC variables, $\fz _1, \cdots, \fz _d$ define a $d-$tuple of isometries, $L_k := M^L _{\fz _k}$ on $\hardy$, the \emph{left free shifts}, and these play the role of the \emph{shift} operator, $S=M_z$, in this NC Hardy space theory. Similarly, right multiplications by the $\fz _k$, $R_k := M^R _{\fz _k}$ also define a $d-$tuple of isometries, the \emph{right free shifts}, and free Hardy space theory is left--right symmetric. Moreover, a (holomorphic) NC function, $h$, is uniformly bounded in the unit row-ball if and only if it defines a bounded left multiplication operator on $\hardy$, in which case the operator-norm of $h(L) := M^L _h$ is the same as the supremum norm of its symbol, $h$, in the unit row-ball \cite[Theorem 3.1]{SSS,Pop-freeholo}. That is, as before, the \emph{NC Hardy algebra}, $\mult$, of uniformly bounded NC functions in the unit row-ball, is equal to the left multiplier algebra of $\hardy$, $\mr{Mult} ^L \, \hardy = \mult$, with equality of norms, $\| h(L) \| _{\scr{B} (\hardy)} = \| h \| _\infty$. In particular, a left multiplier, $B$, between vector-valued free Hardy spaces is contractive in the NC unit row-ball if and only if it is contractive as a left multiplication operator. One can define the (left) NC de Branges--Rovnyak space, $\scr{H} (B)$, of such a contractive $B$ as the operator--range space of the defect operator, $\sqrt{I - B(L) B(L)^*}$, as before. Here, $B(L) = M^L _B$ is the contractive linear operator of left multiplication by $B$, and $\scr{H} (B)$ is always \emph{right shift co-invariant}, in this NC setting \cite{BBF-ncSchur}.

Our main result is an exact non-commutative analogue of the de Branges--Rovnyak model for CNC row contractions. Given a row contraction, $T=(T_1,\cdots, T_d): \cH \otimes \C^d \rightarrow \cH$, define the \emph{defect operators},
\begin{equation*}
    D_T := \sqrt{I_\cH \otimes I_d - T^* T}, \quad \mbox{and} \quad D_{T^*} := \sqrt{I_\cH  - TT^*},
\end{equation*}
and the \emph{defect spaces}, $\scr{D} _T := \nbran D_T ^{-\| \cdot \|}$, $\scr{D} _{T^*} := \nbran D_{T^*} ^{-\| \cdot \|}$. In the theorem statement below, the term \emph{column--extreme (CE)} is a concept that generalizes (and reduces to) the notion of an extreme point of the closed convex set of contractive univariate analytic functions in the complex unit disk. 

\begin{thm*}[\ref{thm:main.model}]
A row contraction, $T = (T_1, \cdots, T_d) : \cH \otimes \C ^d \rightarrow \cH$, is CNC, if and only if there exists a (purely) contractive and column--extreme (operator--valued) NC function, $B_T \in \mult \otimes \scr{B} (\scr{D} _{T} , \scr{D} _{T^*})$, so that $T$ is unitarily equivalent to $X$ acting on the (left) NC de Branges--Rovnyak space, $\scr{H} (B_T)$, where $X^*  = R^* \otimes I_{\scr{D} _{T^*}}| _{\scr{H} (B_T)} \in \scr{B} (\scr{H} (B), \scr{H} (B) \otimes \C ^d)$. 
\end{thm*}
We say that two operator--valued left multipliers, $B \in \mult \otimes \scr{B} (\cH, \cJ)$ and $B' \in \mult \otimes \scr{B} (\cH', \cJ')$ \emph{coincide unitarily} and write $B \simeq B'$ if there exists surjective isometries, $U: \cH \twoheadrightarrow \cH'$ and $V: \cJ \twoheadrightarrow \cJ'$ so that $B' = V^* B U$. That is, for any $Z= (Z_1, \cdots, Z_d)$, $Z_j \in \C ^{n\times n}$, in the unit row-ball, $B' (Z) = [V^* \otimes I_n] B(Z) [U \otimes I_n]$.

The \emph{support} of $B \in \sS_d (\cH,\cJ)$ is the subspace,
$$ \mr{supp} (B) := \bigvee _{\substack{Z \in \bdn; \\ y,v \in \C^n}} \nbran y^* \otimes I_\cH B(Z) ^* v \otimes I_\cJ \subseteq \cH. $$ By construction, $\cH = \mr{supp} (B) \oplus \cH _0$, where $B(Z) v \otimes h_0 =0$ for any $h \in \cH _0$, and $\cH _0$ is the largest subspace of $\cH$ with this property. Similarly, if we define the \emph{support of $B^*$} as 
\be \mr{supp} (B^*) := \bigvee _{\substack{Z \in \bdn; \\ y,v \in \C^n}} \nbran y^* B(Z)  v \subseteq \cJ \label{support}, \ee Then $\cH = \mr{supp} (B) \oplus \cH _0$, $\cJ = \mr{supp} (B^*) \otimes \cJ _0$, and $B$ has the direct sum decomposition 
$$ B = B_0 \oplus \mbf{0}_{\cH _0, \cJ _0}, $$ where $B_0 \in \sS _d (\mr{supp} (B), \mr{supp} (B^*)) := [ \mult \otimes \scr{B} (\mr{supp} (B), \mr{supp} (B^*))]_1$ and $\mbf 0_{\cH _0, \cJ _0} : \cH _0 \rightarrow \cJ _0$ is identically zero. Given $B \in \sS _d (\cH, \cJ)$ and $B' \in \sS _d (\cH', \cJ')$, we say that $B$ and $B'$ \emph{coincide weakly} and write $B \sim B'$, if $B_0 = B | _{\mr{supp}(B)}$ coincides unitarily with $B'_0 | _{\mr{supp} (B')}$. 
\begin{thm*}[\ref{cor:char.map.is.uni.inv}]
The column--extreme \emph{characteristic function}, $B_T$, of a CNC row contraction, $T$, is unique up to weak coincidence, and it is a complete unitary invariant: CNC row contractions, $T_1$ and $T_2$ are unitarily equivalent if and only if $B_{T_1} \sim B_{T_2}$. 
\end{thm*}

The above two theorems can be established, fairly readily, from the NC de Branges--Rovnyak realization and transfer--function theory of \cite{BBF-ncSchur} (and we will sketch a proof in Appendix \ref{appx:Model.via.TFR}). However, our approach is significantly different, provides additional information, and we believe it is of independent interest. Namely, given any CNC row partial isometry, $V: \cH \otimes \C ^d \rightarrow \cH$, we explicitly build an NC reproducing kernel Hilbert space (NC-RKHS) in the sense of Ball--Marx--Vinnikov \cite{BMV-NCrkhs} out of an ``NC resolvent function" for $V$,
$$ \left( I_\cH \otimes I_n - \sum _{j=1} ^d V_j \otimes Z_j ^* \right) ^{-1} =: (I - VZ^*) ^{-1}; \quad \quad Z \in \B ^d _n, $$ for $Z$ in the NC unit row-ball, $\rball$, and we prove that there is a unitary left multiplier from this NC-RKHS onto a (vector--valued) left de Branges--Rovnyak space, $\scr{H} (B)$. Note that the inverse of this ``NC resolvent", 
$$ I_\cH \otimes I_n - \sum _{j=1} ^d V_j \otimes Z_j ^* =: L_V (Z^*); \quad \quad  Z^* = \bsm Z_1 ^* \vsm \\  \vdots \vs \\ Z_d ^* \esm $$ is the \emph{linear pencil} of $V=(V_1, \cdots, V_d)$, evaluated at $Z^*$. Linear pencils of operator tuples appear throughout NC function theory, they are central objects in the realization theory of NC functions, in free real algebraic geometry, and in matrix convexity theory \cite{AMS-opreal,Ball-sys,CR-realize,CR-freefield,Fliess1,Fliess2,HelMc-annals,HMS-realize,HelMcV-convex,Hup-realize,Hup-realize2,Kalman,freefield,Schut}. Any row contraction, $T: \cH \otimes \C ^d \rightarrow \cH$ has a unique \emph{isometric--pure} decomposition, $T = V+C$, where $V: \cH \otimes \C ^d \rightarrow \cH$ is a row partial isometry and $C$ is \emph{purely contractive} in the sense that $\| C \mbf{x} \| < \| \mbf{x} \|$ for any $\mbf{x} \in \cH \otimes \C ^d$, see Lemma \ref{isopure}. Moreover, $V$ is CNC if and only if $T$ is CNC. Using this decomposition, we extend our model from CNC row partial isometries to arbitrary CNC row contractions. 
\begin{thm*}[\ref{thm:main.model.row.partial.iso}) and (\ref{thm:description.char.map.of.CNC.row.contraction}]
Let $T$ be a CNC row contraction on $\cH$ with isometric--pure decomposition $T=V+C$. If $B_T \in \mult \otimes \scr{B} (\sD_T, \sD_{T^*})$ is the characteristic function of $T$, then $C=0$ if and only if $B_T (0) =0$, and $B_T$ is an operator--M\"obius transformation of  the characteristic function, $B_V$, of $V$ by a fixed pure contraction, $\delta _T \in \scr{B} (\cH, \cJ)$, that depends only on $C$, 
$$ B_T(Z) = \big[ D_{\delta_T^*} \otimes I_n \big] \Big( B_V(Z) + \big[ \delta_T \otimes I_n \big] \Big) \Big( I_{\sD_T \otimes \bC^n} + \big[ {\delta_T}^* \otimes I_n \big] B_V(Z) \Big)^{-1} \big[ D_{\delta_T} \otimes I_n \big] $$
for any $Z \in \B ^d _n, \ n \in \N$. Moreover, $B_V$ is unique up to weak coincidence, and is given by
    \begin{eqnarray*}
      B_V(Z) & = &  D_V(Z)^{-1} N_V(Z); \\
  \mbox{where,} \quad      D_V (Z) &= & [I - VV^*] [I - ZV^*]^{-1} [I - VV^*] \otimes I_n \quad \mbox{and} \nonumber \\
        N_V(Z) &= & [I - VV^*] [I - ZV^*]^{-1} [I_\cH \otimes Z] \, [I - V^*V]  \otimes I_n. \nonumber \end{eqnarray*}
\end{thm*}
Our construction builds on ideas originating in the work of M.S. Liv\v{s}ic and M.G. Kre\u{\i}n on the representation theory of partial isometries and generally unbounded symmetric operators \cite{Krein,Krein1949,Livsicone,Livsic}, and this paper can be viewed as sequel to an earlier paper by the first author and A. Ramanantoanina, in which a de Branges--Rovnyak model for a large class of CNC row contractions was constructed using the commutative de Branges--Rovnyak spaces for contractive multipliers between vector--valued Drury--Arveson spaces \cite{MR19}. This commutative model can be recovered from the NC de Branges--Rovnyak model of this paper by ``restricting to level one" of the NC unit row-ball, and we will describe the relationship between these two models in detail in Section \ref{sec:NCvscomm}. Lastly, in Appendix \ref{appx:compare.with.NFP.model}, we compare our characteristic function with Popescu's \cite{Popchar} extension of the classical Sz.-Nagy--Foias characteristic function for a CNC contraction to the case of tuples of CNC row contractions and show that the two coincide weakly.

\section{Background}

All inner products and sesquilinear forms in this paper are assumed to be conjugate linear in their first argument and linear in their second argument.

\subsection{NC function theory}

As described in the introduction, we will be considering Hilbert spaces of \emph{non-commutative (NC) holomorphic functions} in a certain \emph{NC unit row ball}. Recall that the $d-$dimensional complex \emph{NC universe}, $\ncu$, is the set of all row $d-$tuples of complex $n\times n$ matrices, of any finite size, $n \in \N$,
$$ \ncu = \bigsqcup \cdn; \quad \quad \cdn = \C ^{n\times n} \otimes \C ^{1\times d}, $$
and the \emph{NC unit row-ball} is the subset of the NC universe,
$$ \B ^d _\N =\bigsqcup \B ^d _n; \quad \quad \B ^d _n = \left\{ \left. X = (X_1, \cdots, X_d) \in \cdn \ \right| \ \| X \| _{\scr{B} (\C ^n \otimes \C ^d, \C^n)} < 1 \right\}, $$ consisting of all finite--dimensional strict row contractions. Given $X = (X_1, \cdots, X_d) \in \cdn$, we will write
$$ \| X \| _{\mr{row}} := \| X \| _{\scr{B} (\C ^n \otimes \C ^d, \C^n)}, $$ and call this quantity the \emph{row-norm} of $X$. Given $X \in \cdn$, we will write $X^* := \bsm X_1 ^* \vsm \\ \vdots \vs \\ X_d ^* \esm : \C ^n \rightarrow \C ^n \otimes \C ^d$ for the Hilbert space adjoint of $X: \C ^n \otimes \C ^d \rightarrow \C ^n$, and we will call
$$ \| X^* \| _{\mr{col}} := \| X ^* \| _{\scr{B} (\C ^n, \C^n \otimes \C ^d)}, $$ the \emph{column-norm} of $X^*$. 

A subset, $\Om \subseteq \ncu$ is called an \emph{NC set}, if it is closed under direct sums. Namely, if $\Om _n := \Om \cap \cdn$, we write $\Om := \bigsqcup \Om _n$, and $\Om$ is an NC set if $X \in \Om _m$, $Y\in \Om _n$ imply that the component-wise direct sum, $X\oplus Y = (X_1 \oplus Y_1, \cdots, X_d \oplus Y_d ) \in \Om _{m+n}$. Although we will not make significant use of it, the \emph{uniform topology} is defined as follows. Given any $Y \in \cdn$ and $r>0$, consider the  \emph{NC row-ball},
$$ r\cdot \B ^d _{\N n} (Y) := \bigsqcup _{m=1} ^\infty \B ^d _{mn} (Y); \quad \quad \B ^d _{mn} (Y) := \left\{ \left. X \in \C ^{(mn \times mn) \cdot d} \ \right| \ \| X - I_m \otimes Y \| _{\mr{row}} < r \right\}. $$ The \emph{uniform topology} on $\ncu$ is the topology on $\ncu$ generated by the sub-base of open NC sets $r \cdot \B ^d _{\N n } (Y)$, for any $Y \in \cdn$, $n \in \N$ and $r>0$. In particular, the \emph{NC unit row-ball}, $\B ^d _\N$ is  $\B ^d _{\N \cdot 1} (0)$, where $0 := (0,\cdots, 0) \in \C ^{1\times d}$ is the \emph{origin} of the NC universe. 

Given an NC set, $\Om \subseteq \ncu$, a function, $f: \Om \rightarrow \C ^{\N \times \N}$, \emph{i.e.} with range in the $1-$dimensional NC universe, is called a \emph{non-commutative (NC) function}, if it obeys:
\bi
    \item[(i)] $f : \Om _n \rightarrow \C ^{n\times n}$, \emph{i.e.} $f$ \emph{respects the grading},
    \item[(ii)] Given $X \in \Om _m$, $Y\in \Om _m$, 
    $$ f \bpm X & 0 \\ 0 & Y \epm = \bpm f(X) & 0 \\ 0 & f (Y) \epm, $$ \emph{i.e.} $f$ \emph{respects direct sums}. 
    \item[(iii)] Given any $S \in \mr{GL} _n$, if $X \in \Om _n$ and $S ^{-1} X S := ( S^{-1} X_1 S, \cdots, S^{-1} X_d S ) \in \Om _n$, then 
    $$ f (S^{-1} X S ) = S ^{-1} f(X) S, $$ \emph{i.e.} $f$ \emph{respects joint similarities}. 
\ei

We will also need to consider $\cJ-$valued and $\scr{B} (\cH, \cJ)-$valued NC functions, where $\cJ, \cH$ are complex, finite or separable Hilbert spaces. To be precise, given an NC set, $\Om \subseteq \ncu$,  we will say that a function $f : \Om \rightarrow \C ^{\N \times \N} \otimes \cJ$ is an NC $\cJ-$valued function if for any fixed $g \in \cJ$ and $Z \in \Om _n$, $g^* \otimes I_n f(Z) =: f_g (Z)$ is an NC function obeying the axioms (i--iii) above. Similarly, $A : \Om \rightarrow \C ^{\N \times \N} \otimes \scr{B} (\cH, \cJ)$ is an NC $\scr{B} (\cH, \cJ)-$valued function if given any $h \in \cH$ and $g \in \cJ$, 
\begin{equation}\label{eqn:def.op-val.NC.map}
    Z \in \Om _n \ \mapsto \  A _{h,g} (Z) := h^* \otimes I_n A(Z) g \otimes I_n,
\end{equation}
is an NC function. That is, for example, an NC $\C ^{m\times n}-$valued function is an $m\times n$ matrix of NC functions.  

\begin{remark} \label{matrixeval}
It is easily seen that an NC $\cJ-$valued or $\scr{B} (\cH, \cJ)-$valued function, as defined above, need not respect direct sums, \emph{i.e.} axiom (ii). For example, if $A$ is an NC $\C ^{2\times 2}-$valued function, then 
$$ A (Z \oplus W) = \bsm  a_{11} (Z) & 0 & a_{12} (Z) & 0 \\ 0 & a_{11} (W) & 0 & a_{12} (W) \\ a_{21} (Z) & 0 & a_{22} (Z) & 0 \\ 0 & a_{21} (W) & 0 & a_{22} (W) \esm \neq \bsm a_{11} (Z) & a_{12} (Z) & 0 & 0 \\ a_{21} (Z) & a_{22} (Z) & 0 & 0 \\ 0 & 0 & a_{11} (W) & a_{12} (W) \\ 0 & 0 & a_{21} (W) & a_{22} (W) \esm = A(Z) \oplus A(W). $$
One can conjugate $A (Z \oplus W)$ by certain constant permutation matrices, so that if one redefines evaluation of $A$ at $Z \oplus W$ by conjugating by these ``shuffle matrices", then $A(Z\oplus W) = A(Z) \oplus A(W)$ \cite{PV1}.

To be precise, if $H$ is an NC $\scr{B} (\cH, \cJ)-$valued function, fix orthonormal bases $(h_i)$ and $(g_j)$ of $\cH$ and $\cJ$, respectively. When we evaluate $H$ at a point $Z \in \B ^d _n$, we are evaluating each ``matrix--entry", of $H$ at $Z$, so that $\mr{ev}_Z (H) \in \C ^{n \times n} \otimes \scr{B} (\cH, \cJ)$. Let $U_\cH : \C ^{n} \otimes \cH \twoheadrightarrow \cH \otimes \C ^{n}$ and $U_\cJ$ be the ``tensor swap" unitaries defined by
$U_\cH v \otimes h = h \otimes v$, for $v \in \C ^n$ and $h \in \cH$. If we redefine evaluation at $Z \in \B ^d _n$, for every $n \in \N$, by 
\be H \ \mapsto \ U_\cH \mr{ev} _Z (H) U_\cJ ^* =: \wt{H} (Z), \label{modeval} \ee then, with this definition of evaluation, $\wt{H}$ respects direct sums, see \cite[Section 4]{PV1}, and hence is a bona fide $\scr{B} (\cH, \cJ)-$valued NC function that respects all three axioms (i--iii).  For the remainder of the paper, if $H$ is an NC $\scr{B} (\cH, \cJ)-$valued NC function on an NC set, $\Om$, for any finite or separable Hilbert spaces $\cH,\cJ$, we define the \emph{evaluation of $H$ at $Z$} as $H(Z) := U_\cH \mr{ev} _Z (H) U_\cJ ^* $ as above, so that $H(Z)$ respects direct sums. 
\end{remark}

Any ``free polynomial" in $d$ non-commuting variables, \emph{i.e.} any element, $p$, of the free associative algebra over $\C$, $p \in \C \langle \fz _1, \cdots, \fz _d \rangle$, is an NC function on the entire NC universe. Remarkably, for any NC function, $f$, defined in a uniformly open NC set, $\Om \subseteq \ncu$, the following are equivalent: (a) $f$ is locally bounded in the uniform topology, (b) $f$ is continuous in the uniform topology, (c) $f$ is holomorphic in the sense that given any $Y \in \Om _m$ and any ``direction" $H \in \cdm$, the directional or G\^ateaux derivative of $f$ at $Y$ in the direction $H$ exists:
$$ \lim _{t \rightarrow 0} \frac{f(Y +tH) - f(Y)}{t} =: \partial _H f (Y) \quad \exists, $$ (moreover, $f$ has a total or Fr\'echet derivative at $Y$), and (d) $f$ is \emph{uniformly analytic} in a uniformly open neighbourhood at $Y$, which means that it has an `absolutely' convergent Taylor-type power series expansion about $Y$ that converges in an NC row-ball, $r \cdot \B ^d _m\N (Y)$, of positive radius, $r>0$ \cite[Corollary 7.26 and Corollary 7.28]{KVV}. 

We will be particularly interested in NC functions that are bounded (hence uniformly analytic) in the NC unit row ball, $\B ^d _\N$. First, consider a \emph{free formal power series (FPS)}, $f \in \C \langle \!  \langle \fz  \rangle \!  \rangle = \C \langle \!  \langle \fz _1, \cdots, \fz _d \rangle  \! \rangle$. That is, $f$ is a formal power series in the $d$ non-commuting variables, $\fz = (\fz _1, \cdots, \fz _d)$ with complex coefficients,
$$ f (\fz) = \sum _{\om \in \F} \hat{f} _\om \fz ^\om; \quad \quad \hat{f} _\om \in \C. $$ In the above,  $\F$ denotes the free monoid, the set of all words, $\om =i_1 \cdots i_k$, comprised of letters, $i_j$, chosen from the alphabet $\{ 1, \cdots, d\}$. If $\om = i_1\cdots i_k$, we write $\fz ^\om := \fz _{i_1} \cdots \fz _{i_k}$, and if $\om = \emptyset$, we write, $\fz ^\emptyset =:1$, the unit of the ring of free FPS. Here, $\F$ is a monoid with product given by concatenation of words and unit, $\emptyset$, the empty word. Given $\om =i_1 \cdots i_k \in \F$ we will write $|\om | =k$ for the \emph{length} of $\om$, and we define the \emph{transpose} of $\om$ as $\om ^\mrt := i_k \cdots i_1$, an involution on $\F$.

The \emph{free Hardy space}, $\hardy$, is then the Hilbert space of all free FPS with square--summable complex coefficients, 
$$ \hardy = \left\{ \left. f(\fz ) = \sum _{\om \in \F} \hat{f} _\om \fz ^\om \right| \ \sum _{\om \in \F} |\hat{f} _\om | ^2 < +\infty \right\}, $$ equipped with the $\ell ^2-$inner product of the coefficients. (Although we will not use this point of view, $\hardy$ can be readily identified with $\ell ^2 (\F )$ and with $\scr{F} (\C ^d)$, the \emph{full Fock space over $\C ^d$}.) By \cite[Theorem 1.1]{Pop-freeholo}, any $f \in \hardy$ has ``radius of convergence" at least one, and hence defines a uniformly analytic NC function in the unit row-ball, $\B ^d _\N$. It is clear from the definition of $\hardy$ that left multiplication by any of the $d$ NC variables, $L_j := M^L _{\fz _j}$ defines isometries on $\hardy$ with pairwise orthogonal ranges so that $L := (L_1, \cdots, L_d) : \hardy \otimes \C ^d \rightarrow \hardy$ is an isometry from $d$ copies of $\hardy$ into one copy. (Such an isometry is called a \emph{row isometry}.) The free monomials, $\{ \fz ^\om \} _{\om \in \F}$, form a standard orthonormal basis of $\hardy$, and so the \emph{transpose unitary}, $U_\mrt : \hardy \rightarrow \hardy$ defined by $U_\mrt \fz ^\om = \fz ^{\om ^\mrt}$ is a self-adjoint unitary involution on $\hardy$. One can check that the images of the left free shifts under $U_\mrt$ are $R_k := U_\mrt L_k U_\mrt = M^R _{\fz _k}$, the \emph{right free shifts} which are the isometries of right multiplication by the independent NC variables.  

We will also need the \emph{NC Hardy algebra},
$$ \mult := \left\{ h \in \scr{O} (\B ^d _\N ) \left| \ \sup _{Z \in \B ^d _\N} \| h (Z) \| < + \infty \right. \right\}, $$ of all uniformly bounded (hence analytic) NC functions in the NC unit row-ball. This is a unital Banach algebra, when equipped with the supremum norm on $\B ^d _\N$, and it is an operator algebra. Namely, a uniformly analytic NC function in the NC unit row-ball, $h \in \scr{O} _d (\B ^d _\N)$, is uniformly bounded in $\B ^d _\N$, \emph{i.e.} $h \in \mult$, if and only if $h$ defines a \emph{left multiplier} on $\hardy$. That is, $h \in \mult$ if and only if $h\cdot f \in \hardy$ for all $f \in \hardy$ \cite[Theorem 3.1]{SSS} \cite[Theorem 3.1]{Pop-freeholo}, and the operator norm of the linear operator of left multiplication by $h$, $h(L) := M^L _h : \hardy \rightarrow \hardy$ coincides with the supremum norm of its symbol, $\| h (L) \| = \| h \| _\infty$, over the NC unit row-ball. The Hardy algebra, $\mult$, can also be identified with the weak operator topology (WOT)-closure of $\mr{Alg}\{I,L_1, \cdots, L_d \}$, the unital operator algebra generated by the left free shifts.

\subsection{The operator--valued NC Schur classes} \label{ss:NCSchur}

The set of all $\scr{B} (\cH, \cJ)-$valued NC functions which define bounded left multipliers from $\hardy \otimes \cH$ into $\hardy \otimes \cJ$ will be denoted by $\mult \otimes \scr{B} (\cH, \cJ)$. This is an operator space that is closed in the weak operator topology and its closed unit ball, the \emph{NC operator--valued Schur class}, will be denoted by $\sS _d (\cH, \cJ)$. Given a Banach space, $X$, its open and closed unit balls will be denoted by $(X) _1$ and $[X]_1$, respectively.

As before $B \in [\mult \otimes \scr{B} (\cH , \cJ) ] _1 = \sS _d (\cH, \cJ)$ defines a contractive (operator--valued) left multiplier if and only if its supremum--norm on the NC unit row-ball is at most $1$. In general, we say that a linear contraction, $T \in \scr{B} (\cH, \cJ)$ is a \emph{strict contraction}, if $\| T \| <1$ and a \emph{pure contraction} if $\| T h \| <1$ for all $h \in \cH$. Similarly, we will say that an NC Schur class $B \in \sS _d (\cH, \cJ)$ is \emph{strictly} or \emph{purely contractive}, if each $B (Z)$, $Z \in \rball$, is a strict or pure contraction, respectively. 

\begin{prop} \label{pureuni}
Given any $B \in \sS _d (\cH, \cJ)$, there exists a unique orthogonal decomposition $\cH = \cH _0 \oplus \cH'$ and $\cJ = \cJ _0 \oplus \cJ '$ so that $B = B_0 \oplus B'$ where $B_0 \in \sS _d (\cH _0, \cJ _0)$ is purely contractive and $B' \in \sS _d (\cH' , \cJ')$ is a unitary constant. That is, for any $Z \in \B ^d _n$, $B' (Z) = B' (0_n)$ where $B' (0_n) h' = B(0_n) h'$ and $\| B (0_n ) h' \| = \| h' \|$ for any $h' \in \cH ' \otimes \C ^n$. 
\end{prop}
This proposition and its proof are based on \cite[Proposition 2.1, Chapter V]{NF}. For the remainder of this paper, we will only consider purely contractive $B \in \sS _d (\cH, \cJ)$, \emph{i.e.} we assume our NC Schur class functions have no constant unitary part. 
\begin{proof}
For any $n \in \N$, define $\cH' _n \subseteq \cH \otimes \C ^n =: \cH _n$ and $\cJ' _n \subseteq \cJ \otimes \C ^n =: \cJ _n$ by 
$$ \cH' _n := \{ h' \in \cH _n | \ h' = B(0_n) ^* B(0_n) h'  \} $$ and
$$ \cJ' _n := \{ g' \in \cJ _n | \ g' = B(0_n)  B(0_n) ^* g' \}, $$ 
where $0_n := (0_n, \cdots, 0_n) \in \B ^d _n$. Since $B$ is NC, it is clear that $\cJ' _n = \cJ' \otimes \C ^n$ and $\cH' _n = \cH' \otimes \C ^n$ where $\cJ' = \cJ' _1$, $\cH' = \cH' _1$. It is also clear, from this definition that $B(0_n) : \cH ' _n \twoheadrightarrow \cJ' _n$ is a surjective isometry. 

Given any $Z \in \B ^d _n$ and $h' \in \cH ' _n$, consider $b_{z,h'} (\la) := B (\la Z) h' \in H^2 (\D) \otimes \cH' _n$. Then calculate,
\ba \ipr{B(0_n h'}{B(\la Z)h'}_{H^2 \otimes \cH' _n} & = & \int _{\partial \D} \ipr{B(0_n) h'}{B(\zeta Z) h'} _{\cH' _n} m (d\zeta) \\
& = & \| B(0_n) h' \| ^2  = \| h' \| ^2 = \| h' \| \| B(0_n) h' \| \\
& \geq & \| b_{Z,h'} \| _{H^2 \otimes \cH' _n} \| B(0_n) h' \| _{H^2 \otimes \cH' _n}. \ea 
It follows, from equality in the Cauchy--Schwarz inequality, that 
$$ b_{Z,h'} (\la) = \zeta B(0_n ) h' = \zeta b_{Z,h'} (0), $$ for any $\la \in \D$, and hence 
$\zeta =1$ and $B(Z) h' = B(0_n ) h'$. Since $Z \in \B ^d _n$ and $n \in \N$ were arbitrary, we have $B(Z) h' = B(0_n) h'$ for every $n\in \N$ and $Z \in \B ^d _n$. 

Repeating the same argument with $b ^\dag _{Z,g'} (\la ) := B (\ov{\la} Z) ^* g'$, for $Z \in \B ^d _n$ and $g' \in \cJ ' _n$, shows that $B(Z)^* g' = B(0_n ) ^* g'$ for every $g' \in \cJ ' _n$. Now set $\cH _n ^{(0)} := \cH _n \ominus \cH _n '$, $\cJ _n ^{(0)} := \cJ _n \ominus \cJ ' _n$. If $h_0 \in \cH _n ^{(0)}$, $g' \in \cJ _n '$, then
$$ \ipr{g'}{B(Z) h_0} = \ipr{B(Z) ^* g'}{h_0} = \ipr{B(0_n)^* g'}{h_0} =0, $$ so that $B(Z) | _{\cH _n ^{(0)}} : \cH _n ^{(0)} \rightarrow \cJ ^{(0)} _n$. We claim that $B_0 (Z) := B(Z) |_{\cH _n ^{(0)}}$ is purely contractive. Indeed, if there exists $h_0 \in \cH ^{(0)} _n$ so that $\| B_0 (0_n) h_0 \| = \| h_0 \|$, then
\ba \ipr{h_0}{I- B(0_n) ^* B(0_n) h_0} & = & \| h_0 \| ^2 - \| B(0_n) h_0 \| ^2 \\
& = & \| h_0 \| ^2 - \| B_0 (0_n) h_0 \| ^2 =0. \ea Hence $h_0 = B(0_n) ^* B(0_n) h_0$, so that $h_0 \in \cH ' _n \cap \cH _n ^{(0)} = \{ 0 \}$. The decomposition follows and uniqueness is readily verified.
\end{proof}

\begin{lemma} \label{strict}
An element, $B \in \sS _d (\cH , \cJ)$ is strictly contractive if and only if $B(0) \in \sS _d (\cH, \cJ)$ is a strict contraction. 
\end{lemma}
\begin{proof}
If $B(0)$ is a strict contraction, then by \cite[Lemma 5.2.ii]{MR19}, the operator--M\"obius transformation, 
$$ B_0 (L) := \big(I_{\hardy} \otimes D _{B(0)^*}\big) \big(I - B(L) I_{\hardy} \otimes B(0) ^* \big) ^{-1} \big(B (L) - I_{\hardy} \otimes B(0)\big) I_{\hardy} \otimes D_{B(0)} ^{-1}, $$ belongs to the NC operator--valued Schur class, $B_0 \in \sS _d (\cH, \cJ)$. Moreover, by the proof of \cite[Lemma 5.2.ii]{MR19}, since $D_{B(0)} = \sqrt{I_\cH - B(0) ^* B(0)}$ is invertible, for any $h \in \C^n \otimes \cH$ and $Z \in \B ^d _n$, setting $0_n = (0_n, \cdots, 0_n) \in \B ^d _n$, 
\begin{align*}    
& \| h \| ^2 - \| B_0 (Z) h \| ^2 \\
&= \ipr{h}{D_{B(0_n)} (I_{\C^n \otimes \cH} - B(L) ^* B(0_n) I  ) ^{-1} (I_{\C^n \otimes \cH} - B(Z) ^* B(Z) ) ( I_{\C ^n \otimes \cH} - B(0_n) ^* B(Z)) ^{-1} D _{B(0_n)} h }.
\end{align*}
It follows that $B$ will be strictly contractive in $\rball$ if and only if $B_0$ is, where $B_0 (0) =0$. For any unit vectors, $h \in \cH$ and $g \in \cJ$, the scalar NC function, $b _{g,h} (Z) := g^* B_0 (Z) h$, is contractive on $\rball$ and obeys $b _{g,h} (0) =0$. Hence, by the NC Schwarz Lemma of \cite[Corollary 2.5]{Pop-freeholo}, $\| b_{g,h} (Z) \| \leq \| Z \| _{\mr{row}}$. Since $h,g$ are arbitrary unit vectors, it follows that $\| B_0 (Z) \| \leq \| Z \| _{\mr{row}} <1$, so that $B_0$ is strictly contractive in $\rball$ and hence so is $B$. 
\end{proof}

\begin{cor}
If $B \in \sS _d (\cH, \cJ)$ is purely contractive and $\nbdim \cH < +\infty$ or $\nbdim \cJ < +\infty$, then $B$ is strictly contractive. 
\end{cor}

\begin{proof}
If $B \in \sS _d (\cH, \cJ)$ is purely contractive then for every $h \in \cH$ and $g \in \cJ$, $\| B(0) h \| < \| h \|$ and $\| B(0) ^* g \| < \| g \|$. If $\nbdim \cH < +\infty$, then it is clear that $\| B(0) \| < 1$. Similarly if $\nbdim \cJ < +\infty$ then $\| B(0) \| = \| B(0) ^* \| <1$. The claim now follows from the previous lemma. 
\end{proof}

As defined in \cite{JMfree,JM-freeCE}, $B \in \sS_d (\cH,\cJ)$ is said to be \emph{column--extreme} (CE) if, given any $A \in \sS_d (\cH, \cJ')$, 
$$ \bpm B \\ A \epm \in \sS _d (\cH, \cJ \oplus \cJ') \quad \Rightarrow \quad A =0. $$ As proven in \cite[Theorem 6.4]{JM-freeCE}, the following are equivalent:
\bi
    \item[(i)] $B \in \sS_d (\cH, \cJ)$ is CE,
    \item[(ii)] $B 1 \otimes h \notin \scr{H} (B)$ for any $h \in \cH$.
\ei
Classically, $b \in \sS (\cH ,\cJ)$ is CE if and only if it is an extreme point of the contractive, analytic $\scr{B} (\cH, \cJ)-$valued functions in $\D$.

\subsection{NC reproducing kernel Hilbert spaces} Let $\Om \subseteq \ncu$ be an NC set. A Hilbert space of NC functions, $\cH$, on $\Om$, is a \emph{non-commutative reproducing kernel Hilbert space} (NC-RKHS), if for each $Z \in \Om _n$, $n \in \N$, the linear point evaluation map $\ell _Z : \cH \rightarrow (\C ^{n\times n}, \mr{tr} )$, from $\cH$ into the Hilbert space $\C ^{n \times n}$ (where $\mr{tr}$ denotes the trace), equipped with the Hilbert--Schmidt (trace) inner product, $f \in \cH \stackrel{\ell _Z}{\mapsto} f(Z)$, is bounded \cite{BMV-NCrkhs}. In this case, letting $K_Z := (\ell _Z ) ^* : (\C ^{n\times n}, \mr{tr}) \rightarrow \cH$ denote the Hilbert space adjoint, we have that for any rank-one matrix, $yv^* \in \C ^{n\times n}$, $y,v \in \C ^n$, 
$$ \ipr{K_Z yv^*}{f}_{\cH} = \mr{tr} \, vy^*f(Z) = y^* f(Z) v. $$ That is, inner products against the \emph{NC kernel vector}, $K\{Z,y,v \} := K_Z (yv^*)$, give the bounded linear functional 
$\ell _{Z,y,v} : \cH \rightarrow \C$ defined by $\ell _{Z,y,v} (f) := y^* f(Z) v$. If $\cH$ is any NC-RKHS on $\Om$ with NC kernel vectors $K \{ Z, y , v \}$, one defines the \emph{non-commutative reproducing kernel} of $\cH$, $K : \Om _m \times \Om _n \rightarrow \scr{B} ( \C ^{m \times n})$, by: Given $Z \in \Om _m$, $y,v \in \C ^m$, $X \in \Om_n$ and $x,u \in \C ^n$, 
$$ y^* K(Z,W) [vu^*] x := \ipr{K\{Z,y,v\}}{K\{W,x,u\}}_\cH. $$ This defines a collection of (completely bounded) linear maps $$\{ K (Z,W) [ \cdot ] \in \scr{B} (\C ^{m\times n}) | \ Z \in \Om _m, W \in \Om _n \}, $$ so that $K(Z,Z) [\cdot ] : \C ^{m \times m} \rightarrow \C ^{m \times m}$ is completely positive for each $m \in \N$. Hence, this NC reproducing kernel is called a \emph{completely positive NC (CPNC) kernel}. As shown in \cite[Section 2]{BMV-NCrkhs} one can define a CPNC kernel, $K(Z,W)[\cdot]$, on an NC set $\Om$ as a function on $\Om \times \Om$, taking values in completely bounded linear maps, and that has natural invariance properties with respect to the grading, direct sums and joint similarities. Much of the classical reproducing kernel theory of Aronszajn and Moore extends naturally to this non-commutative context \cite{BMV-NCrkhs}. In particular, there is a bijective correspondence between CPNC kernels, $K$, on an NC set, $\Om$, and NC-RKHS of NC functions on $\Om$: If $\cH$ is an NC-RKHS on $\Om$, then its NC reproducing kernel is a CPNC kernel, and conversely, given a CPNC kernel, $K$ on $\Om$, one can construct an NC-RKHS, $\cH$, of NC functions on $\Om$, so that $K$ is its NC reproducing kernel \cite[Theorem 3.1]{BMV-NCrkhs}. This motivates the notation $\cH = \cH _{nc} (K)$, if $\cH$ is an NC-RKHS with CPNC reproducing kernel, $K$. 

In particular, the free Hardy space, $\hardy$, is an NC-RKHS on the NC set $\Om = \B ^d _\N$, the NC unit row-ball. Given $Z \in \B ^d _m$, $W\in \B ^d _n$ and $P \in \C ^{m \times n}$, the CPNC reproducing kernel of $\hardy$ is the \emph{NC Szeg\"o kernel},
$$ K(Z,W) [P] := \sum _{\ell =0} ^\infty \mr{Ad} _{Z,W^*} ^{\circ \ell} (P) = \sum _{\om \in \F} Z^\om P W^{\om *}, $$
where $\mr{Ad} _{Z,W^*} : \C ^{m \times n} \rightarrow \C ^{m \times n}$ is the linear map of \emph{adjunction} by $Z$ and $W^*$, 
$$ \mr{Ad} _{Z,W^*} (P) := \sum _{j=1} ^d Z_j P W_j ^*, $$ and $\mr{Ad} _{Z,W^*} ^{\circ \ell}$ denotes its $\ell-$fold self composition, and $\mr{Ad} _{Z,W^*} ^{\circ 0} := \mr{id} _{m\times n}$ is defined as the linear identity map on $m\times n$ matrices. (See, for example, \cite{BMV-NCrkhs,JM-freeCE,SSS}.) In the above $W ^{\om *} := (W^\om) ^* = (W^*) ^{\om ^\mrt}$.  

All of the NC-RKHS that we consider in this paper will be NC-RKHS of uniformly analytic NC functions on the NC unit row-ball. However, we will also need to consider vector--valued NC-RKHS. To be precise, let $\cJ$ be a finite or separable Hilbert space. A Hilbert space, $\cH$, of NC $\cJ-$valued functions on $\Om \subseteq \ncu$ is a (vector--valued) NC-RKHS, if for any $Z \in \Om _n$, $n \in \N$, the linear point evaluation map $\ell _Z : \cH \rightarrow \C ^{n \times n} \otimes \cJ$, is a bounded linear map into the Hilbert space $\C^{n \times n} \otimes \cJ$ where $\C^{n\times n}$ is equipped with the Hilbert--Schmidt inner product. In this case, as before, for any $Z \in \Om _m$, $y,v\in \C ^m$, and $g \in \cJ$, the NC kernel vectors, 
$$ k \{ Z , y\otimes g, v \} := \ell _Z ^* ( yv^* \otimes g ) \in \cH, $$ implement the linear functionals of point evaluation
$$ \ipr{k\{Z,y\otimes g,v\}}{f}_{\cH _{nc}(k)} = \ipr{y \otimes g}{f(Z) v}_{\C^m \otimes \cJ}. $$ Also as before,
for $Z \in \Om _m$ and $W \in \Om _n$, one defines the NC reproducing kernel $k$, 
$k (Z,W) [\cdot ] \in \scr{B} ( \C ^{m \times n} , \C ^{m \times n} \otimes \scr{B} (\cJ) )$, by the formula 
$$ (y \otimes g) ^* k(Z,W) [vu^*] (x \otimes g') := \ipr{k\{Z, y\otimes g, v\}}{k\{W, x \otimes g', u\}}_{\cH}, $$ and we write $\cH = \cH _{nc} (k)$.

\subsection{NC de Branges--Rovnyak spaces}

To construct our NC de Branges--Rovnyak model, we will need to introduce (left) NC de Branges--Rovnyak spaces of contractive and operator--valued left multipliers between vector--valued free Hardy spaces. Let $\cH, \cJ$ be complex and separable (or finite) Hilbert spaces and consider the vector--valued free Hardy spaces $\hardy \otimes \cH$ and $\hardy \otimes \cJ$, which can be viewed as NC-RKHS of NC $\cH$ and $\cJ-$valued functions on $\B ^d _\N$, respectively. Here, the $\scr{B} (\cH)-$valued CPNC Szeg\"o kernel, $K$, of $\hardy \otimes \cH$ is
$$ K(Z,W) [\cdot ] \otimes I_\cH [ P \otimes A] := K(Z,W)[P] \otimes A; \quad \quad Z\in \B ^d _m, W \in \B ^d _m, \ P \in \C ^{m \times n}, \ A \in \scr{B} (\cH), $$ where $K$ denotes the NC Szeg\"o kernel of $\hardy$.

We define the (left) \emph{de Branges--Rovnyak space} of $B$, as $\scr{H} (B) := \cH _{nc} (K^B)$, where $K^B$ is the CPNC \emph{de Branges--Rovnyak kernel} defined by 
$$ K^B (Z,W) [P \otimes A] := K(Z,W)[P] \otimes A - B(Z) \cdot  \left( K(Z,W) [P] \otimes A \right) \cdot B(W) ^*, $$ where $K$ denotes, as above, the CPNC Sz\"ego kernel of $\hardy$. (One can also define right NC de Branges--Rovnyak spaces of contractive, operator--valued right multipliers. We have chosen to work with left NC de Branges--Rovnyak spaces.) It is easy to check that as a vector space, $\scr{H} (B) \subseteq \hardy \otimes \cJ$, and the linear embedding $\mr{e} : \scr{H} (B) \hookrightarrow \hardy \otimes \cJ$ is a contraction. That is, $\scr{H} (B)$ is \emph{contractively contained} in $\hardy \otimes \cJ$. 

As in the classical setting, NC de Branges--Rovnyak spaces can be equivalently defined as operator--range spaces \cite{FM2,Sarason-dB}. Given any (bounded) linear operator, $A \in \scr{B} (\cH, \cJ)$, the \emph{operator--range space} of $A$, $\scr{R} (A)$, is a Hilbert space so that $\scr{R} (A) = \nbran A \subseteq \cJ$ as a complex vector space, and is a Hilbert space with respect to the inner product,
$$ \ipr{Ax}{Ay}_{\scr{R} (A)} := \ipr{x}{(I- P_{\nbker A}) y}_\cH. $$ It follows that the embedding, $\mr{e} : \scr{R} (A) \hookrightarrow \cJ$, is bounded with $\| \mr{e} \| = \| A \|$. In particular, if $A$ is a linear contraction, $\| A \| _{\scr{B} (\cH, \cJ )} \leq 1$, then this embedding is contractive and we say $\scr{R} (A)$ is contractively contained in $\cJ$. Moreover, if $A$ is a linear contraction, then one defines the \emph{complementary space} of $A$ as
$\scr{R} ^{\mr{c}} (A) := \scr{R} ( \sqrt{I_\cJ - A A ^*} )$, the operator--range space of the defect operator, $D_{A^*} = \sqrt{I-AA^*}$. See \cite[Chapter 16]{FM2} for more details on operator--range spaces, and \cite[Subsection 2.2]{BMN-LDvRK} for applications to reproducing kernel theory. As shown in \cite{Sarason-dB}, if $b \in [H^\infty _1]$ is a contractive analytic function in $\D$, then its de Branges--Rovnyak space can be defined as $\cH (b) := \scr{R} ^c (b(S))$, where $b(S) = M_b : H^2 \rightarrow H^2$ is the contractive linear map of multiplication by $b$. That is, $\cH (b) = \scr{R} ( \sqrt{I-b(S) b(S) ^*})$. Similarly, given any $B \in [ \mult \otimes \scr{B} (\cH , \cJ) ] _1$, one can show that $\scr{H} (B) = \scr{R} ^c (B(L) )$, where $B(L) := M^L _B : \hardy \otimes \cH \rightarrow \hardy \otimes \cJ$ is the contractive linear map of left multiplication by $B$ \cite{JMfree,JM-freeCE}. NC de Branges--Rovnyak spaces are a faithful analogue of classical, single--variable de Branges--Rovnyak spaces, and much of the classical theory extends to this NC setting, see \emph{e.g.} \cite{JM-subFock,JMfree,JM-freeCE}. In particular, as described in \cite{BBF-ncSchur}, any (left) NC de Branges--Rovnyak space, $\scr{H} (B)$, $B \in [ \mult \otimes \scr{B} (\cH, \cJ) ]_1$ is co-invariant for the \emph{right} free shifts, $R_j ^* \otimes I_\cJ \scr{H} (B) \subseteq \scr{H} (B)$, and while $B 1 \otimes h$, $h \in \cH$, generally does not belong to $\scr{H} (B)$, its backward right shifts always do:
$$ R_j ^* B 1 \otimes h \in \scr{H} (B); \quad \quad h \in \cH, \ 1 \leq j \leq d. $$ In the above, $1 = \fz ^\emptyset$ is the constant function (or ``vacuum vector") in the free Hardy space.

\subsection{Properties of CNC row contractions}

We follow notation from \cite{MR19} and introduce \emph{completely non-coisometric (CNC) row contractions}. Let $\cH$ be a complex, separable, or finite--dimensional Hilbert space. We say that a linear map, $T = (T_1, \cdots, T_d ): \cH \otimes \C ^d \rightarrow \cH$ is a \emph{row contraction} if 
$ \| T \| _{\scr{B} (\cH \otimes \C ^d, \cH)} \leq 1$,  or equivalently 
$$ \| T T^* \| _{\scr{B} (\cH) } = \left\| \sum _{j=1} ^d T_j T_J ^* \right\| \leq 1. $$
If a strict inequality holds in the above, we say that $T$ is a \emph{strict} row contraction. 
\begin{definition}\label{def:cnc.row.contraction}
A row contraction, $T = (T_1,\dots,T_d) : \cH \otimes \bC^d \to \cH$, is \emph{completely non-coisometric (CNC)} if there is no non-trivial subspace $\cK \subseteq \cH$ such that:
    \begin{enumerate}
        \item[(i)] $T_j^* \cK \subseteq \cK$ for each $1 \leq j \leq d$, \emph{i.e.} $\cK$ is co-invariant for $T$, and,
        \item[(ii)] $T^* \vert_{\cK} : \cK \to \cK \otimes \bC^d$ is a column isometry.
    \end{enumerate}
\end{definition}

It will be convenient to first consider \emph{CNC row partial isometries}, \emph{i.e.} CNC row contractions $V : \cH \otimes \bC^d \to \cH$ such that $V^*V$ (and hence $VV^*$) are projections, in which case $V^* V = I_{\nbker V^\perp}$. Any contraction, $T \in \scr{B} (\cH, \cJ)$, has an  \emph{isometric--pure decomposition}, $T=V+C$, into a partial isometry, $V$, and a pure contraction, $C$. Recall that for any row contraction $T$ on $\cH \otimes \bC^d$ the defect operators are defined as $$D_T := \sqrt{I_\cH \otimes I_d -T^* T} \qand D_{T^*} := \sqrt{I_\cH - T T^*}.$$ 

\begin{lemma}{\cite[Lemma 3.15]{MR19}}\label{isopure}
    Every row contraction $T : \cH \otimes \bC^d \to \cH$ can be uniquely decomposed as $T = V + C$, where
    \begin{enumerate}
        \item[$(i)$] $V := T P_{\nbker D_T}$ is a row partial isometry on $\cH \otimes \bC^d$,
        \item[$(ii)$] $C := T P_{\nbker D_T^\perp}$ is \emph{pure}, i.e., $\|C \mbf{h} \| < \|\mbf{h}\| \foral \mbf{h} \in \cH \otimes \bC^d$,
        \item[$(iii)$] $\nbker V = \nbker D_T^\perp$ and $\nbran V = \nbran D_{T^*}^\perp \subseteq \nbran C^\perp$.
    \end{enumerate}

    Consequently, $T$ is CNC if and only if $V$ is CNC.
\end{lemma}

Let us introduce some useful objects associated to a given row partial isometry $V$ on $\cH \otimes \bC^d$. For any $Z \in \bB^d_n$ and $n \in \bN$, define the corresponding \emph{restricted range space} in $\cH \otimes \bC^n$ as
\begin{align}
    \cR(V-Z) &= \nbran \big[(V \otimes I_n - I_\cH \otimes Z) (V^*V \otimes I_n)\big] \label{eqn:rest.ran.sp.1} \\
    &= \big[I_\cH \otimes I_n - (V_1^* \otimes Z_1 + \dots + V_d^* \otimes Z_d)\big] (\nbran V \otimes \bC^n). \label{eqn:rest.ran.sp.2}
\end{align}
We will typically refer to the spaces, $\cR (V-Z) ^\perp = \cH \ominus \cR (V-Z)$ as the \emph{$Z$-defect spaces of $V$}.
For notational convenience, we write \eqref{eqn:rest.ran.sp.1} and \eqref{eqn:rest.ran.sp.2} as
\begin{equation}\label{eqn:rest.ran.sp.3}
    \cR(V-Z) = \nbran [(V - Z) V^* V] = [I - ZV^*] (\nbran V \otimes \bC^n),
\end{equation}
where 
$$ ZV^* := (I _\cH \otimes Z_1, \cdots, I_\cH \otimes Z_d ) \bsm V_1 ^* \otimes I_n \vsm \\  \vdots  \vs \\ V_d ^* \otimes I_n \esm = \sum _{j=1} ^d V_j ^* \otimes Z_j. $$ 

For each $n\in\bN$ consider the subspaces,
\begin{equation}\label{eqn:def.H_n'}
    \cH_n' := \bigcap_{Z \in \bB^d_n} \cR(V-Z) \subset \cH \otimes \bC^n,
\end{equation}
and
\begin{equation}\label{eqn:def.H'}
    \cH' := \{ h \in \cH \ | \ h \otimes v \subseteq \cH_n', \ \forall v \in \bC^n,\, n \in \bN \} \subseteq \cH.
\end{equation}

A key relationship between these objects was observed in \cite{MR19}, we mention these results here for the reader's convenience. In the statement below, given a row partial isometry, $V = (V_1, \cdots, V_d) : \cH \otimes \C ^d \rightarrow \cH$ and a row contraction $T: \cJ \otimes \C ^d \rightarrow \cJ$, we say that $T$ \emph{extends} $V$ and write $V \subseteq T$, if $\cH \subseteq \cJ$ and $V = T V^*V$. (Note that this does not necessarily imply that $\nbker V \subseteq \nbker T$.)

\begin{lemma}{\cite[Lemma 3.9]{MR19}}\label{lem:RanVperp.and.R(V-Z).are.iso}
    Let $\cH \subseteq \cJ$ be complex, finite or separable Hilbert spaces, let $V$ be a row partial isometry on $\cH$ and $T$ be a row contraction on $\cJ$ so that $V \subseteq T$. Then,
    \begin{equation*}
        [I - T Z^*] : \cR(V - Z)^\perp \twoheadrightarrow \nbran V^\perp \otimes \bC^n
    \end{equation*}
    is a linear isomorphism (\emph{i.e.} a linear bijection) for each $Z \in \bB^d_n$ and $n \in \bN$.
\end{lemma}
In the above lemma statement,
$$ TZ^* = ( ZT^* ) ^* = \sum _{j=1} ^d T_j \otimes Z_j ^*. $$

For each $u \in \bC^n$ and $n \in \bN$, it will be convenient to define the \emph{tensor evaluation map} $I_\cH \otimes u^* : \cH \otimes \bC^n \to \cH$ given by
\begin{equation} \label{tensormapadj}
    [I_\cH \otimes u^*] (f \otimes v) = (u^*v) f \foral f \in \cH, \, v \in \bC^n.
\end{equation}
The Hilbert space adjoint of this linear map, $[I_\cH \otimes u] := [I_\cH \otimes u^* ] ^* : \cH \rightarrow \cH \otimes \C ^n$ acts as 
\be [I_\cH \otimes u]  h := h \otimes u. \label{tensormap} \ee

\begin{theorem}[Theorem 3.2, \cite{MR19}]\label{thm:char.cnc.row.partial.isom}
    If $V$ is a row partial isometry on $\cH \otimes \bC^d$, then $\cH'$ as in \eqref{eqn:def.H'} is the largest coinvariant subspace of $\cH$ for $V$ such that $V^* \vert_{\cH'}$ is a column isometry. Thus, $V$ is CNC if and only if $\cH = \cH'^\perp$. Moreover, 
    \begin{equation}\label{eqn:model.map.span.cond}
        \cH'^\perp = \bigvee_{\substack{u \in \bC^n,\, Z \in \bB^d_n, \\ n \in \bN}} [I_\cH \otimes u^*] [I - V Z^*]^{-1} (\nbran V^\perp \otimes \bC^n).
    \end{equation}
    
\end{theorem}

\begin{corollary}[Corollary 3.4, \cite{MR19}]\label{cor:char.CNC.row.contractions}
    If $T$ is a row contraction on $\cH \otimes \bC^d$, then $T$ is CNC if and only if
    \begin{equation*}
        \cH = \bigvee_{\substack{u \in \bC^n,\, Z \in \bB^d_n, \\ n \in \bN}} [I_\cH \otimes u^*] [I - T Z^*]^{-1} (\nbran D_{T^*} \otimes \bC^n).
    \end{equation*}
\end{corollary}

One can further show that 
$$ \bigvee_{\substack{u \in \bC^n,\, Z \in \bB^d_n, \\ n \in \bN}} [I_\cH \otimes u^*] [I - T Z^*]^{-1} (\nbran D_{T^*} \otimes \bC^n) = \bigvee _{\om \in \mathbb{F} _d ^+} T^\om \,\nbran D_{T^*}, $$ so that $T$ is CNC if and only if 
\be \cH = \bigvee _{\om \in \mathbb{F} _d ^+} T^\om  \, \nbran D_{T^*}, \label{CNCformula} \ee see \cite[Corollary 3.4]{MR19}.

\subsection{Gleason solutions}\label{subsec:Gleason.solutions}

Concrete examples of CNC row contractions and CNC partial isometries are provided by the extremal \emph{Gleason solutions} on free de Branges--Rovnyak spaces. (And our model will show that all CNC row contractions are given by such Gleason solutions, up to unitary equivalence.) 

\begin{definition}\label{def:Gleason.soln}
    Let $\cH _{nc} (K)$ be an $\cH-$valued NC-RKHS on the NC unit row-ball $\rball$. A row $d-$tuple of operators $X = (X_1, \cdots, X_d ) \in \scr{B} (\cH _{nc} (K) ) ^{1\times d}$ is said to be a \emph{right Gleason solution} for $\cH _{nc} (K)$, if for every $f \in \cH _{nc} (K)$ and every $Z=(Z_1, \cdots, Z_d) \in \bB^d_n$, 
    \be\label{eqn:def.Gleason.soln} (X^*f)(Z) \cdot Z := \sum _{j=1} ^d (X_j ^* f) (Z) Z_j = f(Z) - f(0_n), \ee see \cite{JMfree,JM-freeCE}. Such a (right) Gleason solution is said to be \emph{contractive}, if 
    $$ XX^* \leq I - K_0 {K_0}^*, $$ and \emph{extremal} if equality holds: $XX^* = I -K_0 {K_0} ^*$.    
\end{definition}

Fix Hilbert spaces $\cH$, $\cJ$, and let $B \in \sS_d(\cH, \cJ) = [\bH^\infty_d \otimes \sB(\cH,\cJ]_1$ be a contractive left multiplier from $\bH^2_d \otimes \cH$ into $\bH^2_d \otimes \cJ$. We consider the left de Branges--Rovnyak space $\scr{H}(B) \subseteq \bH^2_d \otimes \cJ$, which is a contractive containment. Since $\sH(B)$ is co-invariant for the right free shifts \cite{BBF-ncSchur}, we define
\begin{equation}
    X_j^* := R_j^* \otimes I_{\cJ} \big\vert_{\sH(B)} \foral 1 \leq j \leq d. \label{ncGS}
\end{equation}
By the results of \cite{JMfree,JM-freeCE}, given any NC de Branges--Rovnyak space, $\scr{H} (B)$, with $B \in \sS_d(\cH, \cJ)$, $\scr{H} (B)$ has the unique contractive Gleason solution given by \eqref{ncGS}, $X^* := R^* \otimes I _{\cJ _0}$, the restriction of the backward right shifts to $\scr{H} (B)$. 

\begin{definition} \label{def:GSforB}
A linear map $\mbf{B} : \cH \rightarrow \scr{H} (B) \otimes \C ^d$ is said to be a \emph{Gleason solution for $B$} if 
\be \mbf{B} (Z) \cdot Z := \sum _{j=1} ^d \mbf{B} _j (Z) Z_j = B(Z) - B(0_n). \label{GSforB} \ee Such a Gleason solution, $\mbf{B}$, is \emph{contractive}, if 
$$ \mbf{B} (Z) ^* \mbf{B(Z)} \leq I_\cH - B(0) ^* B(0), $$ and \emph{extremal} if equality holds. 
\end{definition}

Again, by \cite{JMfree,JM-freeCE}, the unique contractive Gleason solution for any $B \in \sS_d (\cH, \cJ)$ is given by the backward right shifts of $B$, \emph{i.e.} $\mbf{B} = [R^* \otimes I_\cJ] B$. 
By \cite[Theorem 6.4]{JM-freeCE}, the following are equivalent:
\bi
    \item[(i)] $B \in \sS _d (\cH, \cJ)$ is column--extreme (CE). 
    \item[(ii)] $\mbf{B} := R^* \otimes I_\cJ B$ is extremal, and
    \item[(iii)] $X := [R^* \otimes I_\cJ | _{\scr{H} (B)} \big] ^*$ is extremal and 
    $$ \mr{supp} (B) =\bigvee _{\substack{Z \in \B ^d _n; \\ y,v \in \C ^n, \ n \in \N }} I_\cH \otimes y^* B(Z) ^* I_\cJ \otimes v = \cH. $$
\ei
One can verify the following relationship between $X$ and $\mbf{B}$ which will be useful in the sequel. (This is the ``left" version of \cite[Proposition 5.6]{JM-freeCE}.)
\begin{prop} \label{GSkernel}
Let $B \in \sS_d (\cH, \cJ)$ and let $X,\mbf{B}$ be the unique, contractive Gleason solutions for $\scr{H} (B)$ and $B$, respectively. Then, for any $W \in \bB^d_n$, $x, u \in \bC^n$, $g \in \cJ$ and $n \in \bN$,
\begin{equation}\label{eqn:X^*.b.intertwine.ker.action}
    X^* K^B \{ W, g \otimes x, u \} = K^B \{W, g \otimes x, \mr{col} (W) u \} - \mbf{B} [I_\cH \otimes u^*] B(W) ^* g \otimes x.
\end{equation}
\end{prop}
In the above, 
$$ K^B \{W, g \otimes x, \mr{col} (W) u \} := \begin{bmatrix} K ^B \{W, g \otimes x, W_1 u \}  \\ \vdots \\ K ^B \{W, g \otimes x, W_d u \} \end{bmatrix} \in \scr{H} (B) \otimes \C ^d. $$

We now claim that $X$ is CNC whenever $B$ is column-extreme. We will later see in Section \ref{subsec:main.model} that every CNC row contraction can be modeled by the extremal Gleason solution for the (left) NC de Branges--Rovnyak space of some CE $B \in \sS _d (\cH, \cJ)$. Let us now assume that $B$ is column-extreme (CE) so that $X$ is the unique extremal Gleason solution, \emph{i.e.}
\begin{equation*}
    X X^* = I_{\sH(B)} - K_0^B {K_0^B}^*.
\end{equation*}
Here, $K_0^B : \cJ \to \sH(B)$ is defined as
\begin{equation*}
    K_0^Bg := K^B\{ 0, g \otimes 1, 1 \} \foral g \in \cJ.
\end{equation*}
Then,
\begin{equation}\label{eqn:K^b(0,0)}
    K^B(0,0) = {K_0^B}^* K_0^B = I_{\cJ} - B(0) B(0)^* = D_{B(0)^*}^2
\end{equation}
and setting
\begin{equation*}
    \sH_0 := \bigvee_{g \in \cJ} K_0^Bg \subseteq \sH(B),
\end{equation*}
we have
\begin{equation}
    P_{\sH_0} = K_0^B K^B(0,0)^{-1} {K_0^B}^* = K_0^B D_{B(0)^*}^{-2}{K_0^B}^*. \label{project}
\end{equation}
\begin{remark}
    By Lemma \ref{strict}, $B \in \sS_d (\cH, \cJ)$ is strictly contractive if and only if $\| B(0) \| <1$, in which case both defect operators $D_{B(0)^*}$ and $D_{B(0)}$ are invertible. If $B$ is only purely contractive, then the defect spaces $\scr{D} _{B(0)^*}$ and $\scr{D} _{B(0)}$ are non-closed, dense subspaces of $\cJ$ and $\cH \otimes \C ^d$, respectively, so that $D_{B(0)^*} ^{-1}$ and $D_{B(0)} ^{-1}$ are closed and densely--defined operators on these defect spaces. Since $K_0 ^{B *} K_0 ^B = D_{B(0) ^*} ^2$, it follows, from the polar decomposition that $\nbran K_0 ^{B *}  \subseteq \nbdom D_{B(0) ^*} ^{-1}$, and the projection operator in \eqref{project} is still well-defined.
\end{remark}

\begin{remark}\label{rem:X.partial.iso.iff.b(0)=0}
    It is worth pointing out that $X$ is a row partial isometry if and only if $B(0) = 0$. To see this, note that
    \begin{alignat*}{2}
        && X X^* X X^* &= X X^* \\
        \iff && \big(I_{\sH(B)} - K_0^b {K_0^B}^*\big)\big(I_{\sH(B)} - K_0^B {K_0^B}^*\big) &= \big(I_{\sH(B)} - K_0^B {K_0^B}^*\big) \\
        \iff && K_0^B \big[I_{\sH(B)} - D_{B(0)^*}^2\big] {K_0^B}^* &= 0 \\
        \implies && {K_0^B}^* K_0^B \big(B(0)B(0)^*\big) {K_0^B}^* K_0^B &= 0 \\
        \implies && D^2_{B(0)^*} B(0)B(0)^* D^2_{B(0)^*} &= 0.
    \end{alignat*}
    Thus, if $X$ is a row partial isometry, then the injectivity of $D_{B(0)^*}$ forces $B(0) = 0$. Conversely, if $B(0) = 0$ then $D_{B(0)^*} = I_{\sH(B)}$, and we can trace the reverse implications above to conclude that $X$ is a row partial isometry.
\end{remark}

Observe that
\begin{alignat}{2}
    &&D^2_{X^*} &= I_{\sH(B)} - X X^* = K_0^B {K_0^B}^*, \nonumber\\
    \implies&& \nbker D_{X^*} &= \nbker D_{X^*}^2 = \sH_0^\perp, \nonumber\\
    \implies&& P_{\nbker D_{X^*}} &= I_{\sH(B)} - \left( K_0^B K^B(0,0)^{-1} {K_0^B}^* \right), \nonumber\\
    \implies&& V^* := X^* P_{\nbker D_{X^*}} &= X^* \left[I_{\sH(B)} - K_0^B D_{B(0)^*}^{-2} {K_0^B}^*\right].\label{eqn:partial.iso.part,of.Gleason.soln}
\end{alignat}
By Lemma \ref{isopure}, $V$ is a row partial isometry. It is easy to then verify that
\begin{equation}\label{eqn:range.defect.T*}
    \nbran D_{X^*} = \nbran V^\perp = \sH_0 = \bigvee \nbran K_0 ^B
\end{equation}
and, consequently, if $B(0)$ is a pure contraction, $\Gamma(0) := K_0^B D_{B(0)^*}^{-1} : \scr{D} _{B(0)^*} \subseteq \cJ \twoheadrightarrow \nbran V^\perp$ extends by continuity to a surjective isometry from $\cJ$ onto $\nbran V^\perp$. (If $B(0)$ is a strict contraction, $D_{B(0)^*}$ is bijective and invertible.)

\begin{lemma}\label{lem:ker.action.model}
Let $B \in \sS _d (\cH, \cJ)$, as above.  For any $Z \in \bB^d_n$, $u, x \in \bC^n$, $g \in \cJ$ and $n \in \bN$, we have
    \begin{equation}\label{eqn:ker.action.model}
        [I_{\sH(B)} \otimes u^*] [I - X Z^*]^{-1} (K_0^B g \otimes x) = K^B\{Z, g \otimes x, u\}.
    \end{equation}
    In particular,
    \begin{equation}\label{eqn:dBR.sp.span.Gamma(0)}
        \sH(B) = \bigvee_{\substack{u \in \bC^n,\, Z \in \bB^d_n, \\ n \in \bN}} [I_{\sH(B)} \otimes u^*][I - X Z^*]^{-1} (\nbran \Gamma(0) \otimes \bC^n).
    \end{equation}
\end{lemma}

\begin{proof}
    Let $f \in \sH(B)$ be fixed but arbitrary. Then $f \in \hardy \otimes \cJ$ and hence is represented by a free FPS, 
$$ f = \sum _{\om \in \F} \hat{f} _\om \fz ^\om = \sum \hat{f} _{\om ^\mrt} \fz ^{\om ^\mrt}, $$ where $\fz = (\fz _1, \cdots, \fz _d)$ is a $d-$tuple of formal NC variables, and the coefficients, $\hat{f} _\om \in \cJ$ are square--summable, $\sum \| \hat{f} _\om \| _\cJ ^2 < +\infty$. For any $\alpha \in \F$, 
$$ R ^{\alpha ^\mrt *} \otimes I_\cJ f  =  \sum _{\om \in \F} \hat{f} _\om \underbrace{R ^{\alpha ^\mrt *} \fz ^\om}_{= R ^{\alpha ^\mrt *} R^{\om ^\mrt} 1}, $$ where $R ^{\alpha ^\mrt *} R^{\om ^\mrt} 1$ is $0$ unless $\om ^\mrt = \alpha ^\mrt \beta ^\mrt$ for some $\beta \in \F$, so that 
$$  R^{\alpha ^\mrt *} \otimes I_\cJ f = \sum _{\beta \in \F} \hat{f} _{\beta \alpha} R^{\beta ^\mrt} 1 = \sum _\beta \hat{f} _{\beta \alpha} \fz ^\beta. $$

Consequently, for any $g \in \cJ$ and any $\alpha \in \F$,
\be (K_0^B g)^* X^{\alpha ^\mrt *} f = g^* (X^{* \alpha} f  ) (0) = g^* \hat{f} _{\alpha}. \ee
We can therefore calculate
    \begin{align*}
        \bip{[I_{\sH(b)} \otimes u^*] [I - X & Z^*]^{-1} (K_0^B g \otimes x), f}_{\sH(B)} \\
        & = \bip{ K_0^B g \otimes x, [I - X^* Z]^{-1} (f \otimes u)}_{\sH(B) \otimes \bC^n} \\
        & = \sum_{\alpha \in \bF^+_d} \bip{K_0^B g \otimes x, ({X^*}^\alpha \otimes Z^\alpha) (f \otimes u)}_{\sH(B) \otimes \bC^n} \\
        & = \sum_{\alpha \in \bF^+_d} x^* Z^\alpha u \bip{K_0^B g, {X^{\alpha^\dagger}}^* f}_{\sH(B)} \\
        & = \sum_{\alpha \in \bF^+_d} x^* Z^\alpha u \, \ip{g, \hat{f_\alpha}}_{\cJ_0} \\
        & = \ip{g \otimes x, f(Z)u}_{\cJ \otimes \bC^n} \\
        & = \bip{K^B\{Z, g \otimes x, u\}, f}_{\sH(B)}.
    \end{align*}
    Since $f \in \scr{H} (B)$ was chosen arbitrarily, we obtain \eqref{eqn:ker.action.model}.
\end{proof}

Applying Corollary \ref{cor:char.CNC.row.contractions}, Lemma \ref{lem:ker.action.model} and \eqref{eqn:range.defect.T*} establishes that $X$ is CNC.

\begin{corollary}\label{cor:X.is.CNC}
    Let $B \in \sS_d(\cH, \cJ)$ be a CE and purely contractive (left) multiplier. Then, the extremal Gleason solution $X$ for $\sH(B)$ is CNC.
\end{corollary}

\section{Model maps for CNC row partial isometries}\label{sec:model.for.CNC.row.partial.iso}

In this section, given a CNC row partial isometry, $V : \cH \otimes \C ^d \rightarrow \cH$, on a complex, separable Hilbert space, $\cH$, we explicitly construct a ``functional" or function--theoretic model for $V$, as an operator acting on a Hilbert space of graded functions on the unit row-ball, $\B ^d _\N$. 

\subsection{Graded model}\label{subsec:nc.func.theor.model.for.cnc.row.partial.isom}

\begin{definition}\label{def:nc.model.map}
    Let $V : \cH \otimes \bC^d \to \cH$ be a row partial isometry. A \emph{model triple} $(\gamma, \cJ_\infty,\cJ_0)$ for $V$ consists of Hilbert spaces $\cJ_\infty$ and $\cJ_0$ with
    \begin{equation*}
        \dim\cJ_\infty = \dim\nbker V \qand \dim \cJ_0 = \dim\nbran V ^\perp,
    \end{equation*}
    and a \emph{model map},
    \begin{align*}
        \gamma :
        \begin{cases}
            \ \bB^d_n &\rightarrow \scr{B}(\cJ_0 \otimes \bC^n, \cH \otimes \bC^n) \\
            \{\infty\} &\twoheadrightarrow \scr{B}(\cJ_\infty, \cH \otimes \bC^d)
        \end{cases},
    \end{align*}
    such that $\gamma(\infty)$ and $\gamma(0)$ are isometries onto $\nbker V$ and $\nbran V^\perp$, and so that $\gamma(Z): \cJ _0 \otimes \C ^n \rightarrow \cH \otimes \C ^n$ is a linear isomorphism for each $Z \in \bB^d_\bN$ onto the $Z$-defect space of $V$,
    \begin{equation*}
        \nbran \gamma(Z) = \cR(V-Z)^\perp.
    \end{equation*}
\end{definition}

If a model map has the additional property that $\ga (Z) ^* : \cH \otimes \C ^n \rightarrow \cJ _0 \otimes \C ^n$ is a $\scr{B} (\cH, \cJ _0)-$valued NC function in $\B ^d _\N$, we will call it an \emph{NC model map}.

\begin{remark}\label{rem:canon.model.triple}
    Given a row partial isometry $V$ on $\cH$, Lemma \ref{lem:RanVperp.and.R(V-Z).are.iso} shows that if $T : \cH \otimes \C ^d \rightarrow \cH$ is any row contractive extension of $V$, $V \subseteq T$ (again, this means that $T V^* V = V$),
    then
    \begin{equation*}
        [I - T Z^*]^{-1} : \nbran V^\perp \otimes \bC^n \twoheadrightarrow \cR(V - Z)^\perp
    \end{equation*}
    is a linear isomorphism for each $Z \in \bB^d_n$ and $n \in \bN$. Therefore, for each such extension $T\supseteq V$, acting on $\cH$, we obtain the model map, 
    \begin{align}\label{eqn:def.canon.model.map}
        \Gamma _T :
        \begin{cases}
            \ Z &\rightarrow  \ [I - T Z^*]^{-1} \vert_{\nbran V^\perp \otimes \bC^n}, \, \forall Z \in \bB^d_n \\
            \ \infty &\rightarrow \  I_{\nbker V}.
        \end{cases},
    \end{align}
and the model triple, $(\Ga _T, \nbker V, \nbran V ^\perp)$.
It is easily verified that each such $\Ga _T$ is an NC model map, and we will refer to $\Ga _T$ as a \emph{standard NC model map}. Indeed, given any $g \in \nbran V ^\perp$, $h \in \cH$ and $Z \in \B ^d _n$,
\begin{equation}\label{eqn:Ga(Z)^*.is.NC}
    [h^* \otimes I_n] \Ga _T (Z) ^* [g \otimes I_n] = [h^* \otimes I_n] [I_\cH \otimes I_n - Z \otimes T^*]^{-1} [g \otimes I_n], 
\end{equation}
is the realization formula of a uniformly analytic NC function in $\B ^d _\N$, in the sense of \cite{AMS-opreal}. Most simply, choosing $T=V$, we obtain the \emph{canonical NC model map} and \emph{canonical NC model triple}, $(\Ga _V, \nbker V, \nbran V ^\perp)$.     
\end{remark}

Let $V : \cH \otimes \bC^d \to \cH$ be a row partial isometry with a model triple $(\gamma, \cJ_\infty, \cJ_0)$. For any $h \in \cH$, $W \in \bB^d_n$ and $n \in \bN$, we define a mapping $h^\gamma(W) : \bC^n \to \cJ_0 \otimes \bC^n$ as
\begin{equation}\label{eqn:def.h^gamma}
    h^\gamma(W)v := \gamma(W)^* (h \otimes v) \foral v \in \bC^n.
\end{equation}
We now consider the \emph{model space} corresponding to $(\gamma,\cJ_\infty, \cJ_0)$, given by
\begin{equation*}
    \cH^\gamma := \{ h^\gamma \ | \  h \in \cH \}.
\end{equation*}
Each $h^\gamma$ can be thought of as a function $h^\ga : \bB^d_\bN \rightarrow \scr{B}(\bC, \cJ_0)_\bN$. Here, given a Hilbert space, $\cJ$, we define $\cJ _n := \cJ \otimes \C ^{n\times n}$, $\cJ _\N := \bigsqcup _{n=1} ^\infty \cJ _n$, and 
\be \scr{B} (\C, \cJ) _\N := \bigsqcup_{n \in \bN} \scr{B}(\bC^n, \cJ \otimes \bC^n) \cong \bigsqcup_{n \in \bN} \scr{B}(\bC,\cJ) \otimes \bC^{n \times n} \cong {\cJ}_\bN. \label{opNCuni} \ee
In other words, $\cH^\gamma$ is a Hilbert space of graded functions between $\bB^d_\bN$ and ${\cJ_0}_\bN$.

\begin{prop}\label{prop:U_gamma.well.def}
    Let $V: \cH \otimes \C ^d \rightarrow \cH$ be a row partial isometry with model triple, $(\ga, \cJ_\infty, \cJ _0)$. The mapping $\cU^\gamma : \cH \to \cH^\gamma$ given by
    \begin{equation*}
        \cU^\gamma h := h^\gamma \foral h \in \cH
    \end{equation*}
    is a linear isomorphism if and only if $V$ is CNC. If $V$ is CNC then $\cH ^\ga$ is a Hilbert space when equipped with the push-forward inner product:
    \begin{equation*}
        \langle h^\gamma, g^\gamma \rangle_{\cH^\gamma} := \langle h, g \rangle_\cH \foral h^\gamma, g^\gamma \in \cH^\gamma, 
    \end{equation*}
    and $\cU ^\ga : \cH \twoheadrightarrow \cH ^\ga$ is an onto isometry. In this case, $\cH^\gamma$ is a Hilbert space of ${\cJ_0}_\bN$-valued graded functions on $\bB^d_\bN$.
\end{prop}

\begin{proof}
    It suffices to determine the injectivity of $\cU^\gamma$, since it is onto by definition. To this end, fix $h \in \cH$ and note that for each $W \in \bB^d_n$ and $n \in \bN$ we have
    \begin{align*}
        h \in \nbker \cU^\gamma & \Leftrightarrow h^\gamma(W) \equiv 0 \\
        & \Leftrightarrow \ip{g \otimes x, h^\gamma(W)u}_{\cJ_0 \otimes \bC^n} = 0, && \forall x, u \in \bC^n \AND g \in \cJ_0, \\
        & \Leftrightarrow \ip{g \otimes x, \gamma(W)^*(h \otimes u)}_{\cJ_0 \otimes \bC^n} = 0, && \forall x,u \in \bC^n \AND g \in \cJ_0, \\
        & \Leftrightarrow \bip{[I_\cH \otimes u^*] \gamma(W) (g \otimes x),h}_\cH = 0, && \forall x,u \in \bC^n \AND g \in \cJ_0.
    \end{align*}
    Since $W \in \bB^d_n$ and $n \in \bN$ were chosen arbitrarily, this implies
    \begin{equation*}
        h \in \nbker \cU^\ga \Leftrightarrow h \perp \bigvee_{\substack{u \in \bC^n, \, Z \in \bB^d_n, \\ n \in \bN}} [I_\cH \otimes u^*] \nbran \gamma(Z).
    \end{equation*}
    In particular, \eqref{eqn:model.map.span.cond} implies that $\nbker \cU^\gamma = \cH'^\perp$. It then follows from Theorem \ref{thm:char.cnc.row.partial.isom} that $\cU^\gamma$ is injective if and only if $V$ is CNC. The remaining assertions follow easily from this observation.
\end{proof}

\begin{remark} \label{gradedRKHS}
A model map, $\ga$, for an arbitrary row partial isometry, $V$, as defined above in Definition \ref{def:nc.model.map}, need not respect direct sums or joint similarities. Hence, the corresponding \emph{model space}, $\cH ^\ga$, need not be a Hilbert space of NC functions. Nevertheless, we will see that for any $Z \in \B ^d _n$, $y,v \in \C ^n$ and $g \in \cJ _0$, the linear functional, $\ell _{Z,g \otimes y, v} : \cH ^\ga \rightarrow \C$, defined by 
$$ \ell _{Z,g \otimes y, v} (h^\ga) := \ipr{g \otimes y}{h^\ga (Z) v}_{\cJ_0 \otimes \C ^n}, $$ is bounded. Hence, $\cH ^\ga$, does consist of $\cJ_0-$valued functions on the NC unit row-ball, which are \emph{graded} in the sense that $h ^\ga (Z) \in \cJ_0 \otimes \C ^{n\times n}$, for any $Z \in \B ^d _n$, and it is a reproducing kernel Hilbert space in the above sense. We will call $\cH ^\ga$ a \emph{graded RKHS} on $\B ^d _\N$. Namely, for each $Z, g \otimes y$ and $v$ as above, there exists, by the Riesz Lemma, a unique $K ^\ga \{ Z, g \otimes y, v \} \in \cH ^\ga$ so that 
$$ \ell _{Z,g \otimes y,v} (h^\ga) = \ipr{K ^\ga \{ Z, g \otimes y, v \} }{h^\ga } _{\cH ^\ga}. $$ We will also define the \emph{graded reproducing kernel} for $\cH ^\ga$ as in the case of an NC-RKHS: For each $Z \in \B ^d _m$ $y,v \in \C^m$, $W \in \B ^d _n$, $x,u \in \C^n$ and $f,g \in \cJ _0$, $K^\ga (Z,W) [ \cdot ] : \C ^{m \times n} \rightarrow \scr{B} (\cJ _0) \otimes \C ^{m\times n}$ is the linear map defined by
$$ f^* \otimes y^* K^\ga (Z,W) [vu^*] g \otimes x := \ipr{K^\ga \{Z,f \otimes y,v\}}{K^\ga \{W, g \otimes x, u\}}_{\cH ^\ga}. $$
If $\ga$ is an NC model, such as one of the standard NC models, $\Ga _T$, for $T \supseteq V$, or $\ga = \Ga _V$, the canonical NC model, then it easily checked that any $h^\ga \in \cH ^\ga$ is a $\cJ _0-$valued NC function and $\cH ^\ga = \cH _{nc} (K^\ga)$ will be an NC-RKHS of $\cJ _0-$valued NC functions. 
\end{remark}

Let us now explore function-theoretic properties of $\cH^\gamma$ by realizing it as a RKHS of ${\cJ_0}_\bN$-valued graded functions.

\begin{prop}\label{prop:H^gamma.as.RKHS}
    For any given $W \in \bB^d_n$, $x,u \in \bC^n$, $n \in \bN$ and $g \in \cJ_0$, let
    \begin{equation}\label{eqn:def.K^gamma}
        K^\gamma \{W,g \otimes x,u\} := [\cU^\gamma \otimes u^*] \gamma(W) (g \otimes x) \in \cH^\gamma.
    \end{equation}
    Then, for all $h^\gamma \in \cH^\gamma$, we have
    \begin{equation}\label{eqn:kernel.formula.general}
        \langle K^\gamma \{W,g \otimes x,u\},h^\gamma \rangle_{\cH^\gamma} = \langle g \otimes x, h^\gamma(W)u \rangle_{\cJ_0 \otimes \bC^n}.
    \end{equation}

    In particular, the evaluation map $\Lambda_W : \cH^\gamma \to \scr{B}(\bC^n,\cJ_0 \otimes \bC^n)$ given by $\Lambda_W h^\gamma := h^\gamma(W)$ is bounded for each $W \in \bB^d_n$ and $n \in \bN$.
\end{prop}
In the above proposition statement, the operator $[U^\ga \otimes u^*]$ is defined as the composition, $U^\ga [I_\cH \otimes u^*]$, where $I_\cH \otimes u^* : \cH \otimes \C^n \rightarrow \cH$ is the linear tensor map defined in \eqref{tensormap}.
\begin{proof}
    Let $K^\gamma\{W,g \otimes x,u\}$ be as in the hypothesis. Then, using \eqref{tensormap}, for each $h^\gamma \in \cH^\gamma$,
    \begin{align*}
        \bip{K^\gamma\{W,g \otimes x,u\}, h^\gamma}_{\cH^\gamma} &= \bip{[U^\gamma \otimes u^*]\gamma(W)(g \otimes x), U^\gamma h}_{\cH^\gamma} \\
        & = \ipr{\ga (W) g \otimes x}{[I_\cH \otimes u] h}_{\cH} \\
        &= \ip{g \otimes x, \gamma(W)^* (h \otimes u)}_{\cJ_0 \otimes \bC^n} \\
        &= \ip{g \otimes x, h^\gamma(W)u}_{\cJ_0 \otimes \bC^n}.
    \end{align*}
    Thus, \eqref{eqn:kernel.formula.general} holds. The final assertion is immediate from the above.
\end{proof}

The maps $K^\gamma \{ W, g \otimes x, u\}$ will be referred to as the \emph{kernel functions} of $\cH^\gamma$.

\begin{theorem}\label{thm:reproducing.kernel.of.H^gamma}
    If $V$ is CNC, then $\cH^\gamma$ is a graded reproducing kernel Hilbert space of ${\cJ_0}_\bN$-valued graded functions on $\bB^d_\bN$ with the graded kernel $K^\gamma : \bB^d_\bN \times \bB^d_\bN \to \scr{B}(\bC,\scr{B}(\cJ_0))_{\bN \times \bN}$ given by
    \begin{equation}\label{eqn:kernel.of.H^gamma}
        K^\gamma(Z, W)[A] := \gamma(Z)^* [I_\cH \otimes A] \gamma(W),
    \end{equation}
    for each $Z \in \bB^d_n$, $W \in \bB^d_m$, $A \in \bC^{n \times m}$ and $n, m \in \bN$.
\end{theorem}
The graded kernel of a RKHS of graded functions on $\B ^d _\N$ was defined in Remark \ref{gradedRKHS}.
\begin{proof}
    For each pair of kernel functions $K^\gamma\{ Z, f \otimes y, v \}$ and $K^\gamma\{ W, g \otimes x, u \}$ with $Z \in \bB^d_n$, $W \in \bB^d_m$, $m, n \in \bN$ and $f,g \in \cJ_0$, we use the tensor map \eqref{tensormap} to obtain
    \begin{align}
        \ipr{K^\gamma\{ Z, f \otimes y, v \}}{ K^\gamma\{ W, g \otimes x, u \}}_{\cH^\gamma} 
        &= \bip{[\cU^\gamma \otimes v^*] \gamma(Z) (f \otimes y), [\cU^\gamma \otimes u^*]\gamma(W) (g \otimes x)}_{\cH^\gamma} \nonumber\\
        &= \bip{f \otimes y, \gamma(Z)^* [I_\cH \otimes v u^*] \gamma(W) (g \otimes x)}_{\cJ_0 \otimes \bC^n}.\label{eqn:ker.at.Z.W}
    \end{align}
    This implies \eqref{eqn:kernel.of.H^gamma}, since the action of $K^\gamma(Z,W)$ on all $A \in \bC^{n \times m}$ is determined by its action on rank-one matrices $v u^*$, for $v \in \C ^n, u \in \C ^m$.

    Lastly, we need to show that
    \begin{equation*}
        \cH^\gamma = \widehat{\cH}(K) := \bigvee_{n \in \bN} \{ K^\gamma\{Z,f \otimes y,v\} | \ Z \in \bB^d_n, \, y,v \in \bC^n, \, f \in \cJ_0 \}.
    \end{equation*}
    However, this is clear from \eqref{eqn:kernel.formula.general} and Proposition \ref{prop:U_gamma.well.def} since $h^\gamma \in \widehat{\cH}(K)^\perp$ implies that $h^\gamma(W) = 0$ for all $W \in \bB^d_\bN$. This completes the proof.
\end{proof}

\subsection{Right shift action}\label{subsec:right.shifts}

We end this section by recording how the CNC row partial isometry, $V$, transforms under $\cU^\gamma$. First, we define for each $\mbf{h} ^\ga  \in \cH^\gamma \otimes \bC^d$ and $Z \in \bB^d_\bN$ the \emph{right shift} action
\begin{equation*}
    M_Z^R \mbf{h}^\ga (Z) := \sum_{j = 1}^d h_j^\gamma(Z) Z_j.
\end{equation*}
It is unclear from the definition when the right hand side of the above equation lies in $\cH^\gamma$. We address this in the next proposition.

\begin{prop}\label{prop:V.transf.under.U^gamma}
    If $\widehat{V} := \cU^\gamma V [{\cU^\gamma}^* \otimes I_d] : \cH^\gamma \otimes \bC^d \to \cH^\gamma$ then,
    \begin{equation}\label{eqn:V^hat.on.init.sp}
        \widehat{V} \big\vert_{\nbker \widehat{V}^\perp} = M_Z^R.
    \end{equation}
    
    Moreover, $\widehat{V} \big\vert_{\nbker \widehat{V}^\perp}$ is an isometry and
    \begin{equation}\label{eqn:Ran.of.V^hat}
        \nbran \widehat{V} = \left\{ h^\gamma \in \cH^\gamma \ | \  h^\gamma(0) = 0 \right\}.
    \end{equation}
\end{prop}

\begin{proof}
    Let us fix $\mbf{h}^\gamma \in \nbker \widehat{V}^\perp$, $Z \in \bB^d_n$ and $v \in \bC^n$ for some $n \in \bN$, and write $\mbf{h} := ({\cU^\gamma}^* \otimes I_d) \mbf{h}^\gamma$. Using \eqref{eqn:def.h^gamma}, note that
    \begin{align*}
        \widehat{V}\mbf{h}^\gamma(Z)v &= [\cU^\gamma V \mbf{h}](Z)v \\
        &= \gamma(Z)^*(V \mbf{h} \otimes v) \\
        &= \gamma(Z)^*[V \otimes I_n](\mbf{h} \otimes v) \\
        &= \gamma(Z)^*[V \otimes I_n - I_\cH \otimes Z] (\mbf{h} \otimes v) + \gamma(Z)^* [I_\cH \otimes Z](\mbf{h} \otimes v).
    \end{align*}
    Since $\mbf{h} \in \nbker V^\perp = \nbran V^*V$, it follows from \eqref{eqn:rest.ran.sp.1} that
    \begin{equation*}
        [V \otimes I_n - I_\cH \otimes Z](\mbf{h} \otimes v) \in \cR(V-Z)
    \end{equation*}
    and thus, the first summand above is $0$. Here, $[I_\cH \otimes Z] : \cH \otimes \bC^d \otimes \bC^n \to \cH \otimes \bC^n$ is the map defined via
    \begin{equation*}
        [I_\cH \otimes Z] \bsm
            h_1 \otimes v_1 \\
            \vdots \\
            h_d \otimes v_d
        \esm = \sum_{j = 1}^d h_j \otimes Z_j v_j, \foral h_j \in \cH, \, v_j \in \bC^n, \, 1 \leq j \leq d.
    \end{equation*}
    We then continue the above calculation and obtain \eqref{eqn:V^hat.on.init.sp} as follows:
    \begin{align*}
        \widehat{V}\vec{h^\gamma}(Z)v = \gamma(Z)^* [I_\cH \otimes Z](\mbf{h} \otimes v) = \gamma(Z)^* \sum_{j = 1}^d h_j \otimes Z_j v = \sum_{j = 1}^d h_j^\gamma(Z)Z_jv = M_Z^R \mbf{h}^\gamma(Z)v.
    \end{align*}

    That $\widehat{V}$ is an isometry on $\nbker \widehat{V}^\perp$ follows from the fact that $V$ is an isometry on $\nbker V^\perp$ and $\cU^\gamma$ is  unitary. Lastly, in order to show \eqref{eqn:Ran.of.V^hat}, note that any $h^\gamma \in \cH^\gamma$ such that $h^\gamma(0) = 0$ corresponds to an $h \in \nbran V$ under $\cU^\gamma$, via \eqref{eqn:def.h^gamma}. Equation \eqref{eqn:Ran.of.V^hat} follows easily from this correspondence.
\end{proof}

\begin{remark}
While $\widehat{V}$ acts as right multiplication by $Z$ on its initial space, by the above proposition, it cannot act as right multiplication by $Z$ on the entire space, $\cH ^\ga$. This is because, as we shall eventually observe in Section \ref{subsec:char.map}, $\widehat{V}$ is unitarily equivalent to $X$, the extremal Gleason solution of a column--extreme (CE) contractive and operator--valued multiplier, $B \in \sS_d(\cJ_\infty, \cJ_0)$ between vector-valued free Hardy spaces, via a unitary left multiplication operator $M^L _f : \cH ^\ga \twoheadrightarrow \scr{H} (B)$. The contractive Gleason solution, $X$, of any contractive (and operator--valued) left multiplier, $B$, is simply given by the restriction of the backward right shifts, $X ^* := [R^* \otimes I_{\cJ_0}] | _{\scr{H} (B)}$, and since $B$ is CE, there is no $g \in \cJ _\infty$ so that $B 1 \otimes g \in \scr{H} (B)$ by \cite[Theorem 6.4]{JM-freeCE}, and hence there is no $h \in \scr{H} (B)$ so that $R_j \otimes I _{\cJ _0} h \in \scr{H} (B)$ for any $1 \leq j \leq d$, by \cite[Corollary 6.14]{JM-freeCE}.
Hence $\hat{V}$ cannot be equal to right multiplication by $Z$ as then $X$ would also act as right multiplication by $Z$, contradicting that $B$ is CE.
\end{remark}

\begin{lemma}\label{lemma:action.M^R_Z^*.on.ker.func}
    For any kernel function $K^\gamma\{W, g \otimes x, u\}$, 
    \begin{equation*}
        \widehat{V} ^* K^\gamma\{W, g \otimes x, u\} = (I_{\cH^\ga \otimes \bC^d} - P_\infty) \begin{bmatrix}
            K^\gamma\{W, g \otimes x, W_1 u\} \\
            \vdots \\
            K^\gamma\{W, g \otimes x, W_d u\}
        \end{bmatrix} = K^\ga \{ W, g\otimes x, \mr{col} (W)  u \}.
    \end{equation*}
\end{lemma}

\begin{proof}
    Let $\mbf{f}^\ga \in \cH^\ga \otimes \bC^d$ be fixed but arbitrary, and let
$$ \mbf{h}^{\ga} := (I_{\cH^\ga \otimes \bC^d} - P_\infty) \mbf{f}^\ga \in \nbker \widehat{V}^\perp. $$
    Then, note that since $\widehat{V}$ acts as right multiplication by $Z$ on its initial space,
\ba
\ipr{\widehat{V} ^* K^\gamma\{W, g \otimes x, u\}}{\mbf{h}^\gamma}_{\cH^\gamma \otimes \bC^d}
        &= & \ipr{K^\gamma\{W, g \otimes x, u\}}{\widehat{V}\mbf{h}^{\gamma}}_{\cH^\gamma}\\
        &=& \sum_{j = 1}^d \ipr{g \otimes x}{\mbf{h}^\gamma_j(W)W_j u}_{\cJ_0 \otimes \bC^n}\\
        &=& \sum_{j = 1}^d \ipr{K^\gamma\{W, g \otimes x, W_j u\}}{\mbf{h}^{\gamma}_j}_{\cH^\gamma}\\
        &=& \ipr{\begin{bmatrix}
            K^\gamma\{W, g \otimes x, W_1 u\}\\
            \vdots\\
            K^\gamma\{W, g \otimes x, W_d u\}
        \end{bmatrix}}{ \mbf{h}^{\gamma}}_{\cH^\gamma \otimes \bC^d} \\
        &=& \ipr{(I_{\cH^\ga \otimes \bC^d} - P_\infty) K ^\ga \{ W, g \otimes x, \mr{col} (W) u \}}{ \mbf{f}^\ga}_{\cH^\ga \otimes \bC^d}.
    \ea
    This completes the proof.
\end{proof}

\subsection{The characteristic function}\label{subsec:char.map}

In this subsection, we analyze the graded reproducing kernel, $K^\gamma$ introduced in Theorem \ref{thm:reproducing.kernel.of.H^gamma} and show that there is a unitary left multiplier from our model subspace, $\cH ^\ga$, onto an NC de Branges--Rovnyak space. We begin with the construction of an NC Schur class \emph{characteristic function} of any model triple for a CNC row partial isometry. Recall, as in \eqref{opNCuni}, we define 
$$ \scr{B} (\cJ, \cK ) _\N = \bigsqcup _{n=1} ^\infty \scr{B} (\cJ , \cK ) \otimes \C ^{n \times n}. $$

\begin{definition}\label{def:char.map}
Given a CNC row partial isometry $V : \cH \otimes \C ^d \rightarrow \cH$ with model triple $(\ga, \cJ _\infty, \cJ _0)$,  define the graded maps $D^\gamma : \bB^d_\bN \to \scr{B}(\cJ _0)_\bN$ and $N^\gamma : \bB^d_\bN \to \scr{B}(\cJ_\infty, \cJ_0)_\bN$ as
    \begin{align*}
        D^\gamma(Z) &:= \gamma(Z)^* [\gamma(0)\otimes I_n], \\
        N^\gamma(Z) &:= \gamma(Z)^*[I_\cH \otimes Z][\gamma(\infty) \otimes I_n].
    \end{align*}
    Then, the characteristic map $B^\gamma : \bB^d_\bN \to \scr{B}(\cJ_\infty, \cJ_0)_\bN$ is defined as
    \begin{equation}\label{eqn:def.char.map}
        B^\gamma(Z) := D^\gamma(Z)^{-1} N^\gamma(Z),
    \end{equation}
    and clearly satisfies $B^\ga(0) = 0$.
\end{definition}

Invertibility of $D^\gamma(Z)$ is not immediate and needs to be established in \eqref{eqn:def.char.map}. 

\begin{lemma}\label{lem:invertibility.of.denom}
    The graded function, $D^\gamma(Z)$, is invertible for each $Z \in \bB^d_n$ and $n \in \bN$.
\end{lemma}

\begin{proof}
    Let $n \in \bN$ and $Z \in \bB^d_n$ be fixed. It suffices to show that $D^\ga (Z)$ is bijective for each such $Z$. 
First, suppose there exists a $0 \neq \widehat{g} \in \nbker D^\gamma(Z)^*$. Since
    \begin{equation*}
    \gamma(Z) \widehat{g} \in \nbran \gamma(Z) = \cR(V-Z)^\perp    
    \end{equation*}
    and $\gamma(0)$ is an isometry onto $\nbran V^\perp$, we use Lemma \ref{lem:RanVperp.and.R(V-Z).are.iso} to conclude that there exists $0 \neq \widehat{f} \in \cH \otimes \bC^n$ such that $[\gamma(0) \otimes I_n] \widehat{f} \in \nbran V^\perp \otimes \bC^n$ and
    \begin{equation*}
        \gamma(Z) \widehat{g} = [I - VZ^*]^{-1}[\gamma(0) \otimes I_n] \widehat{f}.
    \end{equation*}
    Then, note that
    \begin{align}\label{eqn:inv.of.denom.calc.1}
        \begin{split}
            0 &= \bip{\widehat{f}, D^\gamma(Z)^*\widehat{g}}_{\cJ_0 \otimes \bC^n} \\
        &= \bip{ [\gamma(0) \otimes I_n] \widehat{f}, \gamma(Z)\widehat{g}}_{\cH \otimes \bC^n} \\
        &= \bip{[\gamma(0) \otimes I_n] \widehat{f}, [I- V Z^*]^{-1}[\gamma(0) \otimes I_n] \widehat{f}}_{\cH \otimes \bC^n} \\
        &= \sum_{k \in \bN} \bip{[\gamma(0) \otimes I_n] \widehat{f}, [V Z^*]^k [\gamma(0) \otimes I_n] \widehat{f}}_{\cH \otimes \bC^n}.
        \end{split}
    \end{align}
    Recall that $\widehat{f}$ is chosen so that $[\gamma(0) \otimes I_n] \widehat{f} \in \nbran V^\perp \otimes \bC^n$. However, for $k > 0$ we have
    \begin{equation*}
        [V Z^*]^k [\gamma(0) \otimes I_n] \widehat{f} \in \nbran V \otimes \bC^n.
    \end{equation*}
    Continuing from \eqref{eqn:inv.of.denom.calc.1}, we thus obtain
    \begin{equation*}
        0 = \bip{[\gamma(0) \otimes I_n] \widehat{f}, [\gamma(0) \otimes I_n] \widehat{f}}_{\cH \otimes \bC^n},
    \end{equation*}
    which is a contradiction since $\widehat{f} \neq 0$ and $\gamma(0)$ is an isometry. We therefore conclude that $\nbker D^\ga (Z) ^* = \{ 0 \}$, so that $D^\ga (Z)$ has dense range. (If $\nbdim \cJ _0 < +\infty$, this is sufficient to prove that $D^\ga (Z)$ is invertible.)
    
    We further claim that $D^\ga (Z)$ is bounded below so that it has closed range and is, thence, surjective. If not, then there is a sequence of unit vectors, $(g_n ) \subseteq \cJ _0 \otimes \C ^n$, $\| g _n \| _{\cJ _0 \otimes \C ^n} =1$, so that $\| D ^\ga (Z) g_n \| \rightarrow  0$. Setting $$G_n := [\ga (0) \otimes I_n] g_n,$$ we have that $\| G_n \| =1$ since $\ga (0)$ is an isometry. Using that $\ga (Z) : \cJ _0 \otimes \C ^n \twoheadrightarrow \scr{R} (V-Z) ^\perp$ and $[I- VZ^*] : \scr{R} (V-Z) ^\perp \rightarrow \nbran V ^\perp \otimes \C ^n$ are linear isomorphisms, if $H_n \in \scr{R} (V-Z) ^\perp$ and $h_n \in \cJ _0 \otimes \C ^n$ are chosen so that
    \begin{equation*}
        [I-VZ^*]H_n =G_n \qand \ga (Z) h_n = H_n,
    \end{equation*}
    then the norms of the sequences $(H_n)$, $(h_n)$ are uniformly bounded above and below by strictly positive constants. In particular, there exist $C,c>0$ so that $$\| h_n \| \leq C\qand \| H_n \| \geq c.$$ Then,
\ba C \| D ^\ga (Z) g_n \| & \geq & \ipr{H_n}{D^\ga (Z) g_n} =  \ipr{\ga (Z) h_n}{\ga (0) \otimes I_n g_n} =\ipr{H_n}{G_n} \\
& = & \ipr{H_n}{[I-VZ^*]H_n} = \ipr{H_n}{[I-ZV^*VZ^* - (V-Z)V^*VZ^*] H_n} \\
& = & \ipr{H_n}{[I-ZV^*VZ^*] H_n} \qquad \qquad \qquad \qquad (\because H_n \perp \scr{R} (V-Z).) \\
& \geq & \big(1 - \| Z \| _{\mr{row}} ^2 \big) \| H_n \| ^2 \geq \big(1- \| Z \| _{\mr{row}} ^2\big) c. \ea  
This is a contradiction, since the left hand side of this equation vanishes in the limit, and we conclude that $D^\ga (Z)$ is surjective. 

It remains to prove injectivity. Suppose now that there exists $0 \neq \widehat{f} \in \nbker D^\gamma(Z)$. Since $\gamma(0)$ is an isometry onto $\nbran V^\perp \otimes \bC^n$, it follows that
    \begin{equation*}
        0 \neq [\gamma(0) \otimes I_n] \widehat{f} \in (\nbran V^\perp \otimes \bC^n) \cap \nbker \gamma(Z)^* = (\nbran V^\perp \otimes \bC^n) \cap \nbran \gamma(Z)^\perp.
    \end{equation*}
    Now, we use Lemma \ref{lem:RanVperp.and.R(V-Z).are.iso} once again to obtain $0 \neq \widehat{g} \in \cR(V-Z)^\perp$ such that
    \begin{equation*}
        [\gamma(0) \otimes I_n] \widehat{f} = [I - VZ^*]\widehat{g}.
    \end{equation*}
    We therefore note for each $\widehat{h} \in \cJ_0 \otimes \bC^n$ that
    \begin{align*}
        0 &= \bip{\gamma(Z) \widehat{h}, [\gamma(0) \otimes I_n] \widehat{f}}_{\cH \otimes \bC^n} \\
        &= \bip{\gamma(Z) \widehat{h}, [I - VZ^*] \widehat{g}}_{\cH \otimes \bC^n} \\
        &= \bip{\gamma(Z) \widehat{h}, [I - (V Z^*)^* V Z^*]\widehat{g}}_{\cH \otimes \bC^n} - \underbrace{\bip{\gamma(Z) \widehat{h}, \overbrace{[(V-Z) V^* V] (Z^* \widehat{g})}^{\in \, \cR(V-Z)}}_{\cH \otimes \bC^n}}_{= \; 0, \text{ as } \nbran \gamma(Z) \, = \, \cR(V-Z)^\perp} \\
        &= \bip{\gamma(Z) \widehat{h}, [I - (V Z^*)^* VZ^*] \widehat{g}}_{\cH \otimes \bC^n}.
    \end{align*}
    Since $\nbran \gamma(Z) = \cR(V-Z)^\perp$, we can choose $\widehat{h}$ above so that $\gamma(Z) \widehat{h} = \widehat{g}$. The last equality above therefore leads to a contradiction as
    \begin{equation*}
        \bip{\widehat{g}, [I - (V Z^*)^* V Z^*]\widehat{g}}_{\cH \otimes \bC^n} \geq \big(1 - \|Z\|_{\text{row}}^2 \big) \|\widehat{g}\|^2 > 0.
    \end{equation*}
    Thus, $\nbker D^\ga (Z) = \{ 0 \}$ and the proof is complete.
\end{proof}

We are now ready to prove the main result of this subsection.

\begin{theorem}\label{thm:kernel.and.char.map}
    The graded reproducing kernel $K^\gamma$ of the graded RKHS $\cH ^\ga$ can be realized as
    \begin{align}
    \begin{aligned}\label{eqn:ker.in.terms.of.char.map}
        K^\gamma(Z,W)[A] &= D^\gamma(Z)\big[I_{\cJ_0} \otimes K(Z,W)[A]\big]D^\gamma(W)^* - N^\gamma(Z) \big[I_{\cJ_\infty} \otimes K(Z,W)[A]\big] N^\gamma(W)^* \\
        &= D^\gamma(Z)\Big\{\big[I_{\cJ_0} \otimes K(Z,W)[A]\big] - B^\gamma(Z)\big[I_{\cJ_\infty} \otimes K(Z,W)[A]\big] B^\gamma(W)^*\Big\}D^\gamma(W)^*,
    \end{aligned}
    \end{align}
    where $K$ denotes the NC Szeg\"o kernel of $\hardy$, and $Z \in \bB^d_n$, $W \in \bB^d_m$, $A \in \bC^{n \times m}$ and $n \in \bN$ are arbitrary. Consequently, $B^\ga$ is a contractive $\sB(\cJ_\infty, \cJ_0)$-valued graded map on $\bB^d_\bN$.
\end{theorem}

\begin{proof}
    We define the projections
    \begin{align}
        P_0 &:= P_{\nbran \widehat{V}^\perp} = \cU^\gamma \gamma(0)\gamma(0)^* {\cU^\gamma}^*,\label{eqn:def.proj.P_0}\\
        P_\infty &:= P_{\nbker\widehat{V}} = [\cU^\gamma \otimes I_d]\gamma(\infty)\gamma(\infty)^*[{\cU^\gamma}^* \otimes I_d].\label{eqn:def.proj.P_infty}
    \end{align}
    Since $\widehat{V}$ is a partial isometry, it follows that $\widehat{V} \widehat{V}^* = P_{\nbran \widehat{V}}$, and hence
    \begin{equation}\label{eqn:proof.ker.from.char.map.partial.iso.1}
        \widehat{V} \widehat{V}^* = I_{\cH^\gamma} - P_0.
    \end{equation}
    We intend to compute the following inner product of any two kernel functions in two ways:
    \begin{equation*}
        Q = \ip{K^\gamma\{Z,f \otimes y,v\}, \widehat{V}\widehat{V}^* K^\gamma\{W,g \otimes x,u\}}_{\cH^\gamma}
    \end{equation*}
    
    First, we use \eqref{eqn:proof.ker.from.char.map.partial.iso.1} and note that
\ba Q &= & \ipr{K^\gamma\{Z,f \otimes y,v\}}{ (I_{\cH^\gamma} - P_0) K^\gamma\{W,g \otimes x,u\}}_{\cH^\gamma}\\
    &= & \underbrace{\ipr{K^\gamma\{Z,f \otimes y,v\}}{K^\gamma\{W,g \otimes x,u\}}_{\cH^\gamma}}_{Q_1}   - \underbrace{\ipr{K^\gamma\{Z,f \otimes y,v\}}{P_0 K^\gamma\{W,g \otimes x,u\}}_{\cH^\gamma}}_{Q_2}. \ea
    Next, we apply Lemma \ref{lemma:action.M^R_Z^*.on.ker.func} to compute
    \ba
        Q &=& \ipr{\widehat{V}^* K^\ga \{Z, f \otimes y, v\}}{\widehat{V}^* K^\ga \{W, g \otimes x, u\}}_{\cH^\ga \otimes \bC^d} \\
        &=& \ipr{ K^\gamma \{Z, f \otimes y, \col(Z) v\}}{ (I_{\cH^\gamma \otimes \bC^d} - P_\infty) K^\gamma \{W, g \otimes x, \col(W) u\}}_{\cH^\gamma \otimes \bC^d}\\
        &=&\underbrace{\ipr{K^\gamma \{Z, f \otimes y, \col(Z) v\}}{K^\gamma \{W, g \otimes x, \col(W) u\}}_{\cH^\gamma \otimes \bC^d}}_{Q_3} \\
        & & - \underbrace{\ipr{ K^\gamma \{Z, f \otimes y, \col(Z) v\}}{ P_\infty K^\gamma \{W, g \otimes x, \col(W) u\} }_{\cH^\gamma \otimes \bC^d}}_{Q_4}.       
    \ea
    
    Let us simplify each of the $Q_j$'s. Applying \eqref{eqn:ker.at.Z.W} and the invertibility of $K(Z,W)$ yields
    \ba
        Q_1 - Q_3 &= & \bip{f \otimes y, K^\gamma(Z,W)[v  u^*](g \otimes x)}_{\cJ_0 \otimes \bC^n}
         - \sum_{j=1}^d\bip{f \otimes v, K^\gamma(Z,W)[Z_j(v u^*)W_j^*](g \otimes x)}\\
        &= & \bip{f \otimes y, K^\gamma(Z,W)\left[K^{-1}(Z,W)[v  u^*]\right](g \otimes x)}_{\cJ_0 \otimes \bC^n}.\ea
    Next, we use \eqref{eqn:def.K^gamma} and \eqref{eqn:def.proj.P_0} to get
    \begin{align*}
        Q_2 &= \bip{[\cU^\gamma \otimes v^*]\gamma(Z)(f \otimes y), \cU^\gamma \gamma(0)\gamma(0)^*{\cU^\gamma}^*[\cU^\gamma \otimes u^*]\gamma(W)(g \otimes x)}_{\cH^\gamma}\\
        &= \bip{f \otimes y, \gamma(Z)^* [I_\cH \otimes v^*]^*\gamma(0)\gamma(0)^*[I_\cH \otimes u^*]\gamma(W)(g \otimes x)}_{\cJ_0 \otimes \bC^n}\\
        &= \bip{f \otimes y, \gamma(Z)^* [\gamma(0)\otimes I_n][I_{\cJ_0} \otimes (v \otimes u^*)] [\gamma(0) \otimes I_n]^*\gamma(W)(g \otimes x)}_{\cJ_0 \otimes \bC^n}\\
        &= \bip{f \otimes y, D^\gamma(Z)[I_{\cJ_0} \otimes (v \otimes u^*)]D^\gamma(W)^*(g \otimes x)}_{\cJ_0 \otimes \bC^n}.
    \end{align*}
    Lastly, using \eqref{eqn:def.K^gamma} and \eqref{eqn:def.proj.P_infty}, we can compute

\begin{align*}
 Q_4 &= \ipr{\bsm
            [\cU^\gamma \otimes v^* Z_1^*]\gamma(Z)(f \otimes y) \vsm \\
            \vdots \vs \\
            [\cU^\gamma \otimes v^* Z_d^*]\gamma(Z)(f \otimes y) \esm}{
         \cU^\gamma \otimes I_d \, \gamma(\infty)\gamma(\infty)^* \, {\cU^\gamma}^* \otimes I_d \bsm
            [\cU^\gamma \otimes u^* W_1^*]\gamma(W)(g \otimes x) \vsm \\
            \vdots \vs \\
            [\cU^\gamma \otimes u^* W_d^*]\gamma(W)(g \otimes x)
        \esm}_{\cH ^\ga \otimes \C ^d} \\
        &= \ipr{\bsm
            [I_\cH \otimes v^* Z_1^*]\gamma(Z)(f \otimes y) \vsm \\
            \vdots \vs \\
            [I_\cH \otimes v^* Z_d^*]\gamma(Z)(f \otimes y) \esm}{\gamma(\infty)\gamma(\infty)^* \bsm
            [I_\cH \otimes u^* W_1^*]\gamma(W)(g \otimes x) \vsm \\
            \vdots \vs \\
            [I_\cH \otimes u^* W_d^*]\gamma(W)(g \otimes x)\esm}_{\cH \otimes \bC^d}\\
        &=\ipr{[I_{\cH \otimes \C^d} \otimes v^*] [I_\cH \otimes Z^*] \gamma(Z)(f \otimes y)}{\gamma(\infty) \gamma(\infty)^* [I_{\cH \otimes \bC^d} \otimes u^*] [I_\cH \otimes W^*] \gamma(W)(g \otimes x)}_{\cJ_\infty} \\
        &= \ipr{\gamma(\infty)^* [I_{\cH \otimes \bC^d} \otimes v^*][I_\cH \otimes Z]^*\gamma(Z)(f \otimes y)}{  \gamma(\infty)^* [I_{\cH \otimes \bC^d} \otimes u^*][I_\cH \otimes W]^*\gamma(W)(g \otimes x)}_{\cJ_{\infty}} \\
        &=\ipr{[I_{\cJ_{\infty}} \otimes v^*] [\gamma(\infty) ^* \otimes I_n] [I_\cH \otimes Z ^*] \gamma(Z)(f \otimes y)}{[I_{\cJ_\infty} \otimes u^*] [\gamma(\infty)^* \otimes I_n] [I_\cH \otimes W^*] \gamma(W)(g \otimes x) }_{\cJ_{\infty}} \\
        &=  \ipr{[I_{\cJ_\infty} \otimes v^*]N^\gamma(Z)^*(f \otimes y)}{[I_{\cJ_\infty} \otimes u^*]N^\gamma(W)^*(g \otimes x)}_{\cJ_\infty}\\
        &=  \ipr{f \otimes y}{N^\gamma(Z) [I_{\cJ_\infty} \otimes (v \otimes u^*)] {N^\gamma(W)}^*(g \otimes x)}_{\cJ_0 \otimes \bC^n}.
\end{align*}
    
    The above simplifications, together with the equation $Q_1 - Q_3 = Q_2 - Q_4$ yields \eqref{eqn:ker.in.terms.of.char.map} for all $A = K^{-1}(Z,W)[v \otimes u^*]$. However, since the NC Szeg\"o kernel, $K(Z,W)[\cdot]$ is an invertible linear map for all $Z \in \bB^d_n$ and $W \in \bB^d_m$ and since matrices of the form $v \otimes u^*$ span $\bC^{n \times m}$, we obtain \eqref{eqn:ker.in.terms.of.char.map} at once. That $B^\ga$ is contractive then follows easily by a standard argument in the theory of RKHS (see, for instance, \cite[Theorem 2.1]{Ball2001-lift}).
\end{proof}

We are now in a position to classify all graded models of a given CNC row partial isometry, $V$. We achieve this by establishing uniqueness of the corresponding characteristic functions. 

\begin{definition}\label{def:coincidence.class.of.char.map}
    Two maps $B_j : \bB^d_\bN \to \scr{B}(\cJ_\infty^{(j)}, \cJ_0^{(j)})_\bN$ are said to \emph{coincide unitarily} if there are fixed unitaries $U_0 \in \scr{B}\big(\cJ_0^{(1)}, \cJ_0^{(2)}\big)$ and $U_\infty \in \scr{B}\big(\cJ_\infty^{(1)}, \cJ_\infty^{(2)}\big)$ such that
    \begin{equation}\label{eqn:def.coincidence.class}
        [U_0 \otimes I_n] B_1(Z) = B_2(Z) [U_\infty \otimes I_n] \foral n \in \bN \AND Z \in \bB^d_\bN.
    \end{equation}
\end{definition}

This can be readily checked to be an equivalence relation and therefore we obtain a corresponding equivalence class for each $B : \bB^d_\bN \to \scr{B}(\cJ_\infty,\cJ_0)$, which shall be referred to as the \emph{unitary coincidence class} of $B$.

\begin{prop}\label{prop:uniqueness.of.graded.model}
    Let $\big(\gamma_j, \cJ_\infty^{(j)},\cJ_0^{(j)}\big)$ be two model triples for $V$. Then, the corresponding characteristic maps $B^{\gamma_j}$ must lie in the same coincidence class.
\end{prop}

\begin{proof}
    Define $U_0 = \gamma_2(0)^*\gamma_1(0)$ and $U_\infty = \gamma_2(\infty)^*\gamma_1(\infty)$, and note that they are unitaries since $\gamma_j(0)$ and $\gamma_j(\infty)$ are isometries onto $\nbran V^\perp$ and $\nbker V$ respectively for $j = 1,2$. We also introduce for notational convenience the linear isomorphism $C(Z) \in \scr{B}\big(\cJ_0^{(2)}, \cJ_0^{(1)}\big)$ for each $n \in \bN$ and $Z \in \bB^d_n$, given by
    \begin{equation*}
        C(Z) := \gamma_2(Z)^* (\gamma_1(Z) \gamma_1(Z)^*)^{-1} \gamma_1(Z).
    \end{equation*}
    Now, we compute
    \begin{align}
    \begin{aligned}\label{eqn:uniq.proof.step.1}
        [U_0 \otimes I_n] B^{\gamma_1}(Z) &= [U_0 \otimes I_n] D^{\gamma_1}(Z)^{-1}N^{\gamma_1}(Z) \\
        &= [U_0 \otimes I_n] D^{\gamma_1}(Z)^{-1} C(Z)^{-1} C(Z) N^{\gamma_1}(Z) \\
        &= \big(C(Z)D^{\gamma_1}(Z) [U_0^* \otimes I_n]\big)^{-1} C(Z) N^{\gamma_1}(Z).
    \end{aligned}
    \end{align}
    We then note that
    \begin{align}
    \begin{aligned}\label{eqn:uniq.proof.step.2}
        C(Z)D^{\gamma_1}(Z) &= \gamma_2(Z)^* (\gamma_1(Z) \gamma_1(Z)^*)^{-1} \gamma_1(Z) \gamma_1(Z)^*[\gamma_1(0) \otimes I_n] \\
        &= \gamma_2(Z)^* [\gamma_1(0) \otimes I_n] \\
        &= \gamma_2(Z)^*[\gamma_2(0) \otimes I_n] [(\gamma_2(0)^*\gamma_1(0)) \otimes I_n] \\
        &= D^{\gamma_2}(Z) [R \otimes I_n]
    \end{aligned}
    \end{align}
    and, similarly,
    \begin{align}
    \begin{aligned}\label{eqn:uniq.proof.step.3}
        C(Z)N^{\gamma_1}(Z) &= \gamma_2(Z)^* (\gamma_1(Z) \gamma_1(Z)^*)^{-1} \gamma_1(Z) \gamma_1(Z)^* [I_\cH \otimes Z][\gamma_1(\infty) \otimes I_n]\\
        &= \gamma_2(Z)^*[I_\cH \otimes Z][\gamma_2(\infty) \otimes I_n] [U_\infty \otimes I_n] \\
        &= N^{\gamma_2}(Z)[U_\infty \otimes I_n].
    \end{aligned}
    \end{align}
    Combining \eqref{eqn:uniq.proof.step.1}, \eqref{eqn:uniq.proof.step.2} and \eqref{eqn:uniq.proof.step.3} yields \eqref{eqn:def.coincidence.class}, which completes the proof.
\end{proof}

Recall from Remark \ref{rem:canon.model.triple} that every row partial isometry $V$ is equipped with the canonical NC model triple $(\Gamma, \nbker V, \nbran V^\perp)$, where
\begin{equation*}
    \Gamma(Z) = [I - VZ^*]^{-1}\vert_{\nbran V^\perp \otimes \bC^n} \foral Z \in \bB^d_n,\, n \in \bN.
\end{equation*}
It is then immediate from \eqref{eqn:Ga(Z)^*.is.NC} that $D^\Ga$, $N^\Ga$ and $B^\Ga$ are all operator-valued NC maps on $\bB^d_\bN$ (in the sense of \eqref{eqn:def.op-val.NC.map}). We can combine this knowledge of the standard NC model with Proposition \ref{prop:uniqueness.of.graded.model} to obtain the following remarkable fact:

\begin{theorem}\label{thm:char.map.is.NC.analytic.D.is.a.unitary.mult}
    If $(\ga, \cJ_\infty. \cJ_0)$ is any model triple for a given CNC row partial isometry $V$, then the characteristic function, $B^\ga$, is a strictly contractive and CE element of the NC Schur class $\sS_d(\cJ_\infty, \cJ_0)$ such that $B^\ga(0) = 0$.

    Moreover, $M^L_{D^\ga}$, i.e., left-multiplication by $D^\ga$, is a unitary map from $\sH(B^\ga)$ onto $\cH^\ga$, $X^\ga := {M^L_{D^\ga}}^* \cU^\ga V \big[ {\cU^\ga}^* M^L_{D^\ga} \otimes I_d \big]$ is the unique extremal Gleason solution for $\sH(B^\ga)$ and $\mbf{B}^\ga := \big[{M^L_{D^\ga}}^* \cU^\ga \otimes I_d\big] \ga(\infty)$ is the unique extremal Gleason solution for $B^\ga$.
\end{theorem}
\begin{remark} That is, even though we do not assume that our graded model, $(\ga, \cJ _\infty , \cJ _0)$ respects direct sums or joint similarities, this theorem proves that the characteristic function, $B^\ga$, coincides unitarily with an element of the NC Schur class, and hence, $B^\ga$, is also in the NC Schur class and is an NC operator--valued function in the unit row-ball.
\end{remark}
\begin{proof}
    As $B^\Ga$, the characteristic map corresponding to the canonical NC model triple, is an NC map, it follows from \eqref{eqn:def.coincidence.class} that $B^\ga$ is also an NC operator--valued function in $\rball$. It was already established in Theorem \ref{thm:kernel.and.char.map} that $B^\ga$ is a contractive map, and therefore, $B^\ga \in \sS_d(\cJ_\infty, \cJ_0)$ follows at once. That $B^\ga(0) = 0$ is also immediate since $N^\ga(0) = 0$ by definition. Lastly, $D^\ga$ being a unitary left multiplier from $\sH(B^\ga)$ onto $\cH^\ga$ follows from \eqref{eqn:ker.in.terms.of.char.map} and a simple modification of \cite[Theorem 2.1]{Ball2001-lift} from graded kernels.

    Next, we check that $X^\ga$, as in the hypothesis, is the unique extremal Gleason solution for $\sH(B^\ga)$. Let us suspend the superscript $\ga$ temporarily to simplify the notation below. Recall from Proposition \ref{prop:V.transf.under.U^gamma} that $\widehat{V} := \cU V [{\cU}^* \otimes I_d]$ acts as $M^R_Z$ on its initial space and
    \begin{equation*}
        \nbran \widehat{V} = \{ h \in \cH \ | \ h(0) = 0 \}.
    \end{equation*}
    Since $D$ is a unitary left multiplier, it follows that $X$ also acts as $M^R_Z$ on its initial space and
    \begin{equation*}
        \nbran X = \{ f \in \sH(B) \ | \ f(0) = 0 \}.
    \end{equation*}
    Consequently, we have
    \begin{equation*}
        X X^* = I_{\sH(B)} - K_0^B {K_0^B}^*
    \end{equation*}
    and, for each $f \in \sH(B)$, $Z \in \bB^d_n$ and $n \in \bN$, we can compute
    \begin{align*}
        M^R_Z (X^* f)(Z) = (X X^* f)(Z) = f(Z) - \big(K_0^B {K_0^B}^*f \big)(Z) = f(Z) - K^B \{0, f(0) \otimes 1, 1\}(Z).
    \end{align*}
    As $B(0) = 0$, it is easy to check that $K^B\{0, f(0) \otimes 1, 1\}(Z) = f(0_n)$. Thus, $X$ is the unique extremal Gleason solution for $\sH(B)$ as claimed.

    It remains to verify that $\mbf B = \bsm B_1 \\ \vdots \\ B_d \esm$, as defined in the hypothesis, is the unique and extremal Gleason solution for $B$. First, note that
    \begin{equation*}
        \mbf B^* \mbf B = \ga(\infty)^* \ga(\infty) = I_{\cJ_\infty}.
    \end{equation*}
    As $B(0) =0$, if $\mbf{B}$ is a Gleason solution for $B$, then the preceding equation implies that it is contractive and extremal. For each $g \in \cJ_0$, $h \in \cJ_\infty$, $Z \in \bB^d_n$, $x, u \in \bC^n$ and $n \in \bN$, we have
    \begin{align*}
        \lel g \otimes x, \sum_{j = 1}^d [B_j h (Z)](Z_j u) \rir_{\cJ_0 \otimes \bC^n} &= \sum_{j = 1}^d \bip{K^B\{Z, g \otimes x, Z_j u\}, B_j h}_{\sH(B)} \\
        &= \bip{K^B \{Z, g \otimes x, \col(Z) u\}, \mbf B h}_{\sH(B) \otimes \bC^d} \\
        &= \lel \bsm \cU^* M^L_D K^B \{Z, g \otimes x, Z_1 u\} \\ \vdots \\ \cU^* M^L_D K^B \{Z, g \otimes x, Z_d u\} \esm, \ga(\infty) h \rir_{\cH \otimes \bC^d}.
    \end{align*}
    It is then straightforward to check that 
    \begin{align} \label{eqn:left.mult.action.on.kernel.maps}
    M^L_D K^B\{ Z, g \otimes x, Z_j u \} & = (M^L _{D^{-1}} ) ^* K^B \{ Z, g \otimes x, Z_j u \} \nonumber\\
    & = K^\ga \{ Z, D (Z) ^{-*} g \otimes x, Z_j u\}.
    \end{align}

 As a result, we continue the above calculation and obtain
    \begin{align*}
        \lel g \otimes x, \sum_{j = 1}^d [B_j h (Z)](Z_j u) \rir_{\cJ_0 \otimes \bC^n} &= \lel \bsm \cU^* M^L_D K^B \{Z, g \otimes x, Z_1 u\} \\ \vdots \\ \cU^* M^L_D K^B \{Z, g \otimes x, Z_d u\} \esm, \ga(\infty) h \rir_{\cH \otimes \bC^d} \\
        &= \lel \bsm \cU^* K^\ga \{Z, D(Z)^{-*}\, g \otimes x, Z_1 u\} \\ \vdots \\ \cU^* K^B \{Z, (D(Z)^{-1})^*(g \otimes x), Z_d u\} \esm, \ga(\infty) h \rir_{\cH \otimes \bC^d} \\
        &= \lel \bsm [I_\cH \otimes u^* Z_1^*] \ga(Z)^* D(Z)^{-*}\, g \otimes x \\ \vdots \\ [I_\cH \otimes u^* Z_d^*] \ga(Z)^* D(Z)^{-*} \, g \otimes x \esm, \ga(\infty) h \rir_{\cH \otimes \bC^d} \\
        &= \ip{[I_\cH \otimes u^*][I_\cH \otimes Z]^* \ga(Z)^* D(Z)^{-*} \, g \otimes x, \ga(\infty) h}_{\cH \otimes \bC^d} \\
        &= \bip{g \otimes x, D(Z)^{-1} \ga(Z) [I_\cH \otimes Z] [\ga(\infty) \otimes I_n](h \otimes u)}_{\cJ_0 \otimes \bC^n} \\
        &= \bip{g \otimes x, B(Z)(h \otimes u)}_{\cJ_0 \otimes \bC^n}.
    \end{align*}
    Since $B(0) = 0$, it follows that $\mbf B$ is the (necessarily) unique Gleason solution for $B$. Moreover, since $\mbf{B}$ is extremal, it follows that $B$ is column--extreme, by \cite[Theorem 6.4]{JM-freeCE}.
\end{proof}

Computing the characteristic function with respect to the canonical NC model for a given row partial isometry and combining Proposition \ref{prop:uniqueness.of.graded.model} with Theorem \ref{thm:char.map.is.NC.analytic.D.is.a.unitary.mult} yields our main model for CNC row partial isometries.

\begin{theorem}[\bf NC de Branges--Rovnyak model for CNC row partial isometries]\label{thm:main.model.row.partial.iso}
    A row partial isometry $V : \cH \otimes \bC^d \to \cH$ is CNC if and only if it is unitarily equivalent to the unique extremal Gleason solution for an NC de Branges--Rovnyak space $\sH(B_V)$, where $B_V \in \sS_d(\nbker V, \nbran V^\perp)$ is CE and satisfies $B(0) = 0$. This NC Schur class function, $B_V$, is unique up to unitary coincidence, and is given by
    \begin{eqnarray} \label{eqn:NC.char.map.V}
      B_V(Z) & = &  D_V(Z)^{-1} N_V(Z); \quad \quad Z \in \B ^d _n, \ n \in \N, \\
  \mbox{where,} \quad      D_V (Z) &= & [I - VV^*] [I - ZV^*]^{-1} [I - VV^*] \otimes I_n \quad \mbox{and} \nonumber \\
        N_V(Z) &= & [I - VV^*] [I - ZV^*]^{-1} [I_\cH \otimes Z] \, [I - V^*V]  \otimes I_n. \nonumber \end{eqnarray}
\end{theorem}

\subsection{Model for the partial isometric part of Gleason solutions}\label{subsec:model.constr.Gleason.soln}

Recall that the unique contractive Gleason solution, $X := R^* \otimes I_{\cK}| _{\scr{H} (B)}$, for a column--extreme NC Schur class function, $B \in \sS _d (\cJ, \cK)$ is extremal and CNC. Hence, $X$ has the isometric--pure decomposition, $X = V+C$, where $V$ is a CNC row partial isometry. In this subsection, we calculate the characteristic function, $B_V$, of $V$, and determine the relationship between $B_V$ and the original $B$. This will motivate the extension of our model from CNC row partial isometries to arbitrary CNC row contractions.  

Further recall from Section \ref{subsec:Gleason.solutions}, that if $B  \in \sS_d(\cJ, \cK)$ is CE, then $X^* := R^* \otimes I_{\cK} \vert_{\scr{H}(B)}$ is the unique and extremal CNC Gleason solution $X$ for $\scr{H}(B)$. It was noted that $\Gamma(0) := K_0^B D^{-1}_{B(0)^*} : \cJ_0 \twoheadrightarrow \nbran V^\perp$ extends by continuity to an onto isometry, where $V := X P_{\nbker D_X}$ is the row partial isometric part of $X$ (as in Lemma \ref{isopure}). If we define, for each $Z \in \bB^d_n$ and $n \in \bN$, the map
$$ \Gamma(Z) := [I - XZ^*]^{-1} [\Gamma(0) \otimes I_n] : \cK \otimes \bC^n \twoheadrightarrow \cR(V - Z)^\perp, $$
then $\Gamma$ is a standard NC model map. To complete the model triple, we need to identify a map $\Gamma(\infty)$ that is an isometry onto $\nbker V$. We will construct this map using the unique, contractive and extremal Gleason solution, $\mbf{B}$, for $B$. The following relationship will be useful in our calculations,
\begin{equation}\label{eqn:b^*.acting.on.M^R_W^*}
    \mbf{B}^* K^B\{W, g \otimes x, \col(W) u\} = [I_{\cJ } \otimes u^*]\big(B(W)^* - B(0_n)^*\big) (g \otimes x).
\end{equation}
This holds because, for each $h \in \cJ $, we can use \eqref{GSforB} to calculate
\begin{align*}
    \bip{\mbf{B}^* K^B \{W, g \otimes x, \col(W) u\}, h}_{\cJ} &= \bip{K^B\{W, g \otimes x, \col(W) u\}, \mbf{B} h}_{\sH(B) \otimes \bC^d} \\
    &= \sum_{j = 1}^d \bip{K^B\{W, g \otimes x, W_j u\}, B_j h}_{\sH(B)} \\
    &= \sum_{j = 1}^d \bip{g \otimes x, B_j h (W) W_j u}_{\cJ \otimes \bC^n} \\
    &= \lel g \otimes x, \big(B(W) - B(0_n)\big) (h \otimes u) \rir_{\cJ \otimes \bC^n} \\
    &= \bip{[I_{\cJ} \otimes u^*] \big(B(W)^* - B(0_n)^*\big)(g \otimes x), h}_{\cJ}.
\end{align*}

\begin{lemma}\label{lemma:gamma_infty.isometry}
    The map $\check\Gamma(\infty) := \mbf{B} D_{B(0)}^{-1} : \cJ_\infty \to \sH(B) \otimes \bC^d$ extends by continuity to an isometry into $\nbker V$.
\end{lemma}
In the above statement, if $B(0)$ is a strict contraction, $D_{B(0)}$ is invertible and $\check\Ga (\infty)$ is defined on all of $\cJ $. If, however, $B(0)$ is only a pure contraction, $\check\Ga (\infty)$ is defined, intially, only on the dense linear subspace $\scr{D} _{B(0)} = \nbran D _{B(0)}$. However, as the proof below shows, $\check\Ga (\infty)$ is isometric, and hence extends by continuity to an isometry on $\cJ _\infty$ regardless of whether or not $B(0)$ is a strict or pure contraction.  
\begin{proof}
    It is straightforward to check using that $\mbf{B}$ is extremal, $\emph{i.e.}$ $\mbf{B}^* \mbf{B} = I _\cJ - B(0) ^* B(0)$, that $\check \Gamma(\infty)$ is an isometry:
    \begin{equation*}
        \check \Gamma(\infty)^* \check \Gamma(\infty) = D_{B(0)}^{-1} \mbf{B}^* \mbf{B} D_{B(0)}^{-1} = I_{\cJ}.
    \end{equation*}
    Thus, it suffices to show that $\nbran \check \Gamma(\infty) \subseteq \nbker V = \nbran {V^*}^\perp$. In particular, we will show that for any kernel function $K^B\{W, g \otimes x, u\}$, we have
    \begin{equation}\label{eqn:gamma_infty.lemma.to.show}
        \check \Gamma(\infty)^* V^* K^B\{W, g \otimes x, u\} = 0.
    \end{equation}

    To this end, we use \eqref{eqn:partial.iso.part,of.Gleason.soln} and observe that
    \begin{eqnarray} \label{eqn:gamma_infty.lemma.proof.1}
        &  & D_{B(0)} \check \Gamma(\infty)^* V^* K^B\{W, g \otimes x, u\} = \mbf{B}^* V^* K^B\{W, g \otimes x, u\} \nonumber\\
        &= & \mbf{B}^* X^* \left[I_{\sH(B)} - K_0^B D_{B(0)^*}^{-2} {K_0^B}^*\right] K^B\{W, g \otimes x, u\} \nonumber\\
        &= & \mbf{B}^* X^* K^B\{W, g \otimes x, u\} 
         - \mbf{B}^* X^* K_0^B \underbrace{K^B(0,0)^{-1} {K_0^B}^* K^B\{W, g \otimes x, u\}}_{=:g_0 \in \cJ_0} \nonumber\\
        &= & \mbf{B}^* X^* K^B\{W, g \otimes x, u\} - \mbf{B}^* X^* K^B\{0, g_0 \otimes 1, 1\}.
    \end{eqnarray}
    Let us calculate both terms separately using Proposition \ref{GSkernel} and \eqref{eqn:b^*.acting.on.M^R_W^*}. First, we have
    \begin{align}\label{eqn:gamma_infty.lemma.proof.2}
        \mbf{B}^* X^* K^B\{W, g \otimes x, u\} &= \Big[ [I_{\cJ} \otimes u^*] \big(B(W)^* - B(0_n)^*\big)  - (I_{\cJ} - B(0)^* B(0))[I_{\cJ} \otimes u^*] B(W)^* \Big] (g \otimes x) \nonumber\\
        &= - B(0)^*\big([I_{\cK} \otimes u^*] - B(0)[I_{\cJ} \otimes u^*] B(W)^*\big)(g \otimes x) \nonumber\\
        &= - B(0)^* K^B(0,W)[u^*] (g \otimes x).
    \end{align}
    Similarly, substituting the value of $g_0$ we get
    \begin{align}\label{eqn:gamma_infty.lemma.proof.3}
        \mbf{B}^* X^* K^B\{0, g_0 \otimes 1, 1\} &= - B(0)^* K^B(0,0) g_0 \nonumber\\
        &= - B(0)^* {K_0^B}^* K^B\{ W, g \otimes x, u\} \nonumber\\
        &= - B(0)^* K^B(0, W)[u^*](g \otimes x).
    \end{align}
    The last equality holds because, for each $g' \in \cK$, we can calculate
    \begin{align}
        \bip{{K^B_0}^* K^B\{W, g \otimes x, u\}, g'}_{\cK} &= \bip{K^B\{W, g \otimes x, u\}, K^B\{0, g' \otimes 1, 1\}}_{\sH(B)} \nonumber\\
        &= \bip{g \otimes x, K^B(W,0)[u] g'}_{\cK \otimes \bC^n} \nonumber\\
        &= \bip{K^B(0, W)[u^*](g \otimes x), g'}_{\cK}. \label{eqn:Kernel.at.0.and.W}
    \end{align}

    Plugging the calculations from \eqref{eqn:gamma_infty.lemma.proof.2} and \eqref{eqn:gamma_infty.lemma.proof.3} into \eqref{eqn:gamma_infty.lemma.proof.1}, and noting that $D_{B(0)}$ is injective with dense range immediately yields \eqref{eqn:gamma_infty.lemma.to.show}. This completes the proof.
\end{proof}

Note that $\nbran \check \Gamma(\infty)$ may not be equal to $\nbker V$ as we require, but this can be easily fixed by appending an arbitrary isometry $\Gamma'(\infty) : \cJ' \to \nbker V \ominus \nbran \check \Gamma(\infty)$ as a direct summand,
\begin{equation}\label{eqn:gamma.prime.infty}
    \Ga(\infty) := \check\Ga(\infty) \oplus \Ga'(\infty) : \cJ \oplus \cJ' \twoheadrightarrow \nbker V.
\end{equation}
We therefore combine the observations from Section \ref{subsec:Gleason.solutions} and Lemma \ref{lemma:gamma_infty.isometry} to conclude the following.

\begin{corollary}\label{cor:model.triple.for.Gleason.soln}
    Let $B \in \sS_d(\cJ, \cK)$ be a purely contractive and CE (left) multiplier. Then, $(\Ga, \cJ \oplus \cJ', \cK)$ is a model triple for the row partial isometric part $V$ of the unique extremal Gleason solution $X$ of $\sH(B)$, where
    \begin{equation*}
        \Ga :=
        \begin{cases}
            \ 0 &\mapsto \quad \Gamma(0) := K_0^B D_{B(0)^*}^{-1}; \\
            \ Z &\mapsto \quad [I - XZ^*]^{-1} [\Ga(0) \otimes I_n], \, \forall Z \in \bB^d_\bN; \\
            \ \infty &\mapsto \quad \check \Ga (\infty) \oplus \Ga'(\infty),
        \end{cases}
    \end{equation*}
$\check \Ga(\infty) = \mbf{B} D_{B(0)}^{-1}$ and $\Ga'(\infty)$ is as in \eqref{eqn:gamma.prime.infty}.
\end{corollary}

As defined in Section \ref{subsec:char.map}, we now compute the characteristic map $B^\Ga$ for this model triple. Recall that $B^\Ga(Z) = D^\Ga(Z)^{-1} N^\Ga(Z)$, where, for each $Z \in \bB^d_n$ and $n \in \bN$,
\begin{align*}
    D^\Ga(Z) &:= \Ga(Z)^* [\Ga(0) \otimes I_n]; \\
    N^\Ga(Z) &:= \Ga(Z)^* [I_{\sH(B)} \otimes Z][\Ga(\infty) \otimes I_n].
\end{align*}

\begin{theorem}\label{thm:char.map.gleason.soln}
    Let $B \in \cS_d(\cJ, \cK)$ and $(\Ga, \cJ \oplus \cJ', \cK)$ be as in Corollary \ref{cor:model.triple.for.Gleason.soln}. Then, the characteristic map $B^\Ga$ corresponding to this model triple is given by
    \begin{equation}\label{eqn:char.map.gleason.sol}
        B^\Ga(Z) = D_{B(0_n)^*} \big[I_{\cK \otimes \bC^n} - B(Z) B(0_n)^*\big]^{-1} \big(B(Z) - B(0_n)\big) D_{B(0_n)}^{-1}  \oplus [\mathbf{0}_{\cJ', \cK} \otimes I_n].
    \end{equation} 
\end{theorem}
In the above, $\mathbf{0}_{\cJ', \cK}$ represents the zero-operator between the respective spaces.
\begin{proof}
    Let $Z \in \bB^d_n$ and $n \in \bN$ be fixed but arbitrary. It follows from the definition that
    \ba
        D^\Ga(Z) &= & \Ga(Z)^* [\Ga(0) \otimes I_n]  \\
        &= & \big[D_{B(0)^*}^{-1} {K^B_0}^* \otimes I_n\big] [I - Z X^*]^{-1} \big[K^B_0 D_{B(0)^*}^{-1} \otimes I_n\big] \\
        &= & D_{B(0_n)^*}^{-1} \big[{K^B_0}^* \otimes I_n\big] [I - Z X^*]^{-1} \big[K^B_0 \otimes I_n\big] D_{B(0_n)^*}^{-1}.
    \ea
    
    For each $g, h \in \cK$ and $x, u \in \bC^n$ we can use \eqref{eqn:ker.action.model} to obtain
    \ba
        \Bip{g \otimes x, {K^B_{0_n}}^* [I - Z X^*]^{-1} K^B _{0_n} (h \otimes u)}_{\cK \otimes \bC^n} 
        &=&  \bip{[I - X Z^*]^{-1} \big[K^B_0 g \otimes x\big], K^B_0 h \otimes u}_{\sH(B) \otimes \bC^n} \\
        & = &  \bip{[I_{\sH(B)} \otimes u^*] [I - X Z^*]^{-1} \big[K^B_0 g \otimes x\big], K^B_0 h}_{\sH(B)} \\
        & = & \bip{K^B\{Z, g \otimes x, u\}, K^B\{0, h \otimes 1, 1\}}_{\sH(B)} \\
        & = &  \bip{g \otimes x, \big(K^B(Z,0)[I_{\sH(B)} \otimes u]\big) (h \otimes 1)}_{\cK \otimes \bC^n} \\
        & = & \bip{g \otimes x, h \otimes u - B(Z) [I_{\sH(B)} \otimes u] B(0)^* h}_{\cK \otimes \bC^n} \\
        & = & \bip{g \otimes x, \big[I_{\cK \otimes \bC_n} - B(Z)B(0_n)^*\big](h \otimes u)}_{\cK \otimes \bC^n}.
    \ea
    In conclusion,
    \begin{equation}\label{eqn:char.map.gleason.sol.proof.1}
        D^\Ga(Z) = D_{B(0_n )^*}^{-1}  \big[I_{\cK \otimes \bC^n} - B(Z)B(0_n)^*\big] D_{B(0_n )^*}^{-1}.
    \end{equation}
    
    Next, we compute $N^\Ga(Z)$ by splitting it into two pieces:
    \begin{equation*}
        N^\Ga(Z) = \underbrace{\Ga(Z)^* [I_{\sH(B)} \otimes Z] [\check \Ga(\infty) \otimes I_n]}_{N^{\check \Ga}(Z)} \oplus \underbrace{\Ga(Z)^* [I_{\sH(B)} \otimes Z] [\Ga'(\infty) \otimes I_n]}_{N^{\Ga'}(Z)}
    \end{equation*}
    Let us first show that $N^{\Ga'}(Z) = [\mathbf{0}_{\cJ', \cK} \otimes I_n]$, i.e., $\nbran \Ga'(\infty) \otimes \bC^n \subset \nbker \Ga(Z)^* [I_{\sH(B)} \otimes Z]$. To this end, let $\vec{f} \in \nbran \Ga'(\infty)$ and $u \in \bC^n$ be arbitrary and note that
    \begin{equation*}
        \vec{f} \in \nbker V \cap \nbran \check \Ga(\infty)^\perp = \nbker V \cap \nbker \mbf{B}^*.
    \end{equation*}
    Next, note that
    \ba
        \Ga(Z)^* [I_{\sH(B)} \otimes Z] (\vec{f} \otimes u) 
        & = & [\Ga(0)^* \otimes I_n][I - Z X^*]^{-1} [I_{\sH(B)} \otimes Z] (\vec{f} \otimes u) \\
        &  = & D_{B(0_n)^*}^{-1} \big[{K^B_0}^* \otimes I_n\big] [I - Z X^*]^{-1} \left(\sum_{j = 1}^d f_j \otimes Z_j u \right).
    \ea
    We then fix $g \in \cK$ and $x \in \bC^n$ arbitrary and use \eqref{eqn:ker.action.model} and \eqref{eqn:X^*.b.intertwine.ker.action} to obtain
    \begin{align*}
        &\lel g \otimes x, \big[{K^B_0}^* \otimes I_n\big] [I - X^* Z]^{-1} \left( \sum_{j = 1}^d f_j \otimes Z_j u \right) \rir_{\cK \otimes \bC^n} \\
        &\qquad\qquad\qquad\qquad = \sum_{j = 1}^d \bip{K^B\{Z, g \otimes x, Z_j u\}, f_j}_{\sH(B)} \\
        &\qquad\qquad\qquad\qquad = \bip{K^B\{Z, g \otimes x, \col(Z) u\}, \vec{f}}_{\cH(B) \otimes \bC^d} \\
        &\qquad\qquad\qquad\qquad = \bip{X^* K^B\{Z, g \otimes x, u\}, \vec{f}}_{\sH(B) \otimes \bC^d} + \underbrace{\bip{\mbf B [I_{\sH(B)} \otimes u^*] B(Z)^*(g \otimes x), \vec{f}}_{\sH(B) \otimes \bC^d}}_{= \, 0, \text{ since } \vec{f} \, \in \, \nbker \mbf{B}^*} \\
        &\qquad\qquad\qquad\qquad = \bip{K^B\{Z, g \otimes x, u\}, X \vec{f}}_{\sH(B)}.
    \end{align*}
    Recall from \eqref{eqn:partial.iso.part,of.Gleason.soln} that $V = X - K^B_0 D_{B(0)^*}^{-2} {K^B_0}^* X$. Thus, since $\vec{f} \in \nbker V$ by choice, we get
    \begin{equation*}
        X \vec{f} = K^B_0 D_{B(0)^*}^{-2} {K^B_0}^* X \vec{f}.
    \end{equation*}
    By taking $W = 0$ and $x = u = 1$ in \eqref{eqn:X^*.b.intertwine.ker.action}, it is straightforward to verify
    \begin{equation}\label{eqn:kernel.at.0.and.X*}
        X^* K^B_0 = - \mbf{B} B(0)^*,
    \end{equation}
    which further implies that
    \begin{equation*}
        {K^B_0}^* X \vec{f} = - B(0) \mbf{B}^* \vec{f} = 0,
    \end{equation*}
    since $\vec{f} \in \nbker \mbf{B}^*$ by choice. Therefore, we get $N^{\Ga'}(Z) = [\mathbf{0}_{\cJ', \cK} \otimes I_n]$ as claimed.

    Lastly, note that
    \begin{align*}
        N^{\check \Ga}(Z) &= [\Ga(0)^* \otimes I_n][I - Z X^*]^{-1} [I_{\sH(B)} \otimes Z] [\check \Ga(\infty) \otimes I_n] \\
        &= D_{B(0_n)^*}^{-1} \big[{K^B_0}^* \otimes I_n\big][I - Z X^*]^{-1}[I_{\sH(B)} \otimes Z][\mbf{B} \otimes I_n]D_{B(0_n)}^{-1}.
    \end{align*}
    Then, for each $h \in \cJ_\infty$ and $u \in \bC^n$, we have
    \begin{align*}
        [I_{\sH(B)} \otimes Z][\mbf B \otimes I_n](h \otimes u) &= [I_{\sH(B)} \otimes Z]
        \begin{bmatrix}
            B_1 h \otimes u \\
            \vdots \\
            B_d h \otimes u
        \end{bmatrix}\\
        &= \sum_{j = 1}^d B_jh \otimes Z_j u.
    \end{align*}
    We can therefore compute for each $g \in \cK$ and $x \in \bC^n$ using \eqref{eqn:ker.action.model} and \eqref{ncGS} that
    \begin{align}\label{eqn:char.map.gleason.sol.proof.3}
        &\bip{g \otimes x, \big[{K^B_0}^* \otimes I_n\big] [I - Z X^*]^{-1} [I_{\sH(B)} \otimes Z][\mbf{B} \otimes I_n](h \otimes u)}_{\cK \otimes \bC^n} \nonumber\\
        &\qquad\qquad\qquad\qquad = \sum_{j = 1}^d \bip{[I - X Z^*]^{-1} \big[K^B_0 \otimes I_n\big](g \otimes x), B_j h \otimes Z_j u}_{\sH(B) \otimes \bC^n} \nonumber\\
        &\qquad\qquad\qquad\qquad = \sum_{j = 1}^d \bip{K^B\{Z, g \otimes x, Z_j u\}, B_j h}_{\sH(B)} \nonumber\\
        &\qquad\qquad\qquad\qquad = \sum_{j = 1}^d \bip{g \otimes x, B_j h(Z)Z_j u}_{\cK \otimes \bC^n} \nonumber\\
        &\qquad\qquad\qquad\qquad = \bip{g \otimes x, \big(B(Z) - B(0_n)\big)(f \otimes u)}_{\cK \otimes \bC^n}.
    \end{align}
    Putting the calculation of $N^{\check \Ga}$ and $N^{\Ga'}$ together, we obtain
    \begin{equation}\label{eqn:char.map.gleason.sol.proof.2}
        N^\Ga(Z) = D_{B(0_n)^*}^{-1} \big(B(Z) - B(0_n)\big) D_{B(0_n)}^{-1} \oplus [\mathbf{0}_{\cJ', \cK} \otimes I_n].
    \end{equation}
    Combining \eqref{eqn:char.map.gleason.sol.proof.1} and \eqref{eqn:char.map.gleason.sol.proof.2} immediately yields \eqref{eqn:char.map.gleason.sol}, as required.
\end{proof}

Recall from Remark \ref{rem:X.partial.iso.iff.b(0)=0} that $X$ is a row partial isometry if and only if $B(0) = 0$. Adding this assumption to the above theorem and using \eqref{eqn:ker.in.terms.of.char.map} immediately reveals the following.

\begin{corollary}\label{cor:char.map.gleason.sol.B(0)=0}
    If, in addition to the hypothesis of Theorem \ref{thm:char.map.gleason.soln}, we assume $B(0) = 0$, then
    \begin{equation*}
        B^\Ga(Z) = B(Z) \oplus [\mbf{0}_{\cJ', \cK} \otimes I_n]. \label{suppdecomp}
    \end{equation*}
    Consequently, $K^\Ga = K^B$ and thus $\cH^\Ga = \sH(B)$.
\end{corollary}

The above theorem shows that $B^\Ga$ is similar to a \emph{M\"obius transformation} of $B$ so that $B^\Ga(0) = 0$. In the sequel, we use this knowledge and construct for a given CNC row contraction a corresponding characteristic map by applying the inverse of such a M\"obius transformation of the characteristic map to its row partial isometric part.

\subsection{Weak coincidence}\label{subsec:weak.coincidence}
We introduced the notion of the unitary coincidence class for an NC Schur function $B$ in the introduction. In what follows, it will be convenient to work with a weaker notion of isomorphism. First, recall that for a given $B \in \sS_d(\cJ_\infty, \cJ_0)$, its support is defined as
\begin{equation*}
    \supp(B) = \bigvee_{\substack{y, v \in \bC^n \\ Z \in \bB^d_n, \, n \in \bN}} \nbran [I_{\cJ_\infty} \otimes y^*] B(Z)^* [I_{\cJ_{0}} \otimes v] \subseteq \cJ_\infty.
\end{equation*}
By construction, $\cJ _\infty = \mr{supp} (B) \oplus \mr{supp} (B) ^\perp$, and $B(Z) g =0$ for any $g \in \supp(B) ^\perp$. That is, as in \eqref{suppdecomp} above, $B$ has the direct sum decomposition,
$$ B = B | _{\supp (B)} \oplus \mbf{0} _{\supp (B) ^\perp, \supp (B^*) ^\perp}, $$
where $\supp(B^*)$ is as defined in the introduction.

\begin{definition}\label{def:weak.coincidence}
    Two operator--valued NC Schur class functions, $B\in \sS _d (\cJ, \cK)$ and $B' \in \sS _d (\cJ', \cK')$ are said to \emph{coincide weakly} if $B \vert_{\supp(B)}$ and $B' \vert_{\supp(B')}$ coincide (as in \eqref{eqn:def.coincidence.class}).
\end{definition}

The following lemma provides an equivalent formulation of weak coincidence and also shows that it is an equivalence relation on NC Schur functions.

\begin{lemma}\label{lem:weak.coincidence.equiv}
Two operator--valued NC Schur class functions, $B\in \sS _d (\cJ, \cK)$ and $B' \in \sS _d (\cJ', \cK')$ coincide weakly if and only if there is a constant unitary $U: \cK \twoheadrightarrow \cK '$ that defines a unitary left multiplier of $\scr{H} (B)$ onto $\scr{H} (B')$. 
\end{lemma}
\begin{proof}
If $B = B_0 \oplus \mbf{0}$ and $B' = B' _0 \oplus \mbf{0}$ are the support decompositions of $B$ and $B'$ on $\cJ = \cJ_0 \oplus \cJ _0 ^\perp$ and $\cJ' = \cJ ' _0 \oplus \cJ _0 ^{' \perp}$ then by assumption there exist unitaries $U: \cK \twoheadrightarrow \cK'$ and $T: \cJ _0 \twoheadrightarrow \cJ' _0$ so that $U B_0 = B' _0 T$. Hence, 
$$ I_{\hardy} \otimes I_{\cK'} - B' (L) B' (L) ^*  =  \big[I_{\hardy} \otimes U\big] \big[I _{\hardy} \otimes I_\cK - B(L) B(L) ^* \big] \big[I_{\hardy} \otimes U^*\big], $$ so that 
\ba \scr{H} (B') & = & \nbran \sqrt{ I_{\hardy} \otimes I_{\cK'} - B' (L) B'(L) ^*} \\
& = & \big[I _{\hardy} \otimes U\big] \,  \nbran \sqrt{I_{\hardy} \otimes I_{\cK} - B (L) B(L) ^*} \\
& = & \big[I_{\hardy} \otimes U\big] \, \scr{H} (B). \ea  The converse is readily verified.
\end{proof}

\begin{remark}\label{rem:weak.coincidence.vs.coincidence}
    It is clear that if $B$ and $B'$ coincide then they also coincide weakly. However, the converse is not true. To see this, let $B(Z) \in \sS_d(\cJ_\infty, \cJ_0)$ be arbitrary and note that it does not coincide with $B'(Z) := B(Z) \oplus [\mathbf{0}_{\cJ', \cJ_0} \otimes I_n] \in \sS_d(\cJ_\infty \oplus \cJ', \cJ_0)$ if $\cJ'$ is a non-trivial Hilbert space. However, $\supp(B') \cong \supp(B)$ and therefore they coincide weakly.
\end{remark}

Combining Theorem \ref{thm:main.model.row.partial.iso} and Lemma \ref{lem:weak.coincidence.equiv} immediately gives us the following result.

\begin{cor}\label{cor:char.map.partial.iso.uni.equiv}
The NC characteristic function, $B_V \in \sS _d (\nbker V, \nbran V ^\perp)$ (as in \eqref{eqn:NC.char.map.V}) of a CNC row partial isometry $V$ is a complete unitary invariant: Two CNC row partial isometries, $V, V'$ on $\cH$ and $\cH'$, respectively, are unitarily equivalent if and only if their NC characteristic functions $B_V$, $B_{V'}$ coincide weakly.
\end{cor}

Given the above uniqueness result, we identify $B_V$ with its weak coincidence class, and refer to $B_V$ as \emph{the NC characteristic function} of a given CNC row partial isometry, $V$.

\section{NC de Branges--Rovnyak model for CNC row contractions}\label{sec:sanity.check.Gleason.solns.dbr.sp}

\subsection{NC Frostman shifts}\label{subsec:NC.Frostman.shifts}

Let $\cJ, \cK$ be any two Hilbert spaces and consider $(\scr{B} (\cJ, \cK) )_1$ and $[\scr{B} (\cJ, \cK) ]_1$, the open and closed unit balls of $\scr{B} (\cJ, \cK)$. For any pure contraction, $\alpha \in [\sB(\cJ, \cK)]_1$, an \emph{$\alpha-$M\"obius transformation} $\varphi_\alpha : (\scr{B} (\cJ, \cK) ) _1 \twoheadrightarrow (\scr{B} (\cJ, \cK))_1$ was introduced in \cite[Section 5.1]{MR19} as the map
\begin{equation*}
    \Phi_\alpha(\zeta) := D_{\alpha^*} (I_\cK - \zeta \alpha^*)^{-1}(\zeta - \alpha) D_\alpha^{-1}; \quad \quad \zeta \in (\sB(\cJ, \cK))_1
\end{equation*}
with compositional inverse,
\be \Phi _\alpha ^{-1} (\zeta ) = \Phi _{-\alpha ^*} (\zeta ^* ) ^* = D_{\alpha ^*} ^{-1} (\zeta + \alpha)( I_\cK - \zeta ^* \alpha) ^{-1} D_{\alpha}. \ee

\begin{lemma}{\cite[Lemma 5.2]{MR19}}\label{opMobius}
Let $\alpha \in [\scr{B} (\cJ, \cK)]_1$ be a contraction.
\bi
    \item[(i)] If $\alpha$ is a pure contraction, then $\Phi _\alpha : (\scr{B} (\cJ, \cK)) _1 \twoheadrightarrow (\scr{B} (\cJ, \cK) ) _1$ is a bijection with compositional inverse $\Phi _\alpha ^{-1}$.
    
   \item[(ii)] If $\alpha$ is a strict contraction, then $\Phi _\alpha : [\scr{B} (\cJ, \cK)] _1 \twoheadrightarrow [\scr{B} (\cJ, \cK) ] _1$ is a bijection with compositional inverse $\Phi _\alpha ^{-1}$. Moreover, $\Phi _\alpha$ is a bijection of the set of pure contractions in $[\scr{B} (\cJ, \cK)]_1$ onto itself.
   
   \item[(iii)] Let $\alpha$ be a pure contraction, and let $\beta \in (\sB(\cJ \otimes \bC^m, \cK \otimes \bC^m))_1$, $\ga \in (\sB(\cJ \otimes \bC^n, \cK \otimes \bC^n))_1$ be strict contractions. Then, for each $A \in \bC^{m \times n}$, we have
   \ei
    \begin{eqnarray}\label{eqn:technical.prop.NC.Frostman.transf}
        & & I_{\cK} \otimes A - \Phi_\alpha(\beta)[I_\cJ \otimes A] \Phi_\alpha(\ga)^* = \\
        & &  \big[ D_{\alpha^*} \otimes I_m \big]\big[ I_{\cK \otimes \bC^m} - \beta [\alpha^* \otimes I_m] \big]^{-1} \big[ I_\cK \otimes A - \beta [I_\cJ \otimes A]\ga^* \big] \nonumber \big[ I_{\cK \otimes \bC^n} - [\alpha \otimes I_n] \ga^* \big]^{-1} \big[D_{\alpha^*} \otimes I_n\big]. \nonumber
    \end{eqnarray}
    
\end{lemma}

\begin{remark}
The statements (i--ii) of the above lemma are from \cite[Lemma 5.2 (i--ii)]{MR19}. Statement (iii) is a technical modification of \cite[Lemma 5.2 (iii)]{MR19} that we will need.
Item (i) in the above Lemma is a slightly stronger statement than \cite[Lemma 5.2 (i)]{MR19}. Nevertheless, it follows from the proof of that result. Indeed, if $\alpha$ is any  pure contraction and $\beta$ is a strict contraction, it was shown in the proof of \cite[Lemma 5.2]{MR19} that for any $h \in \nbran D_\alpha$, 
$$ \| h \| ^2 - \| \Phi _\alpha (\beta ) h \| ^2 = \ipr{h}{D_{\alpha} (I- \beta^* \alpha  ) ^{-1} (I-\beta ^* \beta ) (I - \alpha ^* \beta ) ^{-1} D_\alpha h}. $$
Hence, if $\| \beta \| <1$, then it follows that $\Phi _\alpha (\beta)$ extends by continuity to a strict contraction. Similarly this same equation shows that if $\alpha$ is a strict contraction and $\beta$ is a pure contraction, then $\Phi _\alpha (\beta)$ is a pure contraction. 
\end{remark}

To simplify equations, it will be convenient to introduce the following notation. For any $\ell \in \bN$, $\beta \in (\sB(\cJ \otimes \bC^\ell, \cK \otimes \bC^\ell))_1$ and any pure contraction, $\alpha$, let
\begin{align}
    \Xi_\alpha(\beta) &:= \big[ D_{\alpha^*} \otimes I_\ell \big] \big[ I_{\cK \otimes \bC^\ell} - \beta [\alpha^* \otimes I_\ell]\big]^{-1}, \label{eqn:def.Xi}\\
    \Theta_\alpha(\beta) &:= \big( \beta - [\alpha \otimes I_\ell] \big) \big[ D_\alpha^{-1} \otimes I_\ell \big]. \label{eqn:def.Theta}
\end{align}
Then, \eqref{eqn:technical.prop.NC.Frostman.transf} can be interpreted in terms of adjunction maps (defined on appropriate domains) as
\begin{equation}\label{eqn:technical.prop.NC.Frostman.transf.adj}
    I_{\cK \otimes \bC^{n \times m}} - \mr{Ad}_{\Phi_\alpha(\zeta), \Phi_\alpha(\omega)^*} = \mr{Ad}_{\Xi_\alpha(\zeta), \Xi_\alpha(\omega)^*} \circ \big[ I_{\cK \otimes \bC^{n \times m}} - \mr{Ad}_{\zeta, \omega^*} \big].
\end{equation}

\begin{proof} We prove only (iii), proof of (i--ii) can be found in \cite[Lemma 5.2]{MR19}. Note that $\Phi_\alpha(\beta) = \Xi_\alpha(\beta) \Theta_\alpha(\beta)$ for each $\beta \in [\sB(\cJ, \cK)_\bN]_1$. Therefore, we have
    \begin{equation}\label{eqn:proof.technical.prop.NC.Frostman.shift.1}
        I_{\cK \otimes \bC^{n \times m}} - \mr{Ad}_{\Phi_\alpha(\beta), \Phi_{\alpha}(\ga)^*} = I_{\cK \otimes \bC^{n \times m}} - \big[ \mr{Ad}_{\Xi_\alpha(\beta), \Xi_\alpha(\ga)^*} \circ \mr{Ad}_{\Theta_\alpha(\beta), \Theta_\alpha(\ga)^*} \big]
    \end{equation}
    Looking at \eqref{eqn:technical.prop.NC.Frostman.transf.adj} and \eqref{eqn:proof.technical.prop.NC.Frostman.shift.1}, it then suffices to show that
    \begin{equation}
        I_{\cK \otimes \bC^{n \times m}} - \mr{Ad}_{\beta, \ga^*} = \mr{Ad}_{{\Xi_\alpha(\beta)}^{-1}, {\Xi_\alpha(\ga)^*}^{-1}} \ - \ \mr{Ad}_{\Theta_\alpha(\beta), \Theta_\alpha(\ga)^*}.
    \end{equation}

    First, we compute
    \begin{align}
        \mr{Ad}_{{\Xi_\alpha(\beta)}^{-1}, {{\Xi_\alpha(\ga)}^*}^{-1}}[I_\cK \otimes A] &= \big[ I_{\cK \otimes \bC^n} - \beta [\alpha^* \otimes I_n] \big] \big[ D_{\alpha^*}^{-2} \otimes A \big] \big[ I_{\cK \otimes \bC^m} - [\alpha \otimes I_m] \ga^* \big] \nonumber\\
        &= \big[ D_{\alpha^*}^{-2} \otimes A \big] - \beta \big[ \alpha^* D_{\alpha^*}^{-2} \otimes A \big] - \big[ D_{\alpha^*}^{-2} \alpha \otimes A \big] \ga^* + \beta \big[ \alpha^* D_{\alpha^*}^{-2} \alpha \otimes A \big] \ga^*. \label{eqn:proof.technical.prop.NC.Frostman.shift.2}
    \end{align}
    Similarly, we have
    \begin{align}
        \mr{Ad}_{\Theta_\alpha(\beta), \Theta_\alpha(\ga)^*}[I_{\cK} \otimes A] &= \big( \beta - [\alpha \otimes I_n] \big) \big[ D_\alpha^{-2} \otimes A \big] \big( \ga^* - [\alpha^* \otimes I_m] \big) \nonumber\\
        &= \big[ \alpha D_\alpha^{-2} \alpha^* \otimes A \big] - \beta \big[ D_\alpha^{-2} \alpha^* \otimes A \big] - \big[ \alpha D_\alpha^{-2} \otimes A \big] \ga^* + \beta \big[ D_\alpha^{-2} \otimes A \big] \ga^*. \label{eqn:proof.technical.prop.NC.Frostman.shift.3}
    \end{align}
    Using the properties of the defect operators, it is straightforward to check that
    \begin{equation*}
        D_{\alpha^*}^{-2} - \alpha D_\alpha^{-2} \alpha^* = I_\cK; \ \ \alpha^* D_{\alpha^*}^{-2} = D_\alpha^{-2} \alpha^*; \ \ D_{\alpha^*}^{-2} \alpha = \alpha D_\alpha^{-2}; \ \ \alpha^* D_{\alpha^*}^{-2} \alpha - D_\alpha^{-2} = - I_\cJ.
    \end{equation*}
    Thus, subtracting \eqref{eqn:proof.technical.prop.NC.Frostman.shift.3} from \eqref{eqn:proof.technical.prop.NC.Frostman.shift.2} immediately yields \eqref{eqn:proof.technical.prop.NC.Frostman.shift.1}, as required.
\end{proof}

Let $B \in \sS_d(\cJ, \cK)$. If $B$ is strictly contractive (so that $B(0)$ is a strict contraction), then $\BMob{0} := \Phi _{B(0) \otimes I_{\hardy}} \circ B \in \sS _d (\cJ _\infty, \cJ _0)$ is strictly contractive by item (ii) of Lemma \ref{opMobius}. Indeed, $B (L) \in \scr{B} (\cJ \otimes \hardy, \cK \otimes \hardy)$, so that $\Phi _{B(0)} \circ B \in \sS _d (\cJ , \cK)$. Moreover, for any $Z \in \B ^d _n$, $B(Z)$ is a strict contraction, so that $\BMob{0} (Z)$ is a strict contraction by Lemma \ref{opMobius} (ii). If $B \in \sS _d (\cJ, \cK)$ is only purely contractive, then by Theorem \ref{thm:char.map.gleason.soln}, we can still define $\BMob{0} \in \sS _d (\cJ, \cK)$ as the NC Schur class function,
$$ \BMob{0} (Z) := D_{B(0_n)}^* \big(I_{\cK \otimes \C^n} - B(Z) B(0_n)^* \big) ^{-1} (B(Z) - B(0_n)) D_{B(0_n)} ^{-1}; \quad Z \in \B ^d _n. $$ Since $\BMob{0} (0) = 0$, by construction, it follows that $\BMob{0}$ is a strictly contractive Schur class function for any $B \in \sS _d (\cJ, \cK)$. Hence, given any pure contraction $\alpha \in \scr{B} (\cJ, \cK)$ and any purely contractive $B \in \sS _d (\cJ, \cK)$, we can define $\BMob{\alpha} := \Phi _\alpha ^{-1} \circ \BMob{0} \in \sS _d (\cJ, \cK)$ point-wise by 
$$ \BMob{\alpha} (Z) = \Phi _{\alpha \otimes I_n} ^{-1} \circ \BMob{0} (Z), \quad \quad Z \in \B ^d _n, $$ by Lemma \ref{opMobius} (i), since $\BMob{0} (Z)$ is a strict contraction for any $Z \in \rball$.

\begin{definition}
    Let $B \in \sS _d (\cJ, \cK)$ be purely contractive and let $\alpha \in [\scr{B} (\cJ, \cK)]_1$ be a pure contraction. Then the \emph{$\alpha-$Frostman shift of $B$} is the NC operator--valued Schur class function, $\BMob{\alpha}=\Phi _\alpha ^{-1} \circ \BMob{0} \in \sS _d (\cJ, \cK)$.
\end{definition}

\begin{lemma}\label{lem:B(0).Mob.property}
For any purely contractive and column--extreme $B \in \sS _d (\cJ, \cK)$, we have $B = \BMob{B(0)} = (\BMob{0})^{\langle B(0) \rangle}$.
\end{lemma}
\begin{proof}
If $B \in \sS _d (\cJ, \cK)$ is strictly contractive, then this is an immediate consequence of Lemma \ref{opMobius}.
If $B$ is purely contractive, \emph{i.e.} $B(0)$ is a pure contraction, then $\BMob{0} \in \sS _d (\cJ, \cK)$ is defined by Theorem \ref{thm:char.map.gleason.soln}. Moreover, $\BMob{0} (0) =0$, by construction, so that $\BMob{0}$ is strictly contractive by Lemma \ref{strict}. To prove that $\BMob{B(0)} := \Phi _{B(0)} ^{-1} \circ \BMob{0} = B$, consider the standard NC model triple used in the proof of Theorem \ref{thm:char.map.gleason.soln}. Namely, recall that if we define $\Ga (Z) := (I-XZ^*) ^{-1} K^B _{0} D_{B(0_n)^*}^{-1}$, and $\check{\Ga} (\infty ) := \mbf{B} D_{B(0)} ^{-1}$, where $X$ is the unique extremal Gleason solution for $\scr{H} (B)$, then we recover $\BMob{0}$ (up to weak coincidence) via the formula $\BMob{0}(Z) \sim D(Z) ^{-1} N(Z)$, where 
$$ D(Z) = \Ga (Z) ^* \Ga (0_n), \quad \mbox{and} \quad N (Z) = \Ga (Z) ^* [I_{\scr{H}(B)} \otimes Z][\check{\Ga} (\infty) \otimes I_n]. $$ 

We claim that $D(Z) + N(Z) B(0_n) ^* = I _{\scr{D} _{B(0_n)^*}}$. First, consider 
$$ D_{B(0_n) ^*} D(Z) D_{B(0_n)^*}  =  [K_0 ^{B*} \otimes I_n] [I-XZ^*] ^{-*} [K_0 ^B \otimes I_n].$$
Hence, given any $g,g' \in \cK$ and $x,u \in \C^n$,
\ba  & & \ipr{g\otimes x}{[K_0 ^{B*} \otimes I_n] [I- XZ^*] ^{-*} (K_0 ^B g' \otimes u)}_{\cK \otimes \C ^n} \\
& = & \ipr{[I\otimes u^*] [I - X Z^*] ^{-1} (K_0 ^B g \otimes x)}{K_0 ^B g'}_{\cK} \\
& = & \ipr{K^B \{ Z, g\otimes x,u \}}{K_0 ^B g'}_{\scr{H} (B)} \qquad \qquad \mbox{(By Proposition \ref{GSkernel}.)} \\
& = & (g\otimes x) ^* \big(I_\cK \otimes I_n - B(0_n) B(Z) ^*\big) g' \otimes u. \ea Similarly,
$$ D_{B(0_n) ^*}  N(Z) B(0) ^* D_{B(0_n) ^*} = [K_0 ^{B*} \otimes I_n] [I-XZ^*] ^{-*} [I_{\scr{H} (B)} \otimes Z] \mbf{B} B(0_n) ^*, $$ and we calculate as before: For any $g,g' \in \cK$ and $x,u \in \C^n$
\ba
& & \ipr{g\otimes x}{[K_0 ^{B*} \otimes I_n] [I- XZ^*] ^{-*} [I_{\scr{H} (B)} \otimes Z]\mbf{B} B(0_n) ^* (g' \otimes u)}_{\cK \otimes \C ^n} \\
& = & \sum _{j=1} ^d \ipr{K^B \{Z, g\otimes x, Z_j u \}}{\mbf{B} B(0)^* g'}_{\scr{H} (B)} \\
& = & \sum _{j=1} ^d \ipr{g\otimes x}{\mbf{B} _j (Z) Z_j B(0_n) ^* (g' \otimes u)}_{\cK \otimes \C ^n} \\
& = & (g\otimes x) ^* (B(Z) - B(0_n)) B(0_n) ^* (g' \otimes u). \ea
Hence,
\ba D_{B(0) ^*} ( D(Z) + N(Z) B(0_n) ^*) D_{B(0) ^*} & = & I_\cK \otimes I_n - B(0_n) B(Z) ^* + B(0_n) B(Z) ^* - B(0_n) B(0_n) ^*  \\
& = & D_{B(0_n) ^* } ^2, \ea and we conclude that 
$$ D(Z) + N(Z) B(0_n) ^* = I _{\scr{D} _{B(0) ^*}} \otimes I_n = I _{\scr{D} _{B(0_n)^*}}. $$ Next we calculate the denominator of the expression,
\ba \BMob{B(0)} (Z) & = & \big(\Phi _{B(0)} ^{-1} \circ \BMob{0}\big) (Z) \\
& = & D_{B(0_n)^*} ^{-1} \big( \BMob{0} (Z) +B(0_n) \big)\big(I_{\cK} \otimes I_n + B(0_n) ^* \BMob{0} (Z)\big) ^{-1} D_{B(0_n)}. \ea
Namely,
\ba I_\cK \otimes I_n + B(0_n) ^* \BMob{0} (Z) & = & I + B(0_n) ^* D(Z) ^{-1} N(Z) \\
& = & I +  B(0_n ) ^* (I - N(Z) B(0_n)^*) ^{-1} N(Z)  \\
& = & I + (I-B(0_n)^* N(Z)) ^{-1} B(0_n) ^* N(Z) \\
& = & I + (I - B(0_n) ^* N(Z)) ^{-1} - I = (I-  B(0_n) ^* N(Z)) ^{-1}. \ea 
In conclusion, using \eqref{eqn:char.map.gleason.sol.proof.2}, we obtain
\ba D_{B(0_n) ^*}\BMob{B(0)} (Z) & = & \left( D(Z) ^{-1}N(Z) + B(0_n) \right) \left(I - B(0_n) ^* N(Z) \right)  D _{B(0_n)} \\ 
& = & ( D(Z) ^{-1} N(Z) + B(0_n) -D(Z) ^{-1} N(Z) B(0_n) ^* N(Z) -B(0_n) B(0_n) ^* N(Z)) D_{B(0_n)} \\
& = & ( D(Z) ^{-1} N(Z) + B(0_n) -D(Z) ^{-1} (I-D(Z)) N(Z) -B(0_n) B(0_n) ^* N(Z) ) D_{B(0_n)} \\
& = & ( D(Z) ^{-1} N(Z) + B(0_n) - D(Z) ^{-1} N(Z) +N(Z) - B(0_n) B(0_n) ^* N(Z)) D_{B(0_n)} \\
& = & B(0_n) D_{B(0_n)} + D_{B(0_n)^*}^2 N(Z) D_{B(0_n)} \\
& = & D_{B (0_n ) ^*} ( B(0_n) + B(Z) - B(0_n) ) = D_{B (0_n ) ^*} B(Z). \ea 
That is, $\BMob{B(0)} (Z) = B(Z)$ and the proof is complete.
\end{proof}

Our first main result in this subsection introduces $M^{\langle \alpha \rangle}$ -- an NC operator-valued analogue of a \emph{Crofoot multiplier/transform}. Recall the notation $\Xi_\alpha$ from \eqref{eqn:def.Xi}.

\begin{theorem}\label{thm:NC.Crofoot.transf}
    Let $B \in \sS_d(\cJ_\infty, \cJ_0)$ be purely contractive, and let $\alpha \in [\sB(\cJ_\infty, \cJ_0)]_1$ be a pure  contraction. Then, $B^{\langle \alpha \rangle} \in \sS_d(\cJ_\infty, \cJ_0)$ is purely contractive.

    Moreover, there is a unitary multiplier $M^{\langle \alpha \rangle} : \sH(\BMob{0}) \twoheadrightarrow \sH(B^{\langle \alpha \rangle})$ given by
    \begin{equation*}
        M^{\langle \alpha \rangle}(Z) := \Xi_{-\alpha}\big( \BMob{0}(Z) \big) = \big[ \Xi_\alpha \big( B^{\langle \alpha \rangle}(Z) \big) \big]^{-1} \foral Z \in \bB^d_\bN.
    \end{equation*}
\end{theorem}

\begin{proof}
    Using Lemma \ref{opMobius} (iii) and the notation from \eqref{eqn:def.Xi}, we get
    \begin{align*}
        I_{\cJ_0 \otimes \bC^{n \times m}} - \mr{Ad}_{\BMob{0}(Z), \BMob{0}(W)^*} &= I_{\cJ_0 \otimes \bC^{n \times m}} - \mr{Ad}_{\Phi_\alpha (B^{\langle \alpha \rangle}(Z)), \Phi_\alpha (B^{\langle \alpha \rangle}(W))^*} \\
        &= \Xi_\alpha \big( B^{\langle \alpha \rangle}(Z) \big) \Big[ I_{\cJ_0 \otimes \bC^{n \times m}} - \mr{Ad}_{B^{\langle \alpha \rangle}(Z), B^{\langle \alpha \rangle}(W)^*} \Big] \Xi_\alpha \big( B^{\langle \alpha \rangle}(W) \big)^*.
    \end{align*}

    Similarly, since $\Phi_\alpha^{-1}(\beta) = \big[ \Phi_{- \alpha^*}(\beta^*) \big]^*$, we can modify Lemma \ref{opMobius} $(iii)$ to obtain
    \begin{align*}
        I_{\cJ_0 \otimes \bC^{n \times m}} - \mr{Ad}_{B^{\langle \alpha \rangle}(Z), B^{\langle \alpha \rangle}(W)^*} &= I_{\cJ_0 \otimes \bC^{n \times m}} - \mr{Ad}_{\Phi_\alpha^{-1} (\BMob{0}(Z)), \Phi_\alpha^{-1} (\BMob{0}(W))^*} \\
        &= \Xi_{-\alpha} \big( \BMob{0}(Z) \big) \Big[ I_{\cJ_0 \otimes \bC^{n \times m}} - \mr{Ad}_{\BMob{0}(Z), \BMob{0}(W)^*} \Big] \Xi_{-\alpha} \big( \BMob{0}(W) \big)^*.
    \end{align*}

    Combining both these equations and using the fact that $\BMob{0} \in \sS_d(\cJ_\infty, \cJ_0)$, it follows at once from a simple NC modification of \cite[Theorem 2.1]{Ball2001-lift} that $B^{\langle \alpha \rangle} \in \sS_d(\cJ_\infty, \cJ_0)$ and that $M^{\langle \alpha \rangle}$ is a unitary left multiplier.
\end{proof}

The next result shows that $M^{\langle \alpha \rangle}$ preserves contractive and extremal Gleason solutions in the respective spaces as well.

\begin{theorem}\label{thm:M.alpha.preserves.Gleason.soln}
    Let $B \in \sS_d(\cJ_\infty, \cJ_0)$  be purely contractive, let $\alpha \in [\sB(\cJ_\infty, \cJ_0)]_1$ be a pure contraction, and let $\mbf B^{\langle \alpha \rangle}$ and $\mbf B^{\langle 0 \rangle}$ be the unique contractive Gleason solutions for $B^{\langle \alpha \rangle}$ and $\BMob{0}$ respectively. Then,
    \begin{equation*}
        \big[ M^{\langle \alpha \rangle} \otimes I_d \big]^{-1} \mathbf{B}^{\langle \alpha \rangle} D_\alpha^{-1} = \mbf B^{\langle 0 \rangle}.
    \end{equation*}

    Moreover, $\mbf B^{\langle \alpha \rangle}$ is extremal if and only if $\mbf B^{\langle 0 \rangle}$ is extremal. Hence $B$ is column--extreme if and only if $\BMob{\alpha}$ is column--extreme.
\end{theorem}

\begin{proof}
    Let $\Mob{\wt{\mbf{B}}}{\alpha}:=\big[ M^{\langle \alpha \rangle} \otimes I_d \big]^{-1} \mathbf{B}^{\langle \alpha \rangle} D_\alpha^{-1}$. Since $\BMob{0} \in \sS_d(\cJ_\infty, \cJ_0)$ and $\BMob{\alpha}(0) = \alpha$, we can use \eqref{ncGS} to compute
    \begin{align*}
        \big[ I_{\sH(\BMob{0})} \otimes Z \big] \big[ \Mob{\wt{\mbf{B}}}{\alpha} (Z)\otimes I_n] &= \big[ I_{\sH(\BMob{0})} \otimes Z \big] \big[ {{M^{\langle \alpha \rangle}}(Z)}^{-1} \otimes I_d \big] \big[ \mathbf{B}^{\langle \alpha \rangle} D_\alpha^{-1}(Z) \otimes I_n \big] \\
        &= {M^{\langle \alpha \rangle}(Z)}^{-1} \big[ I_{\sH(\BMob{0})} \otimes Z \big] \big[ \mathbf{B}^{\langle \alpha \rangle} D_\alpha^{-1}(Z) \otimes I_n \big] \\
        &= {{M^{\langle \alpha \rangle}}(Z)}^{-1} \big( B^{\langle \alpha \rangle}(Z) - B^{\langle \alpha \rangle}(0_n) \big) \big[ D_\alpha^{-1} \otimes I_n \big] \\
        &= \BMob{0}(Z) \\
        &= \BMob{0}(Z) - \BMob{0}(0_n).
    \end{align*}
    If $\alpha$ is not a strict contraction, then $D_{\alpha} ^{-1}$ is a closed, unbounded linear operator with dense domain. The above calculation remains valid in this case. One can, for example, multiply both sides of the above equation by $D_\alpha$ to reach the same conclusion. We now note that
    \begin{equation*}
        (\Mob{\wt{\mbf{B}}}{\alpha})^* \Mob{\wt{\mbf{B}}}{\alpha} = D_\alpha^{-1} {\mathbf{B}^{\langle \alpha \rangle}}^* \mathbf{B}^{\langle \alpha \rangle} D_\alpha^{-1} \leq D_\alpha^{-1} (I_{\cJ_\infty} - \alpha^* \alpha) D_\alpha^{-1} = I_{\cJ_0}.
    \end{equation*}
    Thus, $\Mob{\wt{\mbf{B}}}{\alpha} = \Mob{\mbf{B}}{0}$ and, since the inequality above is an equality if and only if $\mathbf{B}^{\langle \alpha \rangle}$ is extremal, we get that $\Mob{\mbf{B}}{\alpha}$ is extremal if and only if $\Mob{\mbf{B}}{0}$ is extremal. Recall, by \cite[Theorem 6.4]{JM-freeCE}, that $B \in \sS _d (\cJ, \cK)$ is column--extreme if and only if the unique, contractive Gleason solution for $B$, $\mbf{B} = R^* \otimes I_\cK B \in \scr{H} (B) \otimes \C ^d$ is extremal. Hence $\Mob{\mbf{B}}{\alpha}$ is extremal if and only if $\BMob{\alpha}$ is CE and the proof is complete.
\end{proof}

We record one final observation in this subsection that connects the row partial isometric part of any extremal Gleason solution $X$ for some $\sH(B)$ with the NC operator-valued Crofoot transformation $M^{\langle \alpha \rangle}$ introduced above.

\begin{prop}\label{prop:partial.iso.from.Crofoot.transf}
    Let $B \in \sS_d(\cJ_\infty, \cJ_0)$ be a purely contractive and CE (left) multiplier. Furthermore, let $X$, $X^{\langle 0 \rangle}$ be the unique extremal Gleason solutions for $\sH(B)$, $\sH(\BMob{0})$ respectively and let $\mbf B$, $\mbf B^{\langle 0 \rangle}$ be the corresponding Gleason solutions for $B$, $\BMob{0}$ respectively. If $X = V + C$ is the isometric--pure decomposition of $X$, then
    \begin{equation*}
        V = M^{\langle B(0) \rangle} X^{\langle 0 \rangle} \big[M^{\langle B(0) \rangle} \otimes I_d\big]^*.
    \end{equation*}
\end{prop}

\begin{proof}
We present the proof under the added assumption that $B(0)$ is a strict contraction, the argument extends readily to the case where $B(0)$ is an arbitrary pure contraction.

    It suffices to verify the action of $V^*$ on each kernel function $K^B\{W, g \otimes x, u\}$ with $W \in \bB^d_n$ for some $n \in \bN$. Recall from \eqref{eqn:partial.iso.part,of.Gleason.soln} that $V^* = X^* \big[ I_{\sH(B)} - K^B_0 D_{B(0)^*}^{-2} {K^B_0}^* \big]$. Combining this with \eqref{eqn:X^*.b.intertwine.ker.action}, \eqref{eqn:Kernel.at.0.and.W} and \eqref{eqn:kernel.at.0.and.X*}, we get
    \begin{align*}
        V^* K^B\{W, g \otimes x, u\} &= X^* K^B\{W, g \otimes x, u\} - X^* K^B_0 D_{B(0)^*}^{-2} {K^B_0}^* K^B\{W, g \otimes x, u\} \\
        &= K^B\{W, g \otimes x, \col(W) u\} - \mathbf{B} [I_{\cJ_\infty} \otimes u^*] B(W)^*(g \otimes x) \\
        &\qquad\qquad\qquad\qquad\qquad\qquad + \mathbf{B} B(0)^* D_{B(0)^*}^{-2} K^B(0,W)[u^*](g \otimes x) \\
        &= K^B\{W, g \otimes x, \col(W) u \} - \mathbf{B} D_{B(0)}^{-2} \Big[ \big(I_{\cJ_\infty} - B(0)^* B(0) \big) [I_{\cJ_\infty} \otimes u^*] B(W)^* (g \otimes x) \\
        &\qquad\qquad\qquad\qquad\qquad\qquad - B(0)^* \big( [I_{\cJ_0} \otimes u^*] - B(0) [I_{\cJ_\infty} \otimes u^*] B(W)^* \big) (g \otimes x) \Big] \\
        &= K^B\{W, g \otimes x, \col(W) u \} - \mathbf{B} D_{B(0)}^{-2} \Big[ [I_{\cJ_\infty} \otimes u^*] B(W)^* - B(0)^* [I_{\cJ_0} \otimes u^*] \Big] (g \otimes x) \\
        &= K^B\{W, g \otimes x, \col(W) u \} - \mathbf{B} D_{B(0)}^{-2} [I_{\cJ_\infty} \otimes u^*] \big( B(W)^* - B(0_n)^* \big) (g \otimes x).
    \end{align*}

    It is easy to check from the definition that
    \begin{align*}
        {M^{\langle B(0) \rangle}}^* K^B\{W, g \otimes x, u\} &= K^{\BMob{0}}\{W, M^{\langle B(0) \rangle}(W)^*(g \otimes x), u\}; \\
        \big[ M^{\langle B(0) \rangle} \otimes I_d \big] K^{\BMob{0}}\{ W, g \otimes x, \col(W) u \} &= K^B \{ W, M^{\langle B(0) \rangle}(g \otimes x), \col(W) u \}.
    \end{align*}
    Since $M^{\langle B(0) \rangle}$ is a unitary, we can then use Theorem \ref{thm:M.alpha.preserves.Gleason.soln} to easily compute
    \begin{align*}
        &\big[ M^{\langle B(0) \rangle} \otimes I_d \big] {X^{\langle 0 \rangle}}^* {M^{\langle B(0) \rangle}}^* K^B\{W, g \otimes x, u\} \\
        &\qquad\qquad = \big[ M^{\langle B(0) \rangle} \otimes I_d \big] {X^{\langle 0 \rangle}}^* K^{\BMob{0}}\{W, M^{\langle B(0) \rangle}(W)^*(g \otimes x), u\} \\
        &\qquad\qquad = K^B\{W, g \otimes x, \col(W) u\} - \big[ M^{\langle B(0) \rangle} \otimes I_d \big] \mathbf{B}^{\langle 0 \rangle}[I_{\cJ_\infty} \otimes u^*] \BMob{0}(W)^* M^{\langle B(0) \rangle}(W)^* (g \otimes x) \\
        &\qquad\qquad = K^B\{W, g \otimes x, \col(W) u\} - \mathbf{B} D_{B(0)}^{-1} [I_{\cJ_\infty} \otimes u^*] \BMob{0}(W)^* M^{\langle B(0) \rangle}(W)^*(g \otimes x) \\
        &\qquad\qquad = K^B\{W, g \otimes x, \col(W) u\} - \mathbf{B} D_{B(0)}^{-2} [I_{\cJ_\infty} \otimes u^*] (B(W)^* - B(0_n)^*) (g \otimes x).
    \end{align*}

    Thus, $V = M^{\langle B(0) \rangle} X^{\langle 0 \rangle} \big[ M^{\langle B(0) \rangle} \otimes I_d \big]^*$ as desired, and the proof is complete.
\end{proof}

\begin{corollary}\label{cor:char.map.of.partial.iso.weakly.coincide.char.map}
    Let $B \in \sS_d(\cJ_\infty, \cJ_0)$ be purely contractive and CE. If $(\Ga, \cJ_\infty \oplus \cJ', \cJ_0)$ is the model triple corresponding to the row partial isometric part of $X$ as in Corollary \ref{cor:model.triple.for.Gleason.soln}, then its characteristic map $B^\Ga$ coincides weakly with $\BMob{0} \in \sS_d(\cJ_\infty, \cJ_0)$. In particular, if $B(0) = 0$, then $B^\Ga$ coincides weakly with $B$.
\end{corollary}

Recall from Section \ref{subsec:model.constr.Gleason.soln} that if $B \in \sS_d(\cJ_\infty, \cJ_0)$ is purely contractive and CE, and if $X$ is the unique extremal Gleason solution of $\sH(B)$ with isometric--pure decomposition $X = V + C$, then $(\Ga, \cJ_\infty \oplus \cJ', \cJ_0)$ is a model triple for $V$, where
\begin{equation*}
    \Ga :=
    \begin{cases}
        \ 0 &\mapsto \quad \Gamma(0) := K_0^B D_{B(0)^*}^{-1}; \\
        \ Z &\mapsto \quad [I - XZ^*]^{-1} [\Ga(0) \otimes I_n], \, \forall Z \in \bB^d_\bN; \\
        \ \infty &\mapsto \quad \check \Ga (\infty) \oplus \Ga'(\infty),
    \end{cases}
\end{equation*}
and $\check \Ga(\infty) := \mbf{B} D_{B(0)}^{-1}$, $\Ga'(\infty)$ are as in \eqref{eqn:gamma.prime.infty}.

\begin{prop}\label{prop:char.map.extremal.Gleason.soln.general}
    Let $B \in \sS_d(\cJ_\infty, \cJ_0)$, $X = V + C$, and $(\Ga, \cJ_\infty \oplus \cJ', \cJ_0)$ be as above. Then,
    \begin{equation*}
        \delta^\Ga_X := - \Ga(0)^* C \Ga(\infty) = B(0) \oplus \mbf{0}_{\cJ', \cJ_\infty}.
    \end{equation*}
    
    Consequently, $B_X := {B^\Ga}^{\langle \delta^\Ga_X \rangle}$ coincides weakly with $B$.
\end{prop}

\begin{proof}
    From \eqref{eqn:partial.iso.part,of.Gleason.soln}, it is clear that $C = K^B_0 D_{B(0)^*}^{-2} {K^B_0}^* X$. Thus, as $\mbf B$ is extremal, we can calculate $\delta^\Ga_X$ using \eqref{eqn:K^b(0,0)}, \eqref{eqn:kernel.at.0.and.X*}, and the fact that $\nbran \Ga'(\infty) \subseteq \nbran \mbf{B}^\perp = \nbker \mbf{B}^*$ as follows:
    \begin{align*}
        \delta^\Ga_X &= - D_{B(0)^*}^{-1} \underbrace{{K^B_0}^* K_0^B D_{B(0)^*}^{-2}}_{= I_{\cJ_0}} {K^B_0}^* X \big[ \mbf{B} D_{B(0)}^{-1} \oplus \Ga'(\infty) \big] \\
        &= - D_{B(0)^*}^{-1} {K^B_0}^* X \big[\mbf{B} D_{B(0)}^{-1} \oplus \Ga'(\infty)\big] \\
        &= D_{B(0)^*}^{-1} B(0) \mbf{B}^* \big[\mbf{B} D_{B(0)}^{-1} \oplus \Ga'(\infty)\big] \\
        &= D_{B(0)^*}^{-1} B(0) D_{B(0)} \oplus \mathbf{0}_{\cJ', \cJ_0} \\
        &= B(0) \oplus \mathbf{0}_{\cJ', \cJ_0}.
    \end{align*}

    The second part follows immediately from Lemma \ref{lem:B(0).Mob.property} since we know from Corollary \ref{cor:char.map.of.partial.iso.weakly.coincide.char.map} that $B^\Ga$ coincides weakly with $\BMob{0}$.
\end{proof}

\subsection{The main model}\label{subsec:main.model}

For the remaining discussion, let $T : \cH \otimes \bC^d \to \cH$ be a CNC row contraction with isometric--pure decomposition $T = V + C$ (so that $C$ is a pure row contraction). Also, let $(\ga, \cJ_\infty, \cJ_0)$ be any model triple corresponding to $V$.

\begin{definition}\label{def:char.map.CNC.row.contraction}
    The \emph{defect point} of $T$ (with respect to $\ga$) is defined as
    \begin{equation*}
        \delta_T^\ga := - \ga(0)^* T \ga(\infty) = - \ga(0)^* C \ga(\infty).
    \end{equation*}
    The \emph{characteristic function $B_T$ of $T$} is defined as the weak coincidence class of $B^\ga_T := {B^\ga}^{\langle \delta^\ga_T \rangle}$.
\end{definition}

The following result shows that $B_T$ is independent of the choice of the model for $V$.

\begin{lemma}\label{lem:coincidence.of.char.map.of.T.invariant.under.model}
    If $(\gamma_j, \cJ_\infty^{(j)}, \cJ_0^{(j)})$ are any two model triples corresponding to $V$, then $B^{\ga_1}_T$ and $B^{\ga_2}_T$ coincide (as in \eqref{eqn:def.coincidence.class}).
\end{lemma}

\begin{proof}
    As in the proof of Proposition \ref{prop:uniqueness.of.graded.model}, we have unitaries $U_0 \in \sB(\cJ_0^{(1)}, \cJ_0^{(2)})$ and $U_\infty \in \sB(\cJ_\infty^{(1)}, \cJ_\infty^{(2)})$ given by $\ga_1(0) = \ga_2(0) U_0$, $\ga_1(\infty) = \ga_2(\infty) U_\infty$ such that
    \begin{equation*}
         [U_0 \otimes I_n] B^{\ga_1}(Z) = B^{\ga_2}(Z) [U_\infty \otimes I_n] \foral Z \in \bB^d_n \AND n \in \bN.
    \end{equation*}
    Thus, note that
    \begin{equation*}
        \delta^{\ga_1}_T = -\ga_1(0)^* T \ga_1(\infty) = -U_0^* \ga_2(0)^* T \ga_2(\infty) U_\infty = U_0^* \delta^{\ga_2}_T U_\infty.
    \end{equation*}
    Also, recall from Definition \ref{def:char.map} that $B^{\ga_j}(0) = 0$ and thus, ${B^{\ga_j}}^{\langle 0 \rangle} = B^{\ga_j}$. We can therefore calculate for each $Z \in \bB^d_n$ and $n \in \bN$ that
    \begin{align*}
        &B^{\ga_1}_T(Z) \\
        &\ = {B^{\ga_1}}^{\langle \delta^{\ga_1}_T \rangle}(Z) \\
        &\ = \Big[ D_{{\delta^{\ga_1}_T}^*}^{-1} \otimes I_n \Big] \Big( {B^{\ga_1}}(Z) + [\delta^{\ga_1}_T \otimes I_n] \Big) \Big[ I_{\cJ_0^{(1)} \otimes \bC^n} + [{\delta^{\ga_1}_T}^* \otimes I_n] {B^{\ga_1}}(Z) \Big]^{-1} \Big[ D_{\delta^{\ga_1}_T} \otimes I_n \Big] \\
        &\ = \Big[ U_0^* D_{{\delta^{\ga_2}_T}^*}^{-1} U_0 \otimes I_n \Big] [U_0^* \otimes I_n] \Big( {B^{\ga_2}}(Z) + [\delta^{\ga_2}_T \otimes I_n] \Big) [U_\infty \otimes I_n] \\
        &\ \qquad\qquad\qquad\qquad [U_\infty^* \otimes I_n] \Big[ I_{\cJ_0^{(2)} \otimes \bC^n} + {B^{\ga_2}}(Z) [\delta^{\ga_2}_T \otimes I_n] \Big]^{-1} [U_\infty \otimes I_n] \Big[ U_\infty^* D_{\delta^{\ga_2}_T} U_\infty \otimes I_n \Big] \\
        &\ = [U_0^* \otimes I_n] \Big[ D_{{\delta^{\ga_2}_T}^*}^{-1} \otimes I_n \Big] \Big( {B^{\ga_2}}(Z) + [\delta^{\ga_2}_T \otimes I_n] \Big) \Big[ I_{\cJ_0^{(2)} \otimes \bC^n} + [{\delta^{\ga_2}_T}^* \otimes I_n] {B^{\ga_2}}(Z) \Big]^{-1} \Big[ D_{\delta^{\ga_2}_T} \otimes I_n \Big] [U_\infty \otimes I_n] \\
        &\ = [U_0^* \otimes I_n] B^{\ga_2}_T(Z) [U_\infty \otimes I_n].
    \end{align*}
    Thus, $B^{\ga_1}_T$ coincides with $B^{\ga_2}_T$, as required.
\end{proof}

Combining Theorem \ref{thm:main.model.row.partial.iso}, Proposition \ref{prop:char.map.extremal.Gleason.soln.general} and Lemma \ref{lem:coincidence.of.char.map.of.T.invariant.under.model} yields the following result.

\begin{theorem}\label{thm:description.char.map.of.CNC.row.contraction}
    Let $T : \cH \otimes \bC^d \to \cH$ be a CNC row contraction and let $T = V + C$ be its isometric--pure decomposition. The NC characteristic function, $B_T \in \sS_d(\sD_T, \sD_{T^*})$, of $T$, is the weak coincidence class of the operator-M\"obius transformation of the NC characteristic function $B_V \in \sS_d(\sD_T, \sD_{T^*})$ of $V$ (as in \eqref{eqn:NC.char.map.V}) by a pure contraction, $\delta_T \in \sB(\sD_T, \sD_{T^*})$, that depends only on $C$:
    \begin{equation}\label{eqn:NC.char.map.T}
        B_T(Z) = \big[ D_{\delta_T^*} \otimes I_n \big] \Big( B_V(Z) + \big[ \delta_T \otimes I_n \big] \Big) \Big( I_{\sD_T \otimes \bC^n} + \big[ {\delta_T}^* \otimes I_n \big] B_V(Z) \Big)^{-1} \big[ D_{\delta_T} \otimes I_n \big].
    \end{equation}
    
    In particular, $T$ is a CNC row partial isometry if and only if
    \begin{equation*}
        C = 0 \iff B_T = B_V \iff \delta_T = 0 \iff B_T(0) = 0.
    \end{equation*}
\end{theorem}

We need one final ingredient to complete our construction of our NC de Branges--Rovnyak model for arbitrary CNC row contractions. The following is easy to verify using the fact that $\ga(0)$ and $\ga(\infty)$ are onto isometries.

\begin{lemma}\label{lem:defect.point.bijection}
    Given a row partial isometry $V$ and a model triple $(\ga, \cJ_\infty, \cJ_0)$, the map
    \begin{equation*}
        \delta \mapsto T_\delta := V + \ga(0) \delta \ga(\infty)^*
    \end{equation*}
    defines a bijection from pure contractions in $[\sB(\cJ_\infty, \cJ_0)]_1$ onto row contractions with partial isometric part $V$ and purely contractive part $\ga(0) \delta \ga(\infty)^*$. Moreover, $\delta$ is a strict contraction if and only if the purely contractive part of $T_\delta$ is a strict row contraction. Furthermore, $T_\delta$ is CNC if and only if $V$ is CNC.
\end{lemma}

\begin{theorem}[\bf NC de Branges--Rovnyak model for CNC row contractions]\label{thm:main.model}
    A row contraction $T : \cH \otimes \bC^d \to \cH$ is CNC if and only if it is unitarily equivalent to the (unique) extremal Gleason solution for an NC de Branges--Rovnyak space $\sH(B_T)$, where $B_T$ is purely contractive and CE (left) multiplier. This $B_T$ is unique up to weak unitary coincidence and can be chosen so that $B_T \in \sS_d(\sD_T, \sD_{T^*})$ is given by \eqref{eqn:NC.char.map.T}.
\end{theorem}

\begin{proof}
    If $T$ is unitarily equivalent to the extremal Gleason solution of an NC Schur function $B$ as in the hypothesis, then it is clear from Corollary \ref{cor:X.is.CNC} that $T$ is CNC.

    Conversely, suppose $T$ is CNC and let $(\Ga := \Ga_V, \ \cJ_\infty := \nbker V, \ \cJ_0 := \nbran V^\perp)$ be the canonical NC model triple for $V$ (as in Remark \ref{rem:canon.model.triple}). We know from Theorem \ref{thm:char.map.is.NC.analytic.D.is.a.unitary.mult} that $V$ is unitarily equivalent to the extremal Gleason solution $X^\Ga = {M^L_{D^\Ga}}^* \cU^\Ga V [{\cU^\Ga}^* M^L_{D^\Ga} \otimes I_d]$ for $\sH(B^\Ga)$, corresponding to the extremal Gleason solution $\mbf B^\Ga = [{M^L_{D^\Ga}}^* \cU^\Ga \otimes I_d] \Ga(\infty)$ for $B^\Ga$. Moreover, if we let $\widehat{\cU} :=  {M^L_{D^\Ga}}^* \cU^\Ga$, then it follows that $(\widehat{\Ga} := \widehat{\cU} \Ga, \cJ_\infty, \cJ_0)$ is a model triple corresponding to $X^\Ga$.
    
    By definition we have ${B^\Ga_T}^{\langle 0 \rangle} = B^\Ga$ and $B^\Ga_T (0) = \delta^\Ga_T$, and hence Theorem \ref{thm:M.alpha.preserves.Gleason.soln} tells us that $\mbf B_T := {G^{\langle \delta^\Ga_T \rangle}}^{-1} \mbf B^\Ga$ is the unique extremal Gleason solution for $B^\Ga_T$. If $X_T$ is the corresponding extremal Gleason solution for $\sH(B^\Ga_T)$, then we know from Proposition \ref{prop:partial.iso.from.Crofoot.transf} that $V_T$, the row partial isometric part of $X_T$, satisfies
    \begin{equation*}
        V_T = M^{\langle \delta^\Ga_T \rangle} X [M^{\langle \delta^\Ga_T \rangle} \otimes I_d]^*.
    \end{equation*}
    To show that $T$ is unitarily equivalent to $X_T$, we will show that
    \begin{equation*}
        \widehat{T} := \widehat{\cU} T [\widehat{\cU} \otimes I_d]^* = {M^{\langle \delta^\Ga_T \rangle}}^* X_T [M^{\langle \delta^\Ga_T \rangle} \otimes I_d] =: \widehat{X}.
    \end{equation*}
    Since $\widehat{X}$ and $\widehat{T}$ are both row contractions with row partial isometric part $X$, Lemma \ref{lem:defect.point.bijection} tells us that it suffices to show $\delta^{\widehat{\Ga}}_{\widehat{T}} = \delta^{\widehat{\Ga}}_{\widehat{X}}$ holds. We therefore calculate
    \begin{equation*}
        \delta^{\widehat{\Ga}}_{\widehat{T}} = - \Ga(0)^* \widehat{\cU}^* \widehat{T} [\widehat{\cU} \otimes I_d] \Ga(\infty) = - \Ga(0) T \Ga(\infty) = \delta^\Ga_T.
    \end{equation*}
    Similarly, we have
    \begin{align}\label{eqn:main.model.proof.1}
        \delta^{\widehat{\Ga}}_{\widehat{X}} &= -\big(\widehat{\cU} \Ga(0)\big)^* \widehat{X} \widehat{\cU} \Ga(\infty) \nonumber\\
        &= - \big(\widehat{\cU} \Ga(0)\big)^* {M^{\langle \delta^\Ga_T \rangle}}^* X_T [M^{\langle \delta^\Ga_T \rangle} \otimes I_d] \mbf B^\Ga \nonumber\\
        &= - \big( {X_T}^* M^{\langle \delta^\Ga_T \rangle} \widehat{\cU} \Ga(0) \big)^* \mbf B_T D_{\delta^\Ga_T}^{-1}.
    \end{align}
    Now, define the map $K^\Ga_0 : \cJ_0 \to \cH^\Ga$ as
    \begin{equation*}
        K^\Ga_0 g := K^\Ga\{0, g \otimes 1, 1\} \foral g \in \cJ_0.
    \end{equation*}
    Then, it is easy to check that $\cU^\Ga \Ga(0) = K^\Ga_0$. Then, since $M^L_{D^\Ga} = {M^L_{{D^\Ga}^{-1}}}^*$ is the adjoint of a multiplier, a calculation similar to \eqref{eqn:left.mult.action.on.kernel.maps} yields
    \begin{equation*}
        \widehat{\cU}^\Ga \Ga(0) = K^{B_V}_0 D(0)^{-*} = K_0^{B_V}.
    \end{equation*}
    Similarly, $M^{\langle \delta^\Ga_T \rangle} = \big({M^{\langle \delta^\Ga_T \rangle}}^{-1}\big)^*$ is also a multiplier so we use \eqref{eqn:left.mult.action.on.kernel.maps} once again to conclude
    \begin{equation*}
        M^{\langle \delta^\Ga_T \rangle} \widehat{\cU}^\Ga \Ga(0) = K^{B_T}_0 {M^{\langle \delta^\Ga_T \rangle}}^{-*}(0) = K^{B_T}_0 D_{{\delta^\Ga_T}^*}^{-1}.
    \end{equation*}
    Next, as $X_T$ is the extremal Gleason solution for $\sH(B_T)$, we can use \eqref{eqn:X^*.b.intertwine.ker.action} for $W = 0$ to obtain
    \begin{equation*}
        {X_T}^* M^{\langle \delta^\Ga_T \rangle} \widehat{\cU}^\Ga \Ga(0) = - \mbf B_T {\delta^\Ga_T}^* D_{{\delta^\Ga_T}^*}^{-1}.
    \end{equation*}
    Plugging this back into \eqref{eqn:main.model.proof.1} and using the fact that $\mbf B_T$ is extremal gives us
    \begin{equation*}
        \delta^{\widehat{\Ga}}_{\widehat{T}} = D_{{\delta^\Ga_T}^*}^{-1} \delta^\Ga_T {\mbf B^T}^* \mbf B^T D_{\delta^\Ga_T}^{-1} = \delta^\Ga_T D_{\delta^\Ga_T}^{-1} D_{\delta^\Ga_T}^2 D_{\delta^\Ga_T}^{-1} = \delta^\Ga_T,
    \end{equation*}
    and hence $\delta^{\widehat{\Ga}}_{\widehat{T}} = \delta^{\widehat{\Ga}}_{\widehat{X}}$ as required. This completes the proof.
\end{proof}

\begin{corollary}\label{cor:char.map.is.uni.inv}
    The NC characteristic function, $B_T$, of a CNC row contraction $T$ is a complete unitary invariant: Two CNC row contractions, $T$, $T'$ on $\cH$, $\cH'$, respectively, are unitarily equivalent if and only if their NC characteristic functions $B_T$, $B_{T'}$ coincide weakly. 
\end{corollary}

\begin{proof}
    Suppose $T$, $T'$ as in the hypothesis are unitarily equivalent CNC row contractions, i.e., there exists a unitary $U : \cH \twoheadrightarrow \cH'$ such that $T' = U T [U^* \otimes I_d]$, and let $T = V + C$ and $T' = V' + C'$ be the respective isometric--pure decompositions. It then follows from Lemma \ref{isopure} that
    \begin{equation*}
        V' = U V [U^* \otimes I_d] \qand C' = U C [U^* \otimes I_d].
    \end{equation*}
    
    If $(\ga, \cJ_\infty, \cJ_0)$ is a model triple corresponding to $V$, then it is immediate that $(U \ga, \cJ_\infty, \cJ_0)$ is a model triple corresponding to $V'$. Moreover, it is straightforward to check from the definition that $B^\ga = B^{U \ga}$ and that $\delta^\ga_T = \delta^{U\ga}_{T'}$. Consequently, we get that ${B^\ga}^{\langle \delta^\ga_T \rangle} = {B^{U \ga}}^{\langle \delta^{U \ga}_{T'} \rangle}$ and Lemma \ref{lem:coincidence.of.char.map.of.T.invariant.under.model} shows that $B_T = B_{T'}$.

    Conversely, suppose $B_T \in \sS_d(\cJ, \cK)$ and $B_{T'} \in \sS_d(\cJ', \cK')$ coincide weakly. By Lemma \ref{lem:weak.coincidence.equiv}, this means that there is a unitary $U : \cK \twoheadrightarrow \cK'$ such that $\big[I_{\bH^2_d} \otimes U\big] \sH(B) = \sH(B')$. Using the uniqueness of Gleason solutions and applying \eqref{eqn:def.Gleason.soln}, we immediately get that the Gleason solutions $X$, $X'$ for $\sH(B)$, $\sH(B')$ satisfy
    \begin{equation*}
        X' = \big[I_{\bH^2_d} \otimes U\big] X \big[I_{\bH^2_d} \otimes U \otimes I_d\big].
    \end{equation*}
    Then, Theorem \ref{thm:main.model} shows at once that $T$ and $T'$ are unitarily equivalent, as required.
\end{proof}

\section{Relationship between the commutative and NC models} \label{sec:NCvscomm}

In \cite{MR19}, a row contraction, $T : \cH \otimes \C ^d \rightarrow \cH$ is said to obey the \emph{commutative CNC condition}, written $T$ is CCNC, if 
$$ \bigvee _{z \in \mathbb{B} ^d} (I-Tz^*)^{-1} \nbran D_{T^*} = \cH, $$ where $\mathbb{B} ^d$ denotes the unit ball of $\C^d$. Note that any CCNC row contraction is automatically CNC by Equation (\ref{CNCformula}). One of the main results of \cite{MR19}, Theorem 5.14, is that a row contraction, $T$, is CCNC if and only if it is unitarily equivalent to a contractive, extremal Gleason solution in a multi-variable de Branges--Rovnyak space, $\scr{H} (b)$, for a contractive multiplier, $b \in [ H^\infty _d \otimes \scr{B} (\cJ, \cK)]_1$, where $H^\infty _d$ denotes the multiplier algebra of $H^2 _d$, the Drury--Arveson space. By Popescu's commutant lifting theorem for row contractions, any $b \in [ H^\infty _d \otimes \scr{B} (\cJ, \cK)]_1$ has a contractive (left) ``free lift" to a contractive left multiplier $B \in \sS _d (\cJ, \cK)$, although this lift need not be unique. This lift will be unique if and only if $b$ is column--extreme (CE), by \cite[Corollary 7.2]{JMfree}. Letting $S=(S_1, \cdots, S_d) : H^2 _d \otimes \C ^d \rightarrow H^2 _d$ denote the \emph{Arveson $d-$shift}, $S_i := M_{z_i}$, on the Drury--Arveson space, $H^2 _d$, we write $F(S) := M_F$ for the linear operator of multiplication by $F \in H^\infty _d$. By \cite[Lemma 7.5]{JMfree}, there is a co-isometry, $C_{H^2}$ of $\scr{H} (B)$ onto the multivariable de Branges--Rovnyak space, $$\scr{H} (b) := \scr{R} (\sqrt{I- b(S) b(S)^*}),$$ contractively contained in $H^2 _d \otimes \cK$, with initial space $$\big(I_{\hardy} - B(L)B(L) ^* \big) H^2 _d \otimes \cJ$$ defined by 
$$ C_{H^2} \big(I_{\mathbb{H} ^2} -B(L) B(L)^* \big) h := \big(I_{H^2} -b(S) b(S)^*\big) h; \quad \quad h \in H^2 _d \otimes \cJ. $$ 
Gleason solutions for (commutative) multivariate de Branges--Rovnyak spaces, $\scr{H} (b)$, $b \in [H^\infty _d \otimes \scr{B} (\cJ, \cK)]_1$, and for contractive (Schur-class) multipliers, are defined analogously to the free case (see Subsection \ref{subsec:Gleason.solutions}). (In contrast to the free case, Gleason solutions for a contractive multiplier between vector--valued Drury--Arveson spaces, and for its de Branges--Rovnyak space need not be unique. A sufficient condition for $b \in [H^\infty _d \otimes \scr{B} (\cJ, \cK)]_1$ to have a unique, contractive and extremal Gleason solution, $\mbf{b}$, and for $\scr{H} (b)$ to have a unique extremal Gleason solution, $X$, is that $b$ be \emph{column--extreme (CE)}, as in Subsection \ref{ss:NCSchur}.)  If $X$ is a contractive and extremal Gleason solution for $\scr{H} (b)$, then the unique contractive Gleason solution for $\scr{H}  (B)$ is necessarily extremal and these are related by 
$$ X = C_{H^2}  X^B \big[C_{H^2} ^* \otimes I_d\big], $$ by \cite[Theorem 7.13]{JMfree}. 

Hence, (assuming that $\mr{supp} (B) =\cJ$), we have that if $T$ is a CCNC row contraction, then its commutative model has characteristic function $b$. However, since $T$ is necessarily CNC, it has an NC model with NC characteristic function $B_T$, and $T$ is unitarily equivalent to the extremal Gleason solution, $X^B$, acting on $\scr{H} (B_T)$. Since $X^B$ is extreme, $B$ is necessarily CE. In conclusion, we obtain the following result:

\begin{cor}
A purely contractive $b \in [H^\infty \otimes \scr{B} (\cJ, \cK)]_1$ has a CE free lift $B \in \sS _d (\cJ, \cK) = [\mult \otimes \scr{B} (\cJ, \cK)]_1$ if and only if $\scr{H} (b)$ has a contractive and extremal Gleason solution. 
\end{cor}

\begin{proof}
If $b$ has a column--extreme free lift, $B \in \sS_d (\cJ, \cK)$, then the unique contractive Gleason solution, $X^B$ for $\scr{H} (B)$ is extremal. By \cite[Theorem 7.13]{JMfree}, if we define 
$X := C_{H^2}  X^B \big[C_{H^2} ^* \otimes I_d\big]$, as above, then $X$ is an extremal Gleason solution for $\scr{H}(b)$, and $X$ is then a CCNC row contraction with characteristic function $b$ by \cite{MR19}. 

Conversely if $\scr{H} (b)$ has an extremal Gleason solution, $X$, then $X$ is a CCNC row contraction by \cite{MR19}, and hence $X$ is also CNC. By \cite[Theorem 7.13]{JMfree}, there exists a free lift, $B \in \sS _d (\cJ, \cK)$, of $b$, so that if $\mbf{B}$ and $X^B$ are the unique contractive Gleason solutions for $B$ and $\scr{H} (B)$, respectively, then $\mbf{B}$ and $X^B$ are both extremal so that $B$ is column--extreme by \cite[Theorem 6.4]{JM-freeCE}. 
\end{proof}

\appendix

\section{Comparison with the Sz.-Nagy--Foias--Popescu model}\label{appx:compare.with.NFP.model}

Popescu \cite{Popchar} extended the classical Sz.-Nagy--Foias characteristic function of a CNC contraction to the case of infinite tuples of non-commuting row contractions. In our setting of a (finite) $d$-tuple of row contraction, say $T : \cH \otimes \bC^d \to \cH$, the Sz.-Nagy--Foias--Popescu characteristic function $\Theta_T \in \sS_d(\sD_T, \sD_{T^*})$ is given by
\begin{equation*}
    \Theta_T(Z) := \Big( -[T \otimes I_n] + [D_{T^*} \otimes I_n] [I - ZT^*]^{-1} [I_\cH \otimes Z] [D_T \otimes I_n] \Big) \Big\vert_{\sD_T \otimes \bC^n}.
\end{equation*}

\begin{theorem}\label{thm:NFP.char.coincide.weakly.with.B_T}
    Let $T : \cH \otimes \bC^d \to \cH$ be a CNC row contraction. Then, the characteristic function $B_T$ coincides weakly with the Sz.--Nagy--Foias--Popescu characteristic function $\Theta_T$.
\end{theorem}

\begin{proof}
    Using the main model and Theorem \ref{thm:main.model}, it suffices to show that if $B \in \sS_d(\cJ, \cK)$ is purely contractive and CE, and if $X$ is the unique extremal Gleason solution for $\sH(B)$, then $B$ coincides weakly with $\Theta_X$.
    
    To this end, let $X = V + C$ be the isometric--pure decomposition of $X$, and let $(\Ga, \cJ \oplus \cJ', \cK)$ be the model triple for $V$ as in Corollary \ref{cor:model.triple.for.Gleason.soln}. Since $\Ga(0) := K_0^B D_{B(0)^*}^{-1}$ and $\Ga(\infty) := \mbf B D_{B(0)}^{-1} \oplus \Ga'(\infty)$, where $\mbf B$ is the unique extremal Gleason solution for $B$, are onto isometries, it follows that $\Theta_X$ coincides weakly with
    \begin{equation*}
       \widetilde{\Theta}_X(Z) := -[\Ga(0)^* X \Ga(\infty) \otimes I_n] + [\Ga(0) D_{X^*} \otimes I_n] [I - ZX^*]^{-1} [I_\sH(B) \otimes Z] [D_X \Ga(\infty) \otimes I_n],
    \end{equation*}
    and, by Proposition \ref{prop:char.map.extremal.Gleason.soln.general}, we know that $$\delta^\Ga_X = - \Ga(0)^* X \Ga(\infty) = B(0) \oplus \mbf 0_{\cJ',\cJ}.$$ As $X$ is extremal, we use the property $D_{X^*} = \sqrt{K^B_0 {K^B_0}^*} = K^B_0 D_{B(0)}^{-1} {K^B_0}^*$ to compute
    \begin{align}\label{eqn:NFP.char.coincide.weakly.with.B_T.proof.1}
        \widetilde\Theta_X(Z) &= B(0_n) \oplus [\mbf 0_{\cJ', \cJ} \otimes I_n] \ + \ [\Ga(0) D_{X^*} \otimes I_n] [I - ZX^*]^{-1} [I_{\sH(B)} \otimes Z] [D_X \Ga(\infty) \otimes I_n] \nonumber\\
        &= B(0_n) \oplus [\mbf 0_{\cJ', \cJ} \otimes I_n] \ + \ [{K^B_0}^* \otimes I_n] [I - ZX^*]^{-1} [I_{\sH(B)} \otimes Z] [D_X \Ga(\infty) \otimes I_n].
    \end{align}
    As $\Ga(\infty)$ is an isometry onto $\nbker V$ and $X^* X = V^* V + C^* C$, we note that
    \begin{align*}
        D_X \Ga(\infty) &= \sqrt{I_{\sH(B)} \otimes I_d - X^* X} \ \Ga(\infty) \\
        &= \sqrt{P_{\nbker V} - P_{\nbker V} C^* C P_{\nbker V}} \ \Ga(\infty).
    \end{align*}
    It was observed in \eqref{eqn:partial.iso.part,of.Gleason.soln} that
    \begin{equation*}
        V^* = X^* \Big[ I_{\sH(B)} - K^B_0 D_{B(0)^*}^{-2} {K^B_0}^* \Big],
    \end{equation*}
    from which it follows using \eqref{eqn:kernel.at.0.and.X*} that
    \begin{equation*}
        C^* = X^* K^B_0 D_{B(0)^*}^{-2} {K^B_0}^* = -\mbf B B(0)^* D_{B(0)^*}^{-2} {K^B_0}^* = -\mbf B D_{B(0)}^{-1} B(0)^* D_{B(0)^*}^{-1} {K^B_0}^*.
    \end{equation*}
    From here, it is straightforward to obtain, for each $k \in \bN$,
    \begin{equation*}
        (C^* C)^k = \mbf B D_{B(0)}^{-1} \big(B(0)^* B(0)\big)^k D_{B(0)}^{-1} \mbf B^*.
    \end{equation*}
    Consequently, using the functional calculus and the fact that $\mbf B$ is extremal, we get
    \begin{align*}
        \sqrt{I_{\sH(B) \otimes \bC^d} - X^* X} \ \Ga(\infty) &= P_{\nbker V} \mbf B D_{B(0)}^{-1} \sqrt{I - B(0)^* B(0)} D_{B(0)}^{-1} \mbf B^* P_{\nbker V} \ \big[ \mbf B D_{B(0)}^{-1} \oplus \Ga'(\infty) \big] \\
        &= \mbf B D_{B(0)}^{-1}\mbf B^* \mbf B D_{B(0)}^{-1} \oplus \mbf B D_{B(0)}^{-1} \mbf B^* \Ga'(\infty) \\
        &= \mbf B \oplus \mbf B D_{B(0)}^{-1} \mbf B^* \Ga'(\infty).
    \end{align*}
    From the definition of $\Ga'(\infty)$ (see \eqref{eqn:gamma.prime.infty}), we have $\nbran \Ga'(\infty) = \nbker \mbf B^*$ and so the last equation above becomes
    \begin{equation*}
        \sqrt{I_{\sH(B)} \otimes I_d - X^* X} \Ga(\infty) = \mbf B \oplus \mbf 0_{\cJ', \sH(B) \otimes \bC^d}.
    \end{equation*}
    Plugging this into \eqref{eqn:NFP.char.coincide.weakly.with.B_T.proof.1} and using \eqref{eqn:char.map.gleason.sol.proof.3} gives us
    \begin{align*}
        \widetilde\Theta_X(Z) = B(0_n) \oplus [\mbf 0_{\cJ', \cJ} \otimes I_n] \ + \ \big(B(Z) - B(0_n)\big) \oplus [\mbf 0_{\cJ', \cJ} \otimes I_n] = B(Z) \oplus [\mbf 0_{\cJ', \cJ} \otimes I_n].
    \end{align*}
    As discussed in Remark \ref{rem:weak.coincidence.vs.coincidence}, we then conclude at once that $\widetilde\Theta_X$ coincides weakly with $B$ and, consequently, so does $\Theta_X$. This completes the proof.
\end{proof}

\section{NC de Branges--Rovnyak model via realizations}\label{appx:Model.via.TFR}

There is an alternate and simple proof of the fact that any CNC row contraction can be modeled by extremal Gleason solutions of an NC de Branges--Rovnyak space using the transfer function and realization theory of \cite{BBF-ncSchur}. We believe that our approach gives additional information and is of independent interest, and yet we would be remiss if we did not explain this alternative proof. 

Given complex, separable Hilbert spaces, $\cH$, $\cJ$ and $\cK$ consider a block operator matrix,
$$ M := \bpm \hat{A} & \hat{B} \\ \hat{C} & \hat{D} \epm : \bpm \cH \\ \cJ \epm \rightarrow \bpm \cH \otimes \C ^d \\ \cK \epm. $$ Such a block operator matrix is called a \emph{colligation}. If $M$ is contractive, then the quadruple $(\hat{A},\hat{B},\hat{C},\hat{D})$ defines an \emph{NC transfer function}, $B \in \sS _d (\cJ, \cK)$ via the \emph{realization formula}, 
$$ B(X) := I_n \otimes \hat{D} + I_n \otimes \hat{C} L_{\hat{A}} (X ) ^{-1} X \otimes \hat{B}; \quad \quad X \in \bdn, $$ where $L_{\hat{A}} (X) := I_n \otimes I_\cH - \sum _{j=1} ^d X_j \otimes \hat{A}_j$ is the monic, affine linear pencil of $\hat{A}$ and $X\otimes \hat{B}:= \sum _{j=1} ^d X_j \otimes \hat{B}_j$ \cite[Theorem 3.1]{BBF-ncSchur}. The quadruple, $(\hat{A},\hat{B},\hat{C},\hat{D})$ is a called a Fornasini--Marchesini (FM) \emph{realization} of the NC function, $B$. Conversely, any $B\in \sS _d (\cJ , \cK)$ has the canonical \emph{de Branges--Rovnyak realization} and colligation given by
$$ \bpm X^* & \mbf{b} \\ K_0 ^{B *} & b(0) \epm : \bpm \scr{H} (B) \\ \cJ \epm \rightarrow \bpm \scr{H} (B) \\ \cK \epm, $$ where $X ^* _j := R^* \otimes I _\cK | _{\scr{H} (B)}$ is the unique Gleason solution for $\scr{H} (B)$ and $\mbf{B} := R^* \otimes I_\cJ B \in \scr{H} (B) \otimes \C ^d$ is the unique Gleason solution for $B$. This realization is always coisometric and \emph{observable} in the sense that 
$$ \bigvee _{\om \in \F} \nbran \hat{A} ^{*\om} \hat{C} ^* = \scr{H} (B), $$ where $\hat{A}  = X^*$ and $\hat{C} ^* = K_0 ^B$. By \cite[Corollary 3.9 and Theorem 4.3]{BBF-ncSchur}, any two coisometric and observable realizations of $B \in \sS _d (\cJ, \cK)$, 
$$ \bpm \hat{A} & \hat{B} \\ \hat{C} & \hat{D} \epm : \bpm \cH \\ \cJ \epm \rightarrow \bpm \cH \otimes \C ^d \\ \cK \epm, \quad \mbox{and} \quad \bpm \wt{A} & \wt{B} \\ \wt{C} & \wt{D} \epm : \bpm \wt{\cH} \\ \cJ \epm \rightarrow \bpm \wt{\cH} \otimes \C ^d \\ \cK \epm, $$ are unitarily equivalent in the sense that there exists a unitary $U: \cH \twoheadrightarrow \wt{\cH}$ obeying
$$ \bpm U \otimes I_d & 0 \\ 0 & I _\cK \epm \bpm \hat{A} & \hat{B} \\ \hat{C} & \hat{D} \epm = \bpm \wt{A} & \wt{B} \\ \wt{C} & \wt{D} \epm \bpm U & 0 \\  0 & I _\cJ \epm. $$

If $T : \cH \otimes \C ^d \rightarrow \cH$ is any row contraction, its \emph{Julia matrix}, 
$$ U_T := \bpm T^* & D_T \\ D_{T^*} & -T \epm : \bpm \cH \\ \nbran D_T \epm \rightarrow \bpm \cH \otimes \C ^d \\ \nbran D_{T^*} \epm, $$ is a surjective isometry. In particular, $U_T$ is coisometric and setting $A:= T^*$, $C:= D_{T^*}$, 
$$ \bigvee A^{*\om} C^* = \bigvee _{\om \in \F} T^\om D_{T^*}, $$ so that this FM colligation is observable provided $T$ is CNC by \eqref{CNCformula}. 

By \cite[Theorem 3.1]{BBF-ncSchur}, the NC transfer function, $B_T$, of $U_T$ belongs to the NC Schur class, $B_T \in \sS _d (\scr{D} _{T}, \scr{D} _{T^*})$, and by \cite[Theorem 4.3]{BBF-ncSchur}, the canonical de Branges--Rovnyak colligation for $B_T$ is unitarily equivalent to $U_T$. Hence, $T^*$ is unitarily equivalent to $X^*$, where $X^* := R^* \otimes I _{\scr{D} _{T^*}} | _{\scr{H} (B_T)}$ is the unique contractive and extremal Gleason solution for $\scr{H} (B_T)$ and so $T$ is unitarily equivalent to $X$. Moreover, this operator--valued NC Schur class function, $B_T$, coincides with NC characteristic function of $T$, by our previous results. 

\printbibliography

\end{document}